\documentclass[10pt]{amsart}
\usepackage{amsmath}
\usepackage{amssymb}
\usepackage{amsfonts}
\usepackage{amsthm}
\usepackage{enumerate}
\usepackage[all]{xy}
\usepackage[latin1]{inputenc}
\usepackage{graphicx}
\usepackage{latexsym}
\usepackage{accents}
\usepackage{mathrsfs}
\usepackage{indentfirst}
\usepackage{fancyhdr}
\usepackage{color}
\usepackage{graphicx}
\usepackage{epstopdf}
\input xy

\usepackage{tikz}

\makeatletter
    \newtheoremstyle{definition}% name of the style to be used
        {5pt}% measure of space to leave above the theorem. E.g.: 3pt
        {3pt}% measure of space to leave below the theorem. E.g.: 3pt
        {}% name of font to use in the body of the theorem
        {0pt}% measure of space to indent
        {\scshape}% name of head font
        {.}% punctuation between head and body
        {5pt}% space after theorem head; " " = normal interword space
        {\thmname{#1} \thmnumber{#2} \thmnote{[#3]}} % Manually specify head

\newtheoremstyle{theorems}% name of the style to be used
        {5pt}% measure of space to leave above the theorem. E.g.: 3pt
        {3pt}% measure of space to leave below the theorem. E.g.: 3pt
        {\itshape}% name of font to use in the body of the theorem
        {0pt}% measure of space to indent
        {\scshape}% name of head font
        {.}% punctuation between head and body
        {5pt}% space after theorem head; " " = normal interword space
        {\thmname{#1} \thmnumber{#2}\thmnote{[#3]}} % Manually specify head

\swapnumbers

\theoremstyle{theorems}

\newtheorem{Theorem}{Theorem}[section]
\newtheorem{Lemma}[Theorem]{Lemma}
\newtheorem{Definition-Lemma}[Theorem]{Definition-Lemma}

\newtheorem{Definition}[Theorem]{Definition}

\newtheorem{Proposition}[Theorem]{Proposition}
\newtheorem{Example}[Theorem]{Example}
\newtheorem{Remark}[Theorem]{Remark}
\newtheorem{Conjecture}[Theorem]{Conjecture}

\newcommand{\huang}[1]{\textcolor{red}{Huang: #1}}

\begin{document}

\title [On quantum cluster algebras from unpunctured triangulated surfaces]
 {On quantum cluster algebras from unpunctured triangulated surfaces:\\ arbitrary coefficients and quantization}

\author{Min Huang $\;\;\;\;\;\;$}
\address{Min Huang
\newline D\'{e}partement de math\'{e}matiques, Universit\'{e} de Sherbrooke, Sherbrooke, Qu\'{e}bec,
J1K 2R1, Canada}
\email{minhuang1989@hotmail.com}

\date{version of \today}



\thanks{}

\maketitle

\vspace{-25pt}

\begin{abstract}

We study quantum cluster algebras from unpunctured surfaces with arbitrary coefficients and quantization. We first give a new proof of the Laurent expansion formulas for commutative cluster algebras from unpunctured surfaces, we then give the quantum Laurent expansion formulas for the quantum cluster algebras. Particularly, this gives a combinatorial proof of the positivity for such class of quantum cluster algebras.

\end{abstract}

\smallskip

\tableofcontents

\section{Introduction}

Cluster algebras are commutative algebras that were introduced by Fomin and Zelevinsky around the year 2000. The quantum cluster algebras were later introduced in \cite{BZ}. The theory of cluster algebras is related to numerous other fields including Lie theory, representation theory of algebras, the periodicity issue, Teichm$\ddot{u}$ller theory and mathematical physics.

\medskip

A cluster algebra is a subalgebra of rational function field with a distinguish set of generators, called cluster variables. Different cluster variables are related by a iterated procedure called mutations. By construction, cluster variables are rational functions. In \cite{fz1}, Fomin and Zelevinsky proved that they are Laurent polynomials of initial cluster variables. It was proved that the coefficients of these Laurent polynomials are non-negative, known as positivity of the coefficients, see \cite{LS,GHKK}.

\medskip

The original motivation of Fomin and Zelevinsky is to develop a combinatorial approach to the canonical bases in quantum groups (see \cite{L,K}) and the total positivity in algebraic groups. They conjectured that the cluster structure should serve as an algebraic framework for the study of the ``dual canonical bases" in various coordinate rings and their $q$-deformations. Particularly, they conjectured all cluster monomials belong to the dual canonical bases, this was proved recently by \cite{KKKO,Q}. Generally it can be very hard to write the dual canonical bases explicitly. From this point of view,  find an explicit expansion formula for the (quantum) cluster monomials represents a noteworthy step in research in cluster
theory.

\medskip

Much research approached to this way. For acyclic skew-symmetric cluster algebras, \cite{CC,CK,CK1} gave a Laurent expansion formula, know as \emph{Caldero-Chapoton map}, where the coefficients are Euler-Poincar$\acute{e}$ characteristics of appropriate Grassmannians of quiver representations. \cite{FK,P,PP,PP1} generalized this approach to the cluster algebras which admit a categorification by a 2-Calabi-Yau category; by \cite{A}, \cite{L-F}. In these papers, they named the Laurent expansion formula as \emph{cluster character}. \cite{DE,HL,CHL} generalized the cluster character to the acyclic skew-symmetrizable and acyclic totally sign-skew-symmetrizable cases. However, since all of the above formulas are in terms of Euler-Poincar$\acute{e}$ characteristics (which can be negative), they do not immediately imply the positivity of the coefficients. Moreover, Euler-Poincar$\acute{e}$ characteristics can be very hard to compute generally.

\medskip

In the other direction, much work have been conducted on the cluster algebras from surfaces, introduced by \cite{FST}. For cluster algebras from unpunctured surfaces, \cite{ST} gave an explicit Laurent expansion formula in terms of $T$-paths for the boundary coefficients cases; \cite{S} generalized the result in \cite{ST} to the principal coefficients cases; \cite{MS} gave the Laurent expansion formula in terms of perfect matchings, which turned out to be a very useful tool. \cite{MSW} generalized \cite{MS} to the cluster algebras from surfaces with arbitrary coefficients. More precisely, in these cases, the cluster variables and the arcs in the surface $(\mathcal O,M)$ are one to one correspondence, and seeds and the triangulations are one to one correspondence. Fix a triangulation $T$. For each arc $\gamma$, we can associate a graph $G_{T,\gamma}$, see \cite{MS,CS}. A perfect matching $P$ is a collection of edges in $G_{T,\gamma}$ such that each vertex is incident to exactly one edge in $P$. Then each perfect matching gives a term of the Laurent expansion formula of $x_{\gamma}$ (the cluster variable associate with $\gamma$) respect to $x^{T}$ (the cluster associate with $T$). In terms of perfect matchings of angles, \cite{Y} gave a cluster expansion formula for cluster algebras from surfaces with principal coefficients.

\medskip

Some works on the quantum case. \cite{R} generalized the Caldero-Chapoton map to the quantum case. \cite{R1} gave the explicit quantum Laurent expansion formula for the quantum cluster algebras of type $A$. Recently, \cite{CL} gave an explicit quantum Laurent expansion formula for the quantum cluster algebras of type $A$ and the Kronecker type.

\medskip

The aim of this paper is to give a quantum Laurent expansion formula for quantum cluster algebras from unpunctured surfaces with arbitrary coefficients and quantization. For these quantum cluster algebras, the quantum cluster variables and the arcs in the surface $(\mathcal O,M)$ are one to one correspondence, and quantum seeds and the triangulations are one to one correspondence. The idea is to associate each perfect matching a $q$-power, which aim to be the $q$-coefficients, where $q$ is the quantum parameter. Fix a triangulation $T$ and an arc $\gamma$, a valuation map on the set of perfect matchings of $G_{T,\gamma}$ is constructed in Theorem \ref{mainthm}. It turned out that the valuation map successfully acts the roles of $q$-coefficients, see Theorem \ref{expansion}, where the quantum Laurent expansion formula is given. As corollary, we show the coefficients appear in the quantum Laurent polynomials are non-negative.

\medskip

The paper is organized as follows. We recall some backgrounds on cluster algebras and cluster algebras from unpunctured surfaces in Section \ref{pre}, and quantum cluster algebras and  quantum cluster algebras from unpunctured surfaces in Section \ref{pre2}. The main results Theorems \ref{partition bi}, \ref{expansion-comm} for cluster algebras and Theorems \ref{mainthm}, \ref{expansion} for quantum cluster algebras are stated in Section \ref{cle} and Section \ref{qle}, respectively. To prove Theorems \ref{partition bi}, \ref{mainthm}, some preparations are made in Section \ref{compare}. We prove Theorem \ref{partition bi} in Section \ref{main1} and Theorem \ref{mainthm} in Section \ref{main2}.

\medskip

Throughout this paper, denote by $E(G)$ the edges set of a graph $G$ and denote by $|S|$ the cardinality of a set $S$.

\section{Preliminaries on cluster algebras}\label{pre}

\subsection{Commutative cluster algebras}

In this subsection, we recall the definitions of cluster algebras and geometric type cluster algebras in \cite{fz1}.

\medskip

A triple $(\mathbb P,\oplus,\cdot)$ is called a \emph{semifield} if $(\mathbb P,\cdot)$ is an abelian multiplicative group and $(\mathbb P,\oplus)$ is a commutative semigroup such that $``\oplus"$ is distributive with respect to $``\cdot"$. A \emph{Tropical semifield} ${\rm Trop}(u_1,\cdots,u_l)$ is a semifield freely generated by $u_1,\cdots,u_l$ as abelian groups with $\oplus$ defined by $\prod_ju_j^{a_j}\oplus \prod_ju_j^{b_j}=\prod_ju_j^{min(a_j,b_j)}$. Let $(\mathbb P,\oplus,\cdot)$ be a semifield. The group ring $\mathbb {ZP}$ will be used as \emph{ground ring}. Give an integer $n$, let $\mathcal F$ be the rational functions field in $n$ independent variables, with coefficients in $\mathbb {QP}$.

\medskip

A seed $t$ in $\mathcal F$ consists a trip $(x(t),y(t),B(t))$, where

\begin{enumerate}[$(1)$]

  \item $x(t)=\{x_1(t),\cdots,x_n(t)\}$ such that $\mathcal F$ is freely generated by $x(t)$ over $\mathbb {QP}$.

  \item $y(t)=\{y_1(t),\cdots,y_n(t)\}\subseteq \mathbb P$.

  \item $B(t)=(b_{ij})$ is an $n\times n$ skew-symmetrizable integer matrix.

\end{enumerate}

\medskip

Given a seed $t$ in $\mathcal F$, for any $k\in [1,n]$, we define the \emph{mutation} of $t$ at the $k$-th direction to be the new seed $t'=\mu_k(t)=(x(t'),y(t'),B(t'))$, where

\begin{enumerate}[$(1)$]

  \item \[\begin{array}{ccl} x_i(t') &=&
         \left\{\begin{array}{ll}
              x_i(t), &\mbox{if $i\neq k$},  \\
             \frac{y_k(t)\prod x_i(t)^{[b_{ik}]_{+}}+\prod x_i(t)^{[-b_{ik}]_{+}}}{(y_k(t)\oplus 1)x_k(t)}, &\mbox{otherwise}.
         \end{array}\right.
        \end{array}\]

  \item \[\begin{array}{ccl} y_i(t') &=&
         \left\{\begin{array}{ll}
              y^{-1}_k(t), &\mbox{if $i=k$},  \\
              y_j(t)y_k(t)^{[b_{kj}]_{+}}(1\oplus y_k(t))^{-b_{kj}}, &\mbox{otherwise}.
         \end{array}\right.
        \end{array}\]

  \item $B(t')=(b'_{ij})$ is determined by $B(t)=(b_{ij})$:

  \[\begin{array}{ccl} b'_{ij} &=&
         \left\{\begin{array}{ll}
              -b_{ij}, &\mbox{if $i=k$ or $j=k$},  \\
              b_{ij}+[b_{ik}]_{+}[b_{kj}]_{+}-[-b_{ik}]_{+}[-b_{kj}]_{+}, &\mbox{otherwise}.
         \end{array}\right.
        \end{array}\]

\end{enumerate}
where $[a]_{+}=max(a,0)$.

\medskip

A \emph{cluster algebra} $\mathcal A$ (of rank $n$) over $\mathbb P$ is defined as following:
\begin{enumerate}[$(1)$]
  \item Choose an initial seed $t_0=(x(t_0),y(t_0),B(t_0))$.
  \item All the seeds $t$ are obtained from $t_0$ by iterated mutations at directions $k\in[1,n]$.
  \item $\mathcal A=\mathbb{ZP}[x_i(t)]_{t,i\in [1,n]}$.
  \item $x(t)$ is called a \emph{cluster} for any $t$.
  \item $x_i(t)$ is called a \emph{cluster variable} for any $i\in [1,n]$ and $t$.
  \item A monomial in $x(t)$ is called a \emph{cluster monomial} for any $t$.
  \item $y(t)$ is called a \emph{coefficient tuple} for any $t$.
  \item $B(t)$ is called an \emph{exchange matrix} for any $t$.
\end{enumerate}
In particular, when $\mathbb P={\rm Trop}(u_1,\cdots,u_l)$, we say $\mathcal A$ is of \emph{geometric type}.

\medskip

When $\mathbb P={\rm Trop}(u_1,\cdots,u_l)$, denote $m=n+l$. For a seed $t$ in $\mathcal F$, $y_j(t)=\prod u_i^{a_{ij}}$ for some integers $a_{ij}$. Thus we can write $t$ as $\widetilde x(t),\widetilde B(t)$, where

\begin{enumerate}[$(1)$]

  \item $\widetilde x(t)=\{x_1(t),\cdots,x_n(t),x_{n+1}(t)=u_1,\cdots,x_m(t)=u_l\}$.

  \item $\widetilde B(t)=(b_{ij})$ is an $m\times n$ with $b_{ij}=a_{i-n,j}$ for $i\in [n+1,m]$.

\end{enumerate}

In this case, the mutation of $t$ at direction $k$ is $t'=\mu_k(t)=(\widetilde x(t'),\widetilde B(t'))$, where

\begin{enumerate}[$(1)$]

  \item \[\begin{array}{ccl} x_i(t') &=&
         \left\{\begin{array}{ll}
              x_i(t), &\mbox{if $i\neq k$},  \\
             \frac{\prod x_i(t)^{[b_{ik}]_{+}}+\prod x_i(t)^{[-b_{ik}]_{+}}}{x_k(t)}, &\mbox{otherwise}.
         \end{array}\right.
        \end{array}\]

  \item $\widetilde B(t')=(b'_{ij})$ is determined by $\widetilde B(t)=(b_{ij})$:

  \[\begin{array}{ccl} b'_{ij} &=&
         \left\{\begin{array}{ll}
              -b_{ij}, &\mbox{if $i=k$ or $j=k$},  \\
              b_{ij}+[b_{ik}]_{+}[b_{kj}]_{+}-[-b_{ik}]_{+}[-b_{kj}]_{+}, &\mbox{otherwise}.
         \end{array}\right.
        \end{array}\]

\end{enumerate}

\medskip

\subsection{Cluster algebras from unpunctured surfaces} In this subsection, we recall the cluster algebras associated with unpunctured surfaces by \cite{FST}. Let $\mathcal O$ be a connected oriented Riemann surface with boundary. Fix a non-empty set $M$ of marked points in the closure of $\mathcal O$ with at least one marked point on each boundary component. We call the pair $(\mathcal O, M)$ a \emph{bordered surface with marked points}. Marked points in the interior of $\mathcal O$ are called \emph{punctures}.

\medskip

In this paper, we consider the case that $(\mathcal O,M)$ without punctures, and we simply refer $(\mathcal O, M)$, or $\mathcal O$ where there is no chance of confusion, as an \emph{unpunctured surface}.

\medskip

%A \emph{boundary arc} in $(\mathcal O, M)$ is a curve which lies in the boundary of $\mathcal O$ and connects two marked points without passing through a third.
%
%\medskip

An \emph{arc} $\gamma$ in $(\mathcal O, M)$ is a curve (up to isotopy) in $\mathcal O$ such that: the endpoints are in $M$; $\gamma$ does not cross itself, except its endpoints may coincide; except for the endpoints, $\gamma$ is disjoint from $M$ and from the boundary of $\mathcal O$; and $\gamma$ does not cut
out a monogon or a bigon.

\medskip

For two arcs $\gamma, \gamma'$ in $(\mathcal O, M)$, the \emph{crossing number} $N(\gamma,\gamma')$ of $\gamma$ and $\gamma'$ is the minimum of the numbers of crossings of arcs $\alpha$ and $\alpha'$, where $\alpha$ is isotopic to $\gamma$ and $\alpha'$ is isotopic to $\gamma'$. We call $\gamma$ and $\gamma'$ are \emph{compatible} if the crossing number of $\gamma$ and $\gamma'$ is 0. A \emph{triangulation} is a maximal collection of compatible arcs. Given a triangulation $T$ and a non-boundary arc $\tau$ of $T$, there exists a unique arc $\tau'$ such that $(T\setminus \{\tau\})\cup \{\tau'\}$ is a new triangulation. Denote $(T\setminus \{\tau\})\cup \{\tau'\}$ by $\mu_{\tau}(T)$. For an arc $\gamma$ and a triangulation $T$, the \emph{crossing number} $N(\gamma,T)$ of $\gamma$ and $T$ is defined as $\sum_{\tau\in T}N(\gamma,\tau)$. We call a triangulation $T$ an \emph{indexed triangulation} if the order of the arcs in $T$ is fixed.
 %For an indexed triangulation $T=\{\tau_1,\cdots,\tau_n,\cdots,\tau_m\}$ of $\mathcal O$, we always assume $\tau_i,i=n+1,\cdots,m$ are boundary arcs.

\medskip
\medskip

For two non-boundary arcs $\tau$, $\tau'$ in an indexed triangulation $T=\{\tau_1,\cdots,\tau_n,\cdots,\tau_l\}$ and a triangle $\Delta$ of $T$, we may assume $\tau_1,\cdots,\tau_n$ are the non-boundary arcs, define
\[\begin{array}{ccl} b^{T,\Delta}_{\tau\tau'}=
         \left\{\begin{array}{ll}
              1, &\mbox{if $\tau$ and $\tau'$ are sides of $\Delta$ with $\tau'$ following $\tau$ in the clockwise order},  \\
              -1, &\mbox{if $\tau$ and $\tau'$ are sides of $\Delta$ with $\tau$ following $\tau'$ in the clockwise order}, \\
              0, &\mbox{otherwise.}
         \end{array}\right.
 \end{array}\]
and $b^T_{\tau\tau'}=\sum_{\Delta}b^{T,\Delta}_{\tau\tau'}$.
We say the $n\times n$ matrix $B^T=(b^T_{ij})$ with $b^T_{ij}=b^T_{\tau_i\tau_j}$, the \emph{signed adjacency matrix} of $T$.

\medskip

We say that a cluster algebra $\mathcal A$ is \emph{coming from $(\mathcal O, M)$} if there exists a triangulation $T$ such that $B^{T}$ is an exchange matrix of $\mathcal A$.

\medskip

Throughout this paper, let $\mathcal O$ be an unpunctured surface and $T=\{\tau_1,\cdots,\tau_n,\cdots,\tau_l\}$ be an indexed triangulation, and let $\gamma$ be an oriented arc in $\mathcal O$. When an arc $\tau\in T$ is fixed, denote $\mu_{\tau}(T)$ by $T'$. We always assume $\tau_1,\cdots,\tau_n$ are the non-boundary arcs in $T$.

\medskip

\subsection{Snake graphs and Perfect matchings}

In this subsection, we recall the construction of the snake graph $G_{T,\gamma}$ and the perfect matching of $G_{T,\gamma}$. For more details, see \cite[Section 4]{MSW}, \cite{CS}.

\medskip

Let $p_0$ be the starting point of $\gamma$, and let $p_{d+1}$ be its endpoint. Assume $\gamma$ crosses $T$ at $p_1,\cdots,p_d$ in order.
%We say that the crossing point $p_s$ is \emph{adjacent} to $p_t$ if $|s-t|=1$.
Let $\tau_{i_j}$ be the arc in $T$ containing $p_j$. Let $\Delta_{j-1}$ and $\Delta_{j}$ be the two ideal triangles in $T$ on either side of $\tau_{i_j}$.

\medskip

For each $p_j$, we associate a \emph{tile} $G(p_j)$ as follows. Define $\Delta_1^j$ and $\Delta_2^j$ to be two triangles with edges labeled as in $\Delta_{j-1}$ and $\Delta_{j}$, further, the orientations of $\Delta_1^j$ and $\Delta_2^j$ both agree with those of $\Delta_{j-1}$ and $\Delta_{j}$ if $j$ is odd; the orientations of $\Delta_1^j$ and $\Delta_2^j$ both disagree with those of $\Delta_{j-1}$ and $\Delta_{j}$ otherwise. We glue $\Delta_1^j$ and $\Delta_2^j$ at the edge labeled $\tau_{i_j}$, so that the orientations of $\Delta_1^j$ and $\Delta_2^j$ both either agree or disagree with those of $\Delta_{j-1}$ and $\Delta_{j}$. We say the edge labeled $\tau_{i_j}$ the \emph{diagonal} of $G(p_j)$.

\medskip

The two arcs $\tau_{i_j}$ and $\tau_{i_{j+1}}$ form two edges of the triangle $\Delta_j$. Denote the third edge of $\Delta_j$ by $\tau_{[\gamma_j]}$. After glue the tiles $G(p_j)$ and $G(p_{j+1})$ at the edge labeled $\tau_{[\gamma_j]}$ for $1\leq j<d-1$ step by step, we obtain a graph, denote as $\overline{G_{T,\gamma}}$. Let $G_{T,\gamma}$ be the graph obtained from $\overline{G_{T,\gamma}}$ by removing the diagonal of each tile.

\medskip

In particular, when $\gamma\in T$, let $G_{T,\gamma}$ be the graph with one only edge labeled $\gamma$.

\medskip

\begin{Definition}(\cite[Definition 4.6]{MSW})\label{perctvec}
A \emph{perfect matching} of a graph $G$ is a subset $P$ of the edges of $G$ such that each vertex of $G$ is incident to exactly one edge of $P$. Denote the set of all perfect matchings of $G$ by $\mathcal P(G)$.

\end{Definition}

\medskip

According to the definition, we have the following observation.

\medskip

\begin{Lemma}\label{inone}

Let $G_i$ and $G_{i+1}$ be two consecutive tiles of $G_{T,\gamma}$ sharing same edge $a$. If $b$ (respectively $c$) is an edge of $G_i$ (respectively $G_{i+1}$) which is incident to $a$, then $b$ and $c$ can not in a perfect matching of $G$ at the same time.

\end{Lemma}

\begin{proof}

Without loss of generality, we may assume that $G_{i+1}$ is on the right of $G_i$. In case $b$ and $c$ are incident, then they can not in a perfect matching according to the definition. In case $b$ and $c$ are not incident, we have the following possibilities, see the Figures below.

\centerline{\begin{tikzpicture}
\draw[-] (1,0) -- (2,0);
\draw[-] (1,-1) -- (2,-1);
\draw[-] (2,0) -- (3,0);
\draw[-] (2,-1) -- (3,-1);
\draw[-] (3,0) -- (3,-1);
\node[right] at (1.8,-0.5) {$a$};
\node[right] at (1.2,-0.5) {$G_{i}$};
\node[above] at (1.5,-0.05) {$b$};
\node[right] at (2.1,-0.5) {$G_{i+1}$};;
\node[below] at (2.5,-0.95) {$c$};
\draw [-] (1, 0) -- (1,-1);
\draw [-] (2, 0) -- (2,-1);
\draw [fill] (1,0) circle [radius=.05];
\draw [fill] (2,0) circle [radius=.05];
\draw [fill] (1,-1) circle [radius=.05];
\draw [fill] (2,-1) circle [radius=.05];
\draw [fill] (3,0) circle [radius=.05];
\draw [fill] (3,-1) circle [radius=.05];
\draw[-] (5,0) -- (6,0);
\draw[-] (5,-1) -- (6,-1);
\draw[-] (6,0) -- (7,0);
\draw[-] (6,-1) -- (7,-1);
\draw[-] (7,0) -- (7,-1);
\node[right] at (5.8,-0.5) {$a$};
\node[right] at (5.2,-0.5) {$G_{i}$};
\node[above] at (6.5,-0.05) {$b$};
\node[right] at (6.1,-0.5) {$G_{i+1}$};;
\node[below] at (5.5,-0.95) {$c$};
\draw [-] (5, 0) -- (5,-1);
\draw [-] (6, 0) -- (6,-1);
\draw [fill] (5,0) circle [radius=.05];
\draw [fill] (6,0) circle [radius=.05];
\draw [fill] (5,-1) circle [radius=.05];
\draw [fill] (6,-1) circle [radius=.05];
\draw [fill] (7,0) circle [radius=.05];
\draw [fill] (7,-1) circle [radius=.05];
\end{tikzpicture}}

We shall only prove the first case since the second case can be proved dually. Clearly, our result holds when $i=1$. Suppose our result holds for $i<k$. When $i=k$, if $G_{i-1}$ is on the left of $G_{i}$, then $b,c\in P$ for a perfect matching $P$ implies $b,c'\in P$, see the figure below, this contradicts our induction hypothesis. If $G_{i-1}$ is below to $G_{i}$, then $b,c\in P$ implies $c'\in P$, see the figure below. Thus no edge in $P$ is incident to the vertex $o$ if $i-1=1$ or $G_{i-2}$ is on the left of $G_{i-1}$, and hence $i>2$ and $G_{i-2}$ is below to $G_{i-1}$. Therefore, $b,c\in P$ for a perfect matching $P$ implies $b',c'\in P$, see the figure below. This contradicts to our induction hypothesis. The proof of the lemma is complete.
\end{proof}

\centerline{\begin{tikzpicture}
\draw[-] (0,-1) -- (1,-1);
\draw[-] (0,-2) -- (1,-2);
\draw[-] (0,-1) -- (0,-2);
\draw[-] (1,-1) -- (2,-1);
\draw[-] (1,-2) -- (2,-2);
\draw[-] (2,-1) -- (3,-1);
\draw[-] (2,-2) -- (3,-2);
\draw[-] (3,-1) -- (3,-2);
\node[right] at (1.8,-1.5) {$a$};
\node[right] at (1.2,-1.5) {$G_{i}$};
\node[above] at (1.5,-1.05) {$b$};
\node[right] at (0,-1.5) {$G_{i-1}$};
\node[right] at (2.1,-1.5) {$G_{i+1}$};
\node[below] at (2.5,-1.95) {$c$};
\node[below] at (0.5,-1.95) {$c'$};
\draw [-] (1, -1) -- (1,-2);
\draw [-] (2, -1) -- (2,-2);
\draw [fill] (0,-1) circle [radius=.05];
\draw [fill] (0,-2) circle [radius=.05];
\draw [fill] (1,-1) circle [radius=.05];
\draw [fill] (2,-1) circle [radius=.05];
\draw [fill] (1,-2) circle [radius=.05];
\draw [fill] (2,-2) circle [radius=.05];
\draw [fill] (3,-1) circle [radius=.05];
\draw [fill] (3,-2) circle [radius=.05];
\draw[-] (4,0) -- (5,0);
\draw[-] (4,-1) -- (5,-1);
\draw[-] (4,-1) -- (4,-2);
\draw[-] (5,-1) -- (5,-2);
\draw[-] (4,-2) -- (5,-2);
\draw[-] (4,-2) -- (4,-3);
\draw[-] (5,-2) -- (5,-3);
\draw[-] (4,-3) -- (5,-3);
\draw[-] (5,0) -- (6,0);
\draw[-] (5,-1) -- (6,-1);
\draw[-] (6,0) -- (6,-1);
\node[left] at (4.05,-1.5) {$c'$};
\node[right] at (4.8,-0.5) {$a$};
\node[right] at (4.95,-2.5) {$b'$};
\node[right] at (4.2,-0.5) {$G_{i}$};
\node[above] at (4.5,-0.05) {$b$};
\node[right] at (5.1,-0.5) {$G_{i+1}$};
\node[right] at (4,-1.5) {$G_{i-1}$};
\node[right] at (4,-2.5) {$G_{i-2}$};
\node[below] at (5.5,-0.95) {$c$};
\node[right] at (5,-2.1) {$o$};
\draw [-] (4, 0) -- (4,-1);
\draw [-] (5, 0) -- (5,-1);
\draw [fill] (4,-3) circle [radius=.05];
\draw [fill] (5,-3) circle [radius=.05];
\draw [fill] (4,-2) circle [radius=.05];
\draw [fill] (5,-2) circle [radius=.05];
\draw [fill] (4,0) circle [radius=.05];
\draw [fill] (5,0) circle [radius=.05];
\draw [fill] (4,-1) circle [radius=.05];
\draw [fill] (5,-1) circle [radius=.05];
\draw [fill] (6,0) circle [radius=.05];
\draw [fill] (6,-1) circle [radius=.05];
\end{tikzpicture}}

\begin{Definition} (\cite[Definition 4.7]{MSW})
Let $a_1$ and $a_2$ be the two edges of $\overline{G_{T,\gamma}}$ which lie in the counterclockwise direction from the diagonal of $G(p_1)$. Then the \emph{minimal matching} $P_{-}(G_{T,\gamma})$ is defined as the unique matching which contains only boundary edges and does not contain edges $a_1$ or $a_2$. The \emph{maximal matching} $P_{+}(G_{T,\gamma})$ is the other matching with only boundary edges.

\end{Definition}

\medskip

According to the definitions, we have the following observation.

\medskip

\begin{Lemma}\label{max-min}

Let $a$ be an edge of the tile $G(p_j)$. If $a$ is in the maximal/minimal perfect matching of $G_{T,\gamma}$, then $a$ lies in the counterclockwise/clockwise direction from the diagonal of $G(p_j)$ when $j$ is odd and lies in the clockwise/counterclockwise direction from the diagonal of $G(p_j)$ when $j$ is even.

\end{Lemma}

\begin{proof}

We shall only prove the case of maximal perfect matching. We prove the lemma by induction on $j$. When $j=1$, it follows by the definition of $P_{+}$. Assume it holds for $j<k$. We consider the case $j=k$. We shall prove the case $G(p_{k-1})$ is on the left of $G(p_k)$ since the case $G(p_{k-1})$ below to $G(p_k)$ can be proved similarly. When $G(p_{k-1})$ is on the left of $G(p_k)$, $a$ can not be the left edge of $G(p_k)$. In case $a$ is the upper edge or lower edge, by Lemma \ref{inone}, either the left edge of $G(p_{k-1})$ in $P_{+}$ or the upper edge of $G(p_{k-2})$ in $P_{+}$, as shown in the Figure below. Thus $a$ satisfies the required condition by induction hypothesis. In case $a$ is the right edge, then the upper edge of $G(p_{k-1})$ in $P_{+}$, as shown in Figure below. Thus $a$ satisfies the required condition by induction hypothesis. Our result follows.
\end{proof}

\centerline{\begin{tikzpicture}
\draw[-] (1,0) -- (2,0);
\draw[-] (1,-1) -- (2,-1);
\draw [line width=1.6] (2,0) -- (3,0);
\draw[-] (2,-1) -- (3,-1);
\draw[-] (3,0) -- (3,-1);
\node[right] at (1,-0.5) {$k-1$};
\node[right] at (2.3,-0.5) {$k$};
\node[above] at (2.5,-0.05) {$a$};
\draw [line width=1.6] (1,0) -- (1,-1);
\draw [-] (2, 0) -- (2,-1);
\draw [fill] (1,0) circle [radius=.05];
\draw [fill] (2,0) circle [radius=.05];
\draw [fill] (1,-1) circle [radius=.05];
\draw [fill] (2,-1) circle [radius=.05];
\draw [fill] (3,0) circle [radius=.05];
\draw [fill] (3,-1) circle [radius=.05];
\draw [line width=1.6] (6,0) -- (7,0);
\draw [line width=1.6] (4,0) -- (5,0);
\draw[-] (4,-1) -- (5,-1);
\draw[-] (5,0) -- (6,0);
\draw[-] (6,-1) -- (7,-1);
\draw[-] (7,0) -- (7,-1);
\draw[-] (5,-1) -- (6,-1);
\draw[-] (6,0) -- (6,-1);
\node[right] at (4,-0.5) {$k-2$};
\node[above] at (6.5,-0.05) {$a$};
\node[right] at (5,-0.5) {$k-1$};
\node[right] at (6.3,-0.5) {$k$};
\draw [-] (4, 0) -- (4,-1);
\draw [-] (5, 0) -- (5,-1);
\draw [fill] (4,0) circle [radius=.05];
\draw [fill] (5,0) circle [radius=.05];
\draw [fill] (4,-1) circle [radius=.05];
\draw [fill] (5,-1) circle [radius=.05];
\draw [fill] (6,0) circle [radius=.05];
\draw [fill] (6,-1) circle [radius=.05];
\draw [fill] (7,0) circle [radius=.05];
\draw [fill] (7,-1) circle [radius=.05];
\end{tikzpicture}}

\medskip

\centerline{\begin{tikzpicture}
\draw[-] (1,0) -- (1,-1);
\draw[-] (1,-1) -- (2,-1);
\draw [line width=1.6] (3,0) -- (3,-1);
\draw[-] (2,-1) -- (3,-1);
\draw[-] (2,0) -- (3,0);
\node[right] at (1,-0.5) {$k-1$};
\node[right] at (2.3,-0.5) {$k$};
\node[right] at (2.95,-0.5) {$a$};
\draw [line width=1.6] (1,0) -- (2,0);
\draw [-] (2, 0) -- (2,-1);
\draw [fill] (1,0) circle [radius=.05];
\draw [fill] (2,0) circle [radius=.05];
\draw [fill] (1,-1) circle [radius=.05];
\draw [fill] (2,-1) circle [radius=.05];
\draw [fill] (3,0) circle [radius=.05];
\draw [fill] (3,-1) circle [radius=.05];
\end{tikzpicture}}

\begin{Definition}

Let $P$ be a perfect matching of $G_{T,\gamma}$.

\begin{enumerate}[$(1)$]

  \item \cite{MSW1} We call $P$ can \emph{twist} on a tile $G(p)$ if $G(p)$ has two edges belong to $P$, in such case we define the \emph{twist} $\mu_pP$ of $P$ on $G(p)$ to be the perfect matching obtained from $P$ by replacing the edges in $G(p)$ by the remaining two edges.

  \item In case $P$ can do twist on $G(p)$ with diagonal labeled by $\tau\in T$, we call the pair of the edges of $G(p)$ lying in $P$ a \emph{$\tau$-mutable edges pair in $P$}, any other edge in $P$ is called \emph{non-$\tau$-mutable edge in $P$}.

\end{enumerate}

\end{Definition}

\medskip

The following fact is easy but important.

\medskip

\begin{Lemma}\label{transitive}

For any two perfect matchings $P,Q\in \mathcal P(G_{T,\gamma})$, $Q$ can be obtained from $P$ by a sequence of twists.

\end{Lemma}

\begin{proof}

We prove this lemma by the induction of the number of tiles $d$ of $G_{T,\gamma}$. If $d=0$, then $Q=P$. If $d=1$, then either $Q=P$ or $Q=\mu_{p_1}P$. Thus the lemma holds for $d=0$ and $d=1$. Assume that the lemma holds for $d<k$. When $d=k$, without loss of generality, we may assume that $G(p_2)$ is on the right of $G(p_1)$. Denote the edges of $G(P_1)$ by $b_1,b_2,b_3$ and $b_4$, see the figure below. Thus for any $R\in \mathcal P(G_{T,\gamma})$, $b_1\in R$ or $b_2,b_4\in R$. In case $b_1\in P\cap Q$, then $P\setminus \{b_1\}$ and $Q\setminus \{b_1\}$ are perfect matchings of graph $G(p_2)\cup \cdots \cup G(p_d)$. By the induction hypothesis, $Q\setminus \{b_1\}$ can be obtained from $P\setminus \{b_1\}$ by a sequence of twists, thus $Q$ can be obtained from $P$ by a sequence of twists. In case $b_2,b_4\in P\cap Q$, then $P\setminus \{b_2,b_4\}$ and $Q\setminus \{b_2,b_4\}$ are perfect matchings of graph $G(p_3)\cup \cdots \cup G(p_d)$. By the assumption of induction, $Q\setminus \{b_2,b_4\}$ can be obtained from $P\setminus \{b_2,b_4\}$ by a sequence of twists, thus $Q$ can be obtained from $P$ by a sequence of twists. In case $b_1\in P$ and $b_2,b_4\in Q$, then $b_1\in P\cap\mu_{p_1}Q$. Thus $\mu_{p_1}Q$ can be obtained from $P$ by a sequence of twists, and hence $Q$ can be obtained from $P$ by a sequence of twists. In case $b_1\in Q$ and $b_2,b_4\in P$, then $b_1\in \mu_{p_1}P\cap Q$. Thus $Q$ can be obtained from $\mu_{p_1}P$ by a sequence of twists, and hence $Q$ can be obtained from $P$ by a sequence of twists. The proof of the Lemma is complete.
\end{proof}

\centerline{\begin{tikzpicture}
\draw[-] (1,0) -- (2,0) -- (3,0);
\draw[-] (1,-1) -- (2,-1) -- (3,-1);
\node[above] at (1.5,-0.05) {$b_2$};
\node[left] at (1.05,-0.5) {$b_1$};
\node[below] at (1.5,-0.95) {$b_4$};
\node[right] at (1.7,-0.5) {$b_3$};
\node[right] at (2.2,-0.5) {$p_2$};
\node[left] at (1.8,-0.5) {$p_1$};
\draw [-] (1, 0) -- (1,-1);
\draw [-] (2, 0) -- (2,-1);
\draw [-] (3, 0) -- (3,-1);
\draw [fill] (1,0) circle [radius=.05];
\draw [fill] (2,0) circle [radius=.05];
\draw [fill] (1,-1) circle [radius=.05];
\draw [fill] (2,-1) circle [radius=.05];
\draw [fill] (3,0) circle [radius=.05];
\draw [fill] (3,-1) circle [radius=.05];
\end{tikzpicture}}

\section{Preliminaries on quantum cluster algebras}\label{pre2}

\subsection{Quantum cluster algebras}

In this subsection, we recall the definition of quantum cluster algebras by \cite{BZ}. We follow the convention in \cite{Q}. Fix two integers $n\leq m$. Let $\widetilde B$ be an $m\times n$ integer matrix. Let $\Lambda$ be an $m\times m$ skew-symmetric integer matrix. We call $(\widetilde B,\Lambda)$ \emph{compatible} if $(\widetilde B)^t \Lambda=(D\;\;0)$ for some diagonal matrix $D$ with positive entries, where $(\widetilde B)^t$ is the transpose of $\widetilde B$. Note that in this case, the upper $n\times n$ submatrix of $\widetilde B$ is skew-symmetrizable, and $\widetilde B$ is full rank.

\medskip

Let $q$ be the quantum parameter. A \emph{quantum seed} $t$ consists a compatible pair $(\widetilde B,\Lambda)$ and a collection of indeterminate $X_i(t),i\in [1,m]$, called \emph{quantum cluster variables}. Let $\{e_i\}$ be the standard basis of $\mathbb Z^m$ and $X(t)^{e_i}=X_i(t)$. Define the corresponding \emph{quantum torus} $\mathcal T(t)$ to be the algebra which is freely generated by $X(t)^{\vec{a}},\vec{a}\in \mathbb Z^m$ as $\mathbb Z[q^{\pm 1/2}]$-module, with multiplication on these elements defined by
$$X(t)^{\vec{a}}X(t)^{\vec{b}}=q^{\Lambda(t)(\vec{a},\vec{b})/2}X(t)^{\vec{a}+\vec{b}},$$
where $\Lambda(t)(,)$ means the bilinear form on $\mathbb Z^m$ such that
$$\lambda(t)(e_i,e_j)=\lambda(t)_{ij}.$$

For any $k\in [1,n]$, we define the \emph{mutation} of $t$ at the $k$-th direction to be the new seed $t'=\mu_k(t)=((X_i(t')_{i\in [1,m]}), \widetilde B(t'),\lambda(t'))$, where
\begin{enumerate}[$(1)$]

\item $X_i(t')=X_i(t)$ for $i\neq t$,

\item $X_k(t')=X(t)^{-e_k+\sum_i[b_{ik}]_{+}e_i}+X(t)^{-e_k+\sum_i[-b_{ik}]_{+}e_i}$.

\item $\widetilde{B}(t')=(b'_{ij})$ is determined by $\widetilde B(t)=(b_{ij})$:
\[\begin{array}{ccl} b'_{ij} &=&
         \left\{\begin{array}{ll}
              -b_{ij}, &\mbox{if $i=k$ or $j=k$},  \\
              b_{ij}+[b_{ik}]_{+}[b_{kj}]_{+}-[-b_{ik}]_{+}[-b_{kj}]_{+}, &\mbox{otherwise}.
         \end{array}\right.
 \end{array}\]

\item $\Lambda(t')$ is skew-symmetric and satisfies:
\[\begin{array}{ccl} \Lambda(t')_{ij} &=&
         \left\{\begin{array}{ll}
              \Lambda(t)_{ij}, &\mbox{if $i,j\neq k$},  \\
              \Lambda(t)(e_i,-e_k+\sum_{l}[b_{lk}]_{+}e_l), &\mbox{if $i\neq k=j$},
         \end{array}\right.
 \end{array}\]

\end{enumerate}
where $[a]_{+}=max(a,0)$.

\medskip

One can check easily that $(\widetilde B(t'), \Lambda(t'))$ is compatible since $\widetilde B^t(t)\Lambda(t)=\widetilde B^t(t')\Lambda(t')$.

\medskip

The quantum torus $\mathcal T(t')$ for the new seed $t'$ is defined similarly.

\medskip

We define a \emph{quantum cluster algebra} $\mathcal A_q$ as following:
\begin{enumerate}[$(1)$]

  \item Choose an initial seed $t_0=((X_1,\cdots,X_m), B,\Lambda)$.

  \item All the seeds $t$ are obtained from $t_0$ by iterated mutations at directions $k\in [1,n]$.

  \item $\mathcal A_q=\mathbb Z[q^{\pm1/2}]\langle X_i(t)\rangle_{t,i\in [1,m]}$.

  \item $X_{n+1},\cdots, X_m$ are called \emph{frozen variables} or \emph{coefficients}.

  \item A quantum cluster variable in $t$ is called a \emph{quantum cluster variable} of $\mathcal A_q$.

  \item $X(t)^{\vec{a}}$ for some $t$ and $\vec{a}\in \mathbb N^m$ is called a \emph{quantum cluster monomial}.

\end{enumerate}

\medskip

\begin{Theorem}(Quantum Laurent Phenomenon,\cite{BZ})
Let $\mathcal A_q$ be a quantum cluster algebra and $t$ be a seed. For any quantum cluster variable $X$, we have $X\in \mathcal T(t)$.

\end{Theorem}

\medskip

\begin{Conjecture}(Positivity Conjecture, \cite{BZ})
Let $\mathcal A_q$ be a quantum cluster algebra and $t$ be a seed. For any quantum cluster variable $X$ of $\mathcal A_q$,
$$X\in \mathbb N[q^{\pm1/2}]\langle X(t)^{\vec{a}}\mid \vec{a}\in \mathbb Z^m\rangle.$$

\end{Conjecture}

This conjecture was recently proved by \cite{D} in the skew-symmetric case.

\begin{Remark}

In this paper, to distinguish commutative cluster algebras and quantum cluster algebras, we use the notation $x$ to denote the commutative cluster variables and $X$ to denote the quantum cluster variables.

\end{Remark}

\subsection{Quantum cluster algebras from unpunctured surfaces}
We first fix some notations for the rest of this paper. Let $(\mathcal O, M)$ be an unpunctured surface and $T=\{\tau_1,\cdots,\tau_n,\cdots,\tau_l\}$ be an indexed triangulation. Let $B^T$ be the signed adjacency matrix of $T$. We say a quantum cluster algebra $\mathcal A_q$ is \emph{coming from} $\mathcal O$ if there is a quantum seed $t$ such that the upper $n\times n$-submatrix of $\widetilde B(t)$ is $B^T$. By \cite[Theorem 6.1]{BZ} and \cite{FST}, the quantum seeds/ quantum cluster variables of $\mathcal A_q$ and the triangulations/arcs of $\mathcal O$ are one to one correspondence. Let $\Sigma_q^T=(X^{T},\widetilde B^T, \Lambda^T)$ be the quantum seed of $\mathcal A_q$ associate with $T$, where $X^T=\{X^T_{\alpha}\mid \alpha\in T\}\cup \{X_i\mid i\in [n+1,m]\}$ and $X_i,i\in [n+1,m]$ are the coefficients. Let $X_{\gamma}$ be the quantum cluster variable associate with $\gamma$. Set $X_{\gamma}=1$ if $\gamma$ is a boundary arc. For $i\in [1,n]$, denote by $b^{T}_{\tau_i}$ the $i$-th column of $\widetilde B^T$, denote $(b^{T}_{\tau_i})_{+}=([b^{T}_{ji}]_{+})_j$ the positive part of $b^{T}_{\tau_i}$. Dually, denote $(b^{T}_{\tau_i})_{-}=([-b^{T}_{ji}]_{+})_j$. Clearly, $b^{T}_{\tau}=(b^{T}_{\tau})_{+}-(b^{T}_{\tau})_{-}$. Set $e_{\tau_i}=e_i\in \mathbb Z^m$ with the $i$-th coordinate 1 and others 0 for $i\in [1,n]$.

\medskip

We denote $q^{-\frac{1}{2}\sum_{i<j}\Lambda^T_{ij}}(X^{T}_{\tau_1})^{a_1}\cdots (X^{T}_{\tau_n})^{a_n}X_{n+1}^{a_{n+1}}\cdots X_{m}^{a_{m}}$ by $(X^T)^{\vec{a}}$ for any $\vec{a}=(a_1,\cdots,a_m)\in \mathbb Z^m$. Therefore the quantum cluster monomials of $\mathcal A_q$ are the forms $(X^T)^{\vec{a}}$ for indexed triangulation $T$ and $\vec{a}\in \mathbb N^m$.

\medskip

Since $B^T$ is skew-symmetric, we have $(\widetilde B^T)^t\Lambda=(d^TI\;\;0)$ for some positive integer $d^T$, where $I$ is the $n\times n$ identity matrix. Since $(\widetilde B^T)^t \Lambda^T=(\widetilde B^{T'})^t \Lambda^{T'}$, $d^T=d^{T'}$. Let $\gamma$ be an oriented arc in $\mathcal O$ crossing $T$ with points $p_1,\cdots, p_d$ in order. Assume that $p_1,\cdots,p_d$ belong to the arcs $\tau_{i_1},\cdots,\tau_{i_d}$, respectively in $T$. For any $s\in [1,d]$, denote by $m_{p_s}^{+}(\tau_{i_s},\gamma)$ and $m_{p_s}^{-}(\tau_{i_s},\gamma)$ the numbers of $\tau_{i_s}$ in $\{\tau_{i_t}\mid t>s\}$ and $\{\tau_{i_t}\mid t<s\}$, respectively.

\medskip

Let $P$ be a perfect matching of $G_{T,\gamma}$ with edges labeled $\tau_{j_1},\cdots, \tau_{j_r}$ in order. When $P$ can twist on $G(p_s)$, assume $\tau_{j_t},\tau_{j_{t+1}}$ are edges of $G(p_s)$. Denote by $n_{p_s}^{+}(\tau_{i_s},P)$ and $n_{p_s}^{-}(\tau_{i_s},P)$ the numbers of $\tau_{i_s}$ in $\{\tau_{j_u}\mid u>t+1\}$ and $\{\tau_{j_u}\mid u<t\}$, respectively.

\medskip

Fix $s\in [1,d]$. Assume $a_{1_s},a_{4_s},\tau_{i_s}$ and $a_{2_s},a_{3_s},\tau_{i_s}$ are two triangles in $T$ such that $a_{1_s},a_{3_s}$ are clockwise to $\tau_{i_s}$ and $a_{2_s},a_{4_s}$ are counterclockwise to $\tau_{i_s}$. Note that $a_{i_s},i\in 1,2,3,4$ may be boundary arcs.

\medskip

\begin{Definition}\label{omega}

Suppose that $P$ can twist on $G(p_s)$. If the edges labeled $a_{2_s},a_{4_s}$ of $G(p_s)$ are in $P$, define
$$\Omega(p_s, P)=[n_{p_s}^{+}(\tau_{i_s}, P)-m_{p_s}^{+}(\tau_{i_s}, \gamma)-n_{p_s}^{-}(\tau_{i_s}, P)+m_{p_s}^{-}(\tau_{i_s}, \gamma)]d^T,$$ otherwise, define $$\Omega(p_s, P)=-[n_{p_s}^{+}(\tau_{i_s}, P)-m_{p_s}^{+}(\tau_{i_s}, \gamma)-n_{p_s}^{-}(\tau_{i_s}, P)+m_{p_s}^{-}(\tau_{i_s}, \gamma)]d^T.$$
\end{Definition}

Clearly, we have $\Omega(p_s,\mu_{p_s}(P))=-\Omega(p_s,P)$ if $P$ can twist on $G(p_s)$.

\medskip

\section{Commutative Laurent expansions}\label{cle}

In this section, we give our main result on commutative cluster algebras. Precisely, we give a new proof of the Laurent expansion of a cluster variable with respect to arbitrary cluster of the cluster algebra coming from an unpunctured surface with arbitrary coefficients.

\medskip

Throughout this section, let $(\mathcal O,M)$ be an unpunctured surface, $T$ be an indexed triangulation of $\mathcal O$. Let $\mathcal A$ be a cluster algebra from $\mathcal O$ with semi-field $\mathbb P$. By \cite{FST}, the seeds/cluster variables of $\mathcal A$ and the triangulations/arcs of $\mathcal O$ are one to one correspondence. Let $\Sigma^T=(x^T,y^T,B^T)$ be the seed associated with $T$, where $x^{T}=\{x^{T}_{\alpha}\mid \alpha\in T\}$, $y^{T}=\{y^{T}_{\alpha}\mid \alpha\in T\}\subset \mathbb P$. Let $x_{\gamma}$ be the cluster variable associate with $\gamma$. Set $x_{\gamma}=1$ if $\gamma$ is a boundary arc. Fix $\tau\in T$, denote $\mu_\tau(T)=T'$.

\medskip

\begin{Definition-Lemma}\cite{MSW}

\begin{enumerate}[$(1)$]

  \item If $\gamma$ is an oriented arc and $\tau_{i_1},\cdots, \tau_{i_d}$ is the sequence of arcs in $T$ which $\gamma$ crosses, then we define the \emph{crossing monomial} of $\gamma$ with respect to $T$ to be $$c(\gamma, T)=\textstyle\prod_{j=1}^dx^T_{\tau_{i_j}}.$$

  \item Let $P\in \mathcal P(G_{T,\gamma})$. If the edges of $P$ are labeled $\tau_{j_1},\cdots, \tau_{j_r}$, then we define the \emph{weight} $w^T(P)$ of $P$ to be $x^T_{\tau_{j_1}}\cdots x^T_{\tau_{j_r}}$.

  \item Let $P\in \mathcal P(G_{T,\gamma})$. The set $(P_{-}(G_{T,\gamma})\cup P)\setminus (P_{-}(G_{T,\gamma})\cap P)$ is the set of boundary edges of a (possibly disconnected) subgraph $G_P$ of $G_{T,\gamma}$, which is a union of cycles. These cycles enclose a set of tiles $\bigcup_{j\in J}G(p_{i_j})$, where $J$ is a finite index set. We define the \emph{height monomial} $y^T(P)$ of $P$ by
      $$y^T(P)=\textstyle\prod_{k=1}^n (y^{T}_{\tau_k})^{m_k},$$
      where $m_k$ is the number of tiles in $\bigcup_{j\in J}G(p_{i_j})$ whose diagonal is labeled $\tau_k$.

\end{enumerate}

\end{Definition-Lemma}

\medskip

It should be noted that $x_{\alpha}=1$ if $\alpha$ is a boundary arc.

\medskip

\begin{Definition}\label{com-mon}

Let $Q\in \mathcal P(G_{T,\gamma})$. We define the cluster monomial $x^T(Q)$ associated with $Q$ to be
$$x^T(Q)=\frac{w^T(Q)\cdot y^T(Q)}{c(\gamma, T)\cdot \bigoplus_{P\in \mathcal P(G_{T,\gamma})}y^T(P)},$$
note that the operation $``\oplus"$ in $\bigoplus_{P\in \mathcal P(G_{T,\gamma})}y^T(P)$ is taken in $\mathbb P$.
\end{Definition}

\medskip

According to the definition of $y^T(P)$, we have the following crucial lemma.

\medskip

\begin{Lemma}\label{brick1}

Let $P\in \mathcal P(G_{T,\gamma})$. Suppose $P$ can twist on a tile $G(p)$ with diagonal labeled $a$. Assume the labels of the edges of $G(p)$ are $u_1,u_2,u_3,u_4$ with $u_1,u_3$ are clockwise to $a$ and $u_2,u_3$ are counterclockwise to $a$ in $T$. If the edges labeled $u_1,u_3$ of $G(p)$ are in $P$, then $\frac{y^T(\mu_pP)}{y^T(P)}=y^T_{a}$.

\end{Lemma}

\begin{proof}

For any tile $G(p')\neq G(p)$, it is clear $(P_{-}(G_{T,\gamma})\cup \mu_pP)\setminus (P_{-}(G_{T,\gamma})\cap \mu_pP)$ enclose $G(p')$ if and only if $(P_{-}(G_{T,\gamma})\cup P)\setminus (P_{-}(G_{T,\gamma})\cap P)$ enclose $G(p')$.

When there is an edge of $G(p)$ belongs to $P_{-}(G_{T,\gamma})$. By Lemma \ref{max-min}, the edge labeled $u_1$ or $u_3$ is in $P$. Thus $(P_{-}(G_{T,\gamma})\cup \mu_pP)\setminus (P_{-}(G_{T,\gamma})\cap \mu_pP)$ enclose $G(p)$ but $(P_{-}(G_{T,\gamma})\cup P)\setminus (P_{-}(G_{T,\gamma})\cap P)$ not. Thus our result follows in this case.

When there is no edge of $G(p)$ belongs to $P_{-}(G_{T,\gamma})$. Assume $G(p)=G(p_i)$ for some $i$. Then either $G(p_{i-1})$ is left to $G(p_i)$ and $G(p_{i+1})$ is right to $G(p_i)$ or $G(p_{i-1})$ is below $G(p_i)$ and $G(p_{i+1})$ is above $G(p_i)$. We may assume the former case happens. Then $i>1$ and $i$ is odd. Thus $(P_{-}(G_{T,\gamma})\cup \mu_pP)\setminus (P_{-}(G_{T,\gamma})\cap \mu_pP)$ enclose $G(p)$ but $(P_{-}(G_{T,\gamma})\cup P)\setminus (P_{-}(G_{T,\gamma})\cap P)$ not. Thus our result follows in this case.
\end{proof}

\begin{Definition} Let $S,S'$ be two sets. A subset $\mathcal{S}$ of its power set is called \emph{a partition} of $S$ if $\cup_{R\in \mathcal S}R=S$ and $R\cap R'=\emptyset$ if $R\neq R'\in \mathcal S$. \emph{A partition map} from a set $S$ to a set $S'$ is a map from a partition of $S$ to a partition of $S'$. A partition map is called \emph{a partition bijection} if it is a bijection.

\end{Definition}

\medskip

\begin{Remark}

To give a partition bijection from $S$ to $S'$ is equivalent to associate a subset $\pi(s)\subset S'$ for each $s\in S$ which satisfies $\pi(s_1)=\pi(s_2)$ if $\pi(s_1)\cap \pi(s_2)\neq \emptyset$ for $s_1,s_2\in S$ and $\cup_{s\in S}\pi(s)=S'$.

\end{Remark}

\medskip

With the above preparations, we now state our first main result on commutative cluster algebras from unpunctured surfaces.

\medskip

\begin{Theorem}\label{partition bi}

Let $(\mathcal O,M)$ be an unpunctured surface and $T$ be an indexed triangulation. Let $\gamma$ be an oriented arc in $\mathcal O$. For any $\tau\in T$, there are partitions $\mathfrak P$ and $\mathfrak P'$ of $\mathcal P(G_{T,\gamma})$ and $\mathcal P(G_{T',\gamma})$, respectively, and a bijection $\pi: \mathfrak P\rightarrow \mathfrak P'$. Moreover, $\pi$ satisfies

\begin{enumerate}[$(1)$]

  \item $|S|=1$ or $|\pi(S)|=1$ for any $S\in \mathfrak P$.

  \item $\sum_{P\in S}x^T(P)=\sum_{P'\in \pi(S)}x^{T'}(P')$ for any $S\in \mathfrak P$.

\end{enumerate}

\end{Theorem}

\medskip

\begin{Remark}

The partitions $\mathfrak P, \mathfrak P'$ depend on $T, \gamma$ and $\tau$. To avoid the lengthy of the symbols, we do not write them as $\mathfrak P(T, \gamma, \tau), \mathfrak P'(T, \gamma, \tau)$.

\end{Remark}

\medskip

As corollary of Theorem \ref{partition bi}, we can give the commutative Laurent expansion.

\medskip

\begin{Theorem}\label{expansion-comm}

Let $(\mathcal O,M)$ be an unpuntured surface and $T$ be an indexed triangulation. If $\gamma$ is an oriented arc in $\mathcal O$, then the Laurent expansion of $x_{\gamma}$ with respect to the cluster $x^{T}$ is $$x_{\gamma}=\textstyle\sum_{P\in \mathcal P(G_{T,\gamma})} x^T(P).$$

\end{Theorem}

\begin{proof}

We prove the Theorem by the induction of the crossing number $N(\gamma,T)$ of $\gamma$ and $T$. If $N(\gamma,T)=0$, then $\gamma\in T$, $P=\{\gamma\}$. Clearly, we have $x_{\gamma}=x^T_{\gamma}$. Assume $N(\gamma,T)>0$ and the result holds for the case the crossing number is less than $N(\gamma,T)$. We can choose an arc $\tau\in T$ such that $N(\gamma,T')<N(\gamma, T)$. By induction hypothesis, $x_{\gamma}=\sum_{P'\in \mathcal P(G_{T',\gamma})} x^{T'}(P').$ By Theorem \ref{partition bi}, $\sum_{P'\in \mathcal P(G_{T',\gamma})} x^{T'}(P')=\sum_{P\in \mathcal P(G_{T,\gamma})} x^{T}(P).$ Therefore, $$x_{\gamma}=\textstyle\sum_{P\in \mathcal P(G_{T,\gamma})} x^T(P).$$ Our result follows.
\end{proof}

\medskip

\section{Quantum Laurent expansions}\label{qle}

In this section, we give our main result on quantum cluster algebras. Precisely, we give the quantum Laurent expansion of a quantum cluster variable with respect to arbitrary quantum cluster of $\mathcal A_q$. As an application, we prove the positivity for such class of quantum cluster algebras.

\medskip

When specialize $q=1$, we get a commutative cluster algebra $\mathcal A_q\mid_{q=1}$. Let $P\in \mathcal P(G_{T,\gamma})$. For the element $x^T(P)$ associate with $P$ in Definition \ref{com-mon}, there clearly exists a unique $\vec{a}(P)\in \mathbb Z^m$ such that $(X^T)^{\vec{a}(P)}\mid_{q=1}=x^T(P)$. Denote $(X^T)^{\vec{a}(P)}$ by $X^T(P)$. Let $X_{\gamma}$ be the quantum cluster variable associate with $\gamma$.

\medskip

We now state our main result on quantum cluster algebras from unpunctured surfaces.

\medskip

\begin{Theorem}\label{mainthm}

Let $(\mathcal O,M)$ be a surface without punctures and $T$ be an indexed triangulation. Let $\gamma$ be an oriented arc in $\mathcal O$.

\begin{enumerate}[$(1)$]

  \item  There uniquely exists a \emph{valuation} map $v:\mathcal P(G_{T,\gamma})\rightarrow \mathbb Z$ such that

\begin{enumerate}[$(a)$]

  \item (initial conditions) $v(P_{+}(G_{T,\gamma}))=v(P_{-}(G_{T,\gamma}))=0$.

  \item (iterated relation) If $P\in \mathcal P(G_{T,\gamma}))$ can twist on a tile $G(p)$, then $$v(P)-v(\mu_pP)=\Omega(p,P).$$

\end{enumerate}

  \item Further, if $\tau\in T$ and $\pi$ is the partition bijection from $\mathcal P(G_{T,\gamma})$ to $\mathcal P(G_{\mu_\tau T, \gamma})$ given in Theorem \ref{partition bi}, then, for any $S\in \mathfrak P$,
      $$\textstyle\sum_{P\in S}q^{v(P)/2}X^T(P)=\textstyle\sum_{P'\in \pi(S)}q^{v(P')/2}X^{T'}(P').$$

\end{enumerate}

\end{Theorem}

\medskip

As a corollary of Theorem \ref{mainthm}, we can give the quantum Laurent expansion.

\medskip

\begin{Theorem}\label{expansion}

Let $(\mathcal O,M)$ be an unpuntured surface and $T$ be an indexed triangulation. If $\gamma$ be an oriented arc in $\mathcal O$, then the quantum Laurent expansion of $X_{\gamma}$ with respect to the quantum cluster $X^{T}$ is $$X_{\gamma}=\textstyle\sum_{P\in \mathcal P(G_{T,\gamma})} q^{v(P)/2}X^T(P).$$

\end{Theorem}

\begin{proof}

We prove the theorem by the induction of the crossing number $N(\gamma,T)$ of $\gamma$ and $T$. If $N(\gamma,T)=0$, then $\gamma\in T$, $P=\{\gamma\}$ and $v(P)=0$. Clearly, we have $X_{\gamma}=X_{\gamma}$. Assume $N(\gamma,T)>0$ and the result holds for the case the crossing number is less than $N(\gamma,T)$. We can choose an arc $\tau\in T$ such that $N(\gamma,T')<N(\gamma, T)$. By induction hypothesis, $X_{\gamma}=\sum_{P'\in \mathcal P(G_{T',\gamma})} q^{v(P')/2}X^{T'}(P').$ By Theorem \ref{partition bi} and Theorem \ref{mainthm} (2), $\sum_{P'\in \mathcal P(G_{T',\gamma})} q^{v(P')/2}X^T(P')=\sum_{P\in \mathcal P(G_{T,\gamma})} q^{v(P)/2}X^T(P).$ Therefore, $$X_{\gamma}=\textstyle\sum_{P\in \mathcal P(G_{T,\gamma})} q^{v(P)/2}X^T(P).$$ Our result follows.
\end{proof}

\medskip

The remainder question is: can we give an explicit formula for the valuation map $v$?

\medskip

As an immediately corollary of Theorem \ref{expansion}, we have the positivity of our case.

\medskip

\begin{Theorem}\label{positive}

Let $\mathcal O$ be a surface without punctures. The positivity holds for the quantum cluster algebra $\mathcal A_q(\mathcal O)$.

\end{Theorem}

\begin{proof}

Since $q^{v(P)}\in \mathbb N[q^{\pm 1}]$ for any $P\in \mathcal P(G_{T,\gamma})$, by Theorem \ref{expansion}, our result follows.
\end{proof}

\medskip

\begin{Example}

Let $(\mathcal O, M)$, $T$ and $\gamma$ be as shown in the following figure. Assume $\mathcal A_q$ with principal coefficients at $T$ and principal quantization, i.e. $\widetilde B^T=((B^T)^t\;\;I_2)^t$ and $\Lambda^T=\left(\begin{array}{cc}
0 & -I_2\\
I_2 & B^T
\end{array}\right)$, where $I_2$ is the $2\times 2$ identity matrix.

\medskip

\centerline{\includegraphics{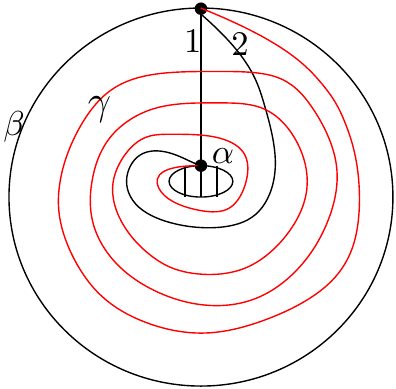}}

The arcs are labeled $1,2,\alpha,\beta$, as shown. Thus, $G_{T,\gamma}$ is the following.

\centerline{\begin{tikzpicture}
\draw[-] (0,0) -- (1,0);
\draw[-] (0,0) -- (0,-1);
\draw[-] (0,-1) -- (1,-1);
\draw[-] (1,0) -- (2,0);
\draw[-] (1,-1) -- (2,-1);
\draw[-] (2,0) -- (3,0);
\draw[-] (2,-1) -- (3,-1);
\draw[-] (3,0) -- (3,-1);
\draw[-] (3,0) -- (4,0);
\draw[-] (3,-1) -- (4,-1);
\draw[-] (4,0) -- (4,-1);
\draw[-] (4,0) -- (5,0);
\draw[-] (4,-1) -- (5,-1);
\draw[-] (5,0) -- (5,-1);
\draw[dashed] (1,0) -- (2,-1);
\draw[dashed] (0,0) -- (1,-1);
\draw[dashed] (2,0) -- (3,-1);
\draw[dashed] (3,0) -- (4,-1);
\draw[dashed] (4,0) -- (5,-1);
\node[above] at (0.5,-0.05) {$2$};
\node[above] at (1.5,-0.05) {$1$};
\node[above] at (2.5,-0.05) {$2$};
\node[above] at (3.5,-0.05) {$1$};
\node[above] at (4.5,-0.05) {$2$};
\node[below] at (4.5,-0.95) {$2$};
\node[below] at (0.5,-0.95) {$2$};
\node[below] at (1.5,-0.95) {$1$};
\node[below] at (2.5,-0.95) {$2$};
\node[below] at (3.5,-0.95) {$1$};
\node[right] at (1.4,-0.45) {$2$};
\node[right] at (3.4,-0.45) {$2$};
\node[right] at (2.4,-0.5) {$1$};
\node[right] at (4.4,-0.5) {$1$};
\node[left] at (0.8,-0.5) {$1$};
\node[left] at (0.05,-0.5) {$\alpha$};
\node[left] at (1.25,-0.5) {$\beta$};
\node[left] at (2.25,-0.5) {$\alpha$};
\node[left] at (3.25,-0.5) {$\beta$};
\node[left] at (4.25,-0.5) {$\alpha$};
\node[right] at (4.95,-0.5) {$\beta$};
\draw [-] (1, 0) -- (1,-1);
\draw [-] (2, 0) -- (2,-1);
\draw [fill] (0,0) circle [radius=.05];
\draw [fill] (0,-1) circle [radius=.05];
\draw [fill] (1,0) circle [radius=.05];
\draw [fill] (2,0) circle [radius=.05];
\draw [fill] (1,-1) circle [radius=.05];
\draw [fill] (2,-1) circle [radius=.05];
\draw [fill] (3,0) circle [radius=.05];
\draw [fill] (3,-1) circle [radius=.05];
\draw [fill] (4,0) circle [radius=.05];
\draw [fill] (4,-1) circle [radius=.05];
\draw [fill] (5,0) circle [radius=.05];
\draw [fill] (5,-1) circle [radius=.05];
\end{tikzpicture}}

\centerline{\includegraphics{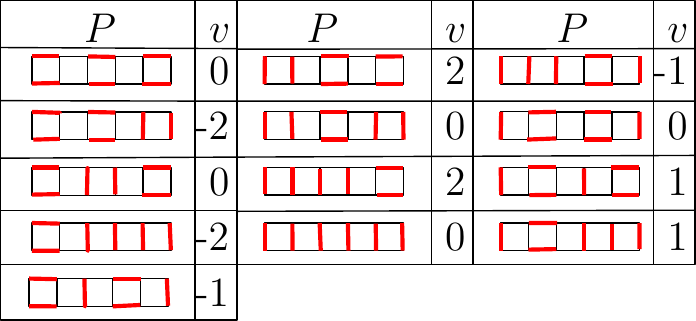}}

\medskip

Therefore, $$
\begin{array}{rcl} X_{\gamma}

& = & q^{0}x^{-3e_1+4e_2+3e_3+2e_4}+q^{2/2}x^{-3e_1+2e_2+2e_3+2e_4}+q^{-1/2}x^{-e_1-2e_2+e_4}  \vspace{2pt}  \\

& + & q^{-2/2}x^{-3e_1+2e_2+2e_3+2e_4}+q^{0/2}x^{-3e_1+e_3+2e_4}+q^{0/2}x^{e_1-2e_2} \vspace{2pt}  \\

& + & q^{0/2}x^{-3e_1+2e_2+2e_3+2e_4}+q^{2/2}x^{-3e_1+e_3+2e_4}+q^{1/2}x^{-e_1+e_3+e_4} \vspace{2pt}  \\

& + & q^{-2/2}x^{-3e_1+e_3+2e_4}+q^{0/2}x^{-3e_1-2e_2+2e_4}+q^{1/2}x^{-e_1-2e_2+e_4} \vspace{2pt}  \\

& + & q^{-1/2}x^{-e_1+e_3+e_4} \vspace{2pt}  \\

& = & x^{-3e_1+4e_2+3e_3+2e_4}+ +x^{e_1-2e_2} + x^{-3e_1-2e_2+2e_4}\vspace{2pt}\\

& + & (q^{1/2}+q^{-1/2})x^{-e_1-2e_2+e_4} + (q^{1/2}+q^{-1/2})x^{-e_1+e_3+e_4} \vspace{2pt}  \\

& + & (q+1+q^{-1})x^{-3e_1+e_3+2e_4} + (q+1+q^{-1})x^{-3e_1+2e_2+2e_3+2e_4}.
\end{array}$$

\end{Example}

\medskip

\section{Comparison of $\mathcal P(G_{T,\rho})$ and $\mathcal P(G_{T',\rho})$}\label{compare}

Let $T$ be an indexed triangulation of $\mathcal O$ and $\gamma$ be an oriented arc in $\mathcal O$. Assume $T$ can do flip at $\tau$. Denote by $\tau'$ the arc obtained from $T$ by flip at $\tau$. Let $p_0$ be the starting point of $\gamma$, and let $p_{d+1}$ be its endpoint. Assume $\gamma$ crosses $T$ at $p_1,\cdots, p_d$ in order. Thus, $p_j,j\in [1,d]$ divide $\gamma$ into some segments. Assume $p_j\in \tau_{i_j}\in T$. For $j\in [1,d-1]$, $\tau_{i_j}$ and $\tau_{i_{j+1}}$ form two edges of a triangle $\Delta_j$ of $T$ such that the segment connecting $p_j$ and $p_{j+1}$ lies inside of $\Delta_j$. Denote the third edge of $\Delta_j$ by $\tau_{[\gamma_j]}$. We choose a point on the segment connecting $p_j$ and $p_{j+1}$ for the following two cases:

\begin{enumerate}[$(1)$]

  \item one of $\tau_{i_j}$ and $\tau_{i_{j+1}}$ is in a same triangle with $\tau$, the other one is not.

  \item $\tau_{i_j},\tau_{i_{j+1}}\neq \tau$, $\tau_{i_j}$ and $\tau_{i_{j+1}}$ are in a same triangle with $\tau$ but $\tau_{[\gamma_j]}\neq \tau$.

\end{enumerate}

\medskip

Denote the points chosen by $o_1,o_2,\cdots,o_{k-1}$ in order and denote $p_0=o_0, p_{d+1}=o_{k}$. Denote the subcurve connecting $o_{i-1}$ and $o_i$ of $\gamma$ by $\gamma_{i}$ for $i\in [1,k]$. See the following figures for example.

\centerline{\includegraphics{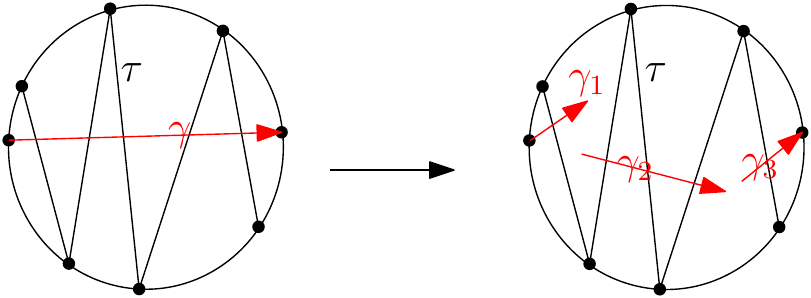}}

\centerline{\includegraphics{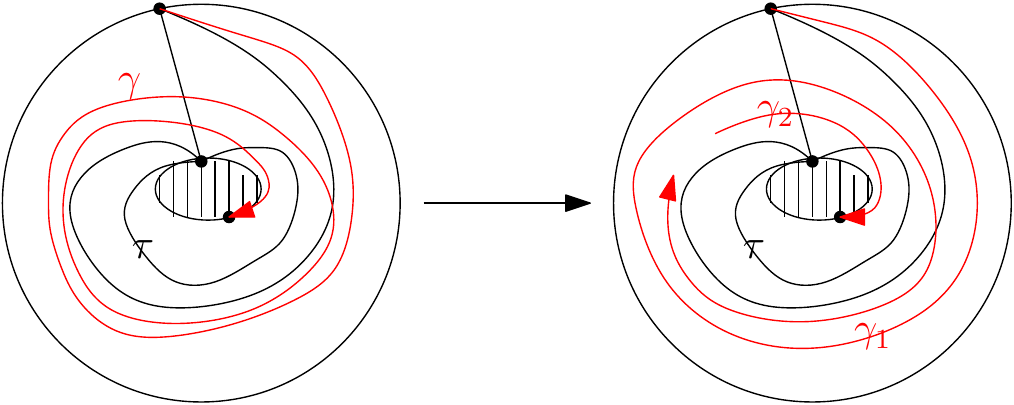}}

\medskip

Let $\rho$ be one of $\gamma_i,i\in [1,k]$. In this section, we compare $\mathcal P(G_{T,\rho})$ with $\mathcal P(G_{\mu_{\tau}(T),\rho})$. Note that if $\rho$ crosses no arc which is in the same triangle with $\tau$, then $G_{T,\rho}=G_{\mu_{\tau}(T),\rho}$. Throughout this section, assume $\tau'\neq \gamma\notin T$ and $\rho$ crosses at least one arc which is in the same triangle with $\tau$. We do not consider the orientation of $\rho$.

\medskip

Denote the first tile and the last tile of $G_{T,\rho}$ by $G_1$ and $G_s$, respectively. Then $G_{T,\gamma}$ can be obtained by gluing some graph $G$ left or below to $G_1$ and some graph $G'$ right or upper to $G_s$. Note that $G$ is empty if and only if $\rho$ and $\gamma$ have the same starting point, $G'$ is empty if and only if $\rho$ and $\gamma$ have the same endpoint. When $G$ is not empty, we say the left almost or lower almost edge of $G_1$ is the \emph{first gluing edge} of $G_{T,\rho}$, and when $G'$ is not empty, we say the right almost or upper almost edge of $G_s$ is the \emph{last gluing edge} of $G_{T,\rho}$. Similarly, we can define the first/last gluing edge of $G_{T',\rho}$. Assume $\rho$ crosses $T$ at $p_j,\cdots,p_r$ for some $1\leq j\leq r\leq d$, one can see the labels of the first/last gluing edges of $G_{T,\rho}$ and $G_{\mu_{\tau}(T),\rho}$ are the same, precisely, the first gluing edges are labeled $\tau_{[\gamma_{j-1}]}$ and the last gluing edges are labeled $\tau_{[\gamma_r]}$.

\medskip

For two edges of $G_{T,\rho}$ labeled $\tau$, we say that they are \emph{equivalent} if they are incident to a same diagonal of a tile of $G_{T,\rho}$. Each such equivalence class is called a \emph{$\tau$-equivalence class} in $G_{T,\rho}$. We divide the edges of $G_{T,\rho}$ labeled $\tau$ into the following types for all cases up to a difference of relative orientation.

\begin{enumerate}[$(I)$]

  \item there are two non-incident edges in its equivalence class.

  \item there are two incident edges in its equivalence class.

  \item there is one edge in its equivalence class and it is incident to a diagonal.

  \item there is one edge in its equivalence class and it is not incident to a diagonal.

\end{enumerate}

Type (I):

\centerline{\begin{tikzpicture}
\draw[-] (0,0) -- (1,0);
\draw[-] (0,0) -- (0,-1);
\draw[-] (0,-1) -- (1,-1);
\draw[-] (1,0) -- (2,0);
\draw[-] (1,-1) -- (2,-1);
\draw[-] (2,0) -- (3,0);
\draw[-] (2,-1) -- (3,-1);
\draw[-] (3,0) -- (3,-1);
\draw[dashed] (1,0) -- (2,-1);
\draw[dashed] (0,0) -- (1,-1);
\draw[dashed] (2,0) -- (3,-1);
\node[right] at (1.4,-0.45) {$\tau$};
\node[above] at (0.5,0) {$\tau$};
\node[below] at (2.5,-1) {$\tau$};
\draw [-] (1, 0) -- (1,-1);
\draw [-] (2, 0) -- (2,-1);
\draw [fill] (0,0) circle [radius=.05];
\draw [fill] (0,-1) circle [radius=.05];
\draw [fill] (1,0) circle [radius=.05];
\draw [fill] (2,0) circle [radius=.05];
\draw [fill] (1,-1) circle [radius=.05];
\draw [fill] (2,-1) circle [radius=.05];
\draw [fill] (3,0) circle [radius=.05];
\draw [fill] (3,-1) circle [radius=.05];
%\draw[-] (5,0) -- (6,0);
%\draw[-] (5,-1) -- (6,-1);
%\draw [-] (5, 0) -- (5,-1);
%\draw [-] (6, 0) -- (6,-1);
%\draw[-] (5,0) -- (5,1);
%\draw[-] (5,-1) -- (5,-2);
%\draw[-] (6,0) -- (6,1);
%\draw[-] (6,-1) -- (6,-2);
%\draw[-] (5,1) -- (6,1);
%\draw[-] (5,-2) -- (6,-2);
%\draw[dashed] (5,0) -- (6,-1);
%\draw[dashed] (5,1) -- (6,0);
%\draw[dashed] (5,-1) -- (6,-2);
%\node[right] at (6,-1.5) {$\tau$};
%\node[left] at (5.9,-0.4) {$\tau$};
%\node[left] at (5,0.5) {$\tau$};
%\draw [fill] (5,0) circle [radius=.05];
%\draw [fill] (6,0) circle [radius=.05];
%\draw [fill] (5,-1) circle [radius=.05];
%\draw [fill] (6,-1) circle [radius=.05];
%\draw [fill] (6,-2) circle [radius=.05];
%\draw [fill] (5,-2) circle [radius=.05];
%\draw [fill] (5,1) circle [radius=.05];
%\draw [fill] (6,1) circle [radius=.05];
\end{tikzpicture}}

Type (II):

\centerline{\begin{tikzpicture}
\draw[-] (1,0) -- (2,0);
\draw[-] (1,-1) -- (2,-1);
\draw[-] (2,0) -- (3,0);
\draw[-] (2,-1) -- (3,-1);
\draw[-] (3,0) -- (3,-1);
\draw[-] (1,-1) -- (1,-2);
\draw[-] (2,-1) -- (2,-2);
\draw[-] (1,-2) -- (2,-2);
\draw[dashed] (1,0) -- (2,-1);
\draw[dashed] (2,0) -- (3,-1);
\draw[dashed] (1,-1) -- (2,-2);
\node[right] at (1.4,-0.45) {$\tau$};
\node[right] at (1.95,-1.5) {$\tau$};
\node[below] at (2.5,-1) {$\tau$};
\draw [-] (1, 0) -- (1,-1);
\draw [-] (2, 0) -- (2,-1);
\draw [fill] (1,0) circle [radius=.05];
\draw [fill] (2,0) circle [radius=.05];
\draw [fill] (1,-1) circle [radius=.05];
\draw [fill] (2,-1) circle [radius=.05];
\draw [fill] (1,-2) circle [radius=.05];
\draw [fill] (2,-2) circle [radius=.05];
\draw [fill] (3,0) circle [radius=.05];
\draw [fill] (3,-1) circle [radius=.05];
%\draw[-] (6,0) -- (7,0);
%\draw[-] (6,-1) -- (7,-1);
%\draw[-] (7,0) -- (7,-1);
%\draw[-] (7,-1) -- (7,-2);
%\draw[-] (6,-2) -- (7,-2);
%\draw[-] (5,-1) -- (6,-1);
%\draw [-] (6, 0) -- (6,-1);
%\draw[-] (6,-1) -- (7,-1);
%\draw[-] (5,-2) -- (6,-2);
%\draw[-] (5,-1) -- (5,-2);
%\draw[-] (6,-1) -- (6,-2);
%\draw[dashed] (5,-1) -- (6,-2);
%\draw[dashed] (6,0) -- (7,-1);
%\draw[dashed] (6,-1) -- (7,-2);
%\node[above] at (5.5,-1.05) {$\tau$};
%\node[left] at (6.05,-0.5) {$\tau$};
%\node[right] at (6.4,-1.45) {$\tau$};
%\draw [fill] (7,0) circle [radius=.05];
%\draw [fill] (7,-2) circle [radius=.05];
%\draw [fill] (7,-1) circle [radius=.05];
%\draw [fill] (6,0) circle [radius=.05];
%\draw [fill] (5,-1) circle [radius=.05];
%\draw [fill] (6,-1) circle [radius=.05];
%\draw [fill] (5,-2) circle [radius=.05];
%\draw [fill] (6,-2) circle [radius=.05];
\end{tikzpicture}}

Type (III)

\centerline{\begin{tikzpicture}
\draw[-] (1,0) -- (2,0);
\draw[-] (1,-1) -- (2,-1);
\draw[-] (1,-1) -- (1,-2);
\draw[-] (2,-1) -- (2,-2);
\draw[-] (1,-2) -- (2,-2);
\draw[dashed] (1,0) -- (2,-1);
\draw[dashed] (1,-1) -- (2,-2);
\node[right] at (1.4,-0.45) {$\tau$};
\node[right] at (1.95,-1.5) {$\tau$};
\draw [-] (1, 0) -- (1,-1);
\draw [-] (2, 0) -- (2,-1);
\draw [fill] (1,0) circle [radius=.05];
\draw [fill] (2,0) circle [radius=.05];
\draw [fill] (1,-1) circle [radius=.05];
\draw [fill] (2,-1) circle [radius=.05];
\draw [fill] (1,-2) circle [radius=.05];
\draw [fill] (2,-2) circle [radius=.05];
%\draw[-] (3,0) -- (4,0);
%\draw[-] (3,-1) -- (4,-1);
%\draw[-] (4,0) -- (5,0);
%\draw[-] (4,-1) -- (5,-1);
%\draw[-] (5,0) -- (5,-1);
%\draw[dashed] (3,0) -- (4,-1);
%\draw[dashed] (4,0) -- (5,-1);
%\node[right] at (3.4,-0.45) {$\tau$};
%\node[below] at (4.5,-1) {$\tau$};
%\draw [-] (3, 0) -- (3,-1);
%\draw [-] (4, 0) -- (4,-1);
%\draw [fill] (3,0) circle [radius=.05];
%\draw [fill] (4,0) circle [radius=.05];
%\draw [fill] (3,-1) circle [radius=.05];
%\draw [fill] (4,-1) circle [radius=.05];
%\draw [fill] (5,0) circle [radius=.05];
%\draw [fill] (5,-1) circle [radius=.05];
\draw[-] (6,0) -- (7,0);
\draw[-] (6,-1) -- (7,-1);
\draw[-] (7,0) -- (7,-1);
\draw[-] (7,-1) -- (7,-2);
\draw[-] (6,-2) -- (7,-2);
\draw [-] (6, 0) -- (6,-1);
\draw[-] (6,-1) -- (7,-1);
\draw[-] (6,-1) -- (6,-2);
\draw[dashed] (6,0) -- (7,-1);
\draw[dashed] (6,-1) -- (7,-2);
\node[left] at (6.05,-0.5) {$\tau$};
\node[right] at (6.4,-1.45) {$\tau$};
\draw [fill] (7,0) circle [radius=.05];
\draw [fill] (7,-2) circle [radius=.05];
\draw [fill] (7,-1) circle [radius=.05];
\draw [fill] (6,0) circle [radius=.05];
\draw [fill] (6,-1) circle [radius=.05];
\draw [fill] (6,-2) circle [radius=.05];
%\draw[-] (9,-1) -- (10,-1);
%\draw[-] (10,-1) -- (10,-2);
%\draw[-] (9,-2) -- (10,-2);
%\draw[-] (8,-1) -- (9,-1);
%\draw[-] (9,-1) -- (10,-1);
%\draw[-] (8,-2) -- (9,-2);
%\draw[-] (8,-1) -- (8,-2);
%\draw[-] (9,-1) -- (9,-2);
%\draw[dashed] (8,-1) -- (9,-2);
%\draw[dashed] (9,-1) -- (10,-2);
%\node[above] at (8.5,-1.05) {$\tau$};
%\node[right] at (9.4,-1.45) {$\tau$};
%\draw [fill] (10,-2) circle [radius=.05];
%\draw [fill] (10,-1) circle [radius=.05];
%\draw [fill] (8,-1) circle [radius=.05];
%\draw [fill] (9,-1) circle [radius=.05];
%\draw [fill] (8,-2) circle [radius=.05];
%\draw [fill] (9,-2) circle [radius=.05];
\end{tikzpicture}}

Type (IV)

\centerline{\begin{tikzpicture}
\draw[-] (1,0) -- (2,0);
\draw[-] (1,-1) -- (2,-1);
\draw[-] (1,0) -- (1,-1);
\draw[-] (2,0) -- (2,-1);
\draw[dashed] (1,0) -- (2,-1);
\draw [fill] (1,0) circle [radius=.05];
\draw [fill] (2,0) circle [radius=.05];
\draw [fill] (1,-1) circle [radius=.05];
\draw [fill] (2,-1) circle [radius=.05];
\node[left] at (1.05,-0.5) {$\tau$};
%\draw[-] (3,0) -- (4,0);
%\draw[-] (3,-1) -- (4,-1);
%\draw[-] (3,0) -- (3,-1);
%\draw[-] (4,0) -- (4,-1);
%\draw[dashed] (3,0) -- (4,-1);
%\draw [fill] (3,0) circle [radius=.05];
%\draw [fill] (4,0) circle [radius=.05];
%\draw [fill] (3,-1) circle [radius=.05];
%\draw [fill] (4,-1) circle [radius=.05];
%\node[right] at (3.98,-0.45) {$\tau$};
\draw[-] (5,0) -- (6,0);
\draw[-] (5,-1) -- (6,-1);
\draw[-] (5,0) -- (5,-1);
\draw[-] (6,0) -- (6,-1);
\draw[dashed] (5,0) -- (6,-1);
\draw [fill] (5,0) circle [radius=.05];
\draw [fill] (6,0) circle [radius=.05];
\draw [fill] (5,-1) circle [radius=.05];
\draw [fill] (6,-1) circle [radius=.05];
\node[above] at (5.5,-0.05) {$\tau$};
%\draw[-] (7,0) -- (8,0);
%\draw[-] (7,-1) -- (8,-1);
%\draw[-] (7,0) -- (7,-1);
%\draw[-] (8,0) -- (8,-1);
%\draw[dashed] (7,0) -- (8,-1);
%draw [fill] (7,0) circle [radius=.05];
%\draw [fill] (8,0) circle [radius=.05];
%\draw [fill] (7,-1) circle [radius=.05];
%\draw [fill] (8,-1) circle [radius=.05];
%\node[below] at (7.5,-0.95) {$\tau$};
\end{tikzpicture}}

\medskip

Denote by $n^{\tau}(T,\rho)$ the number of $\tau$-equivalence classes in $G_{T,\rho}$ and we list the $\tau$-equivalence classes according to the order of the tiles. Denote by $S_{T,\rho}$ the set containing all sequences $\nu=(\nu_{1},\cdots,\nu_{n^{\tau}(T,\rho)})$ which satisfies $\nu_{j}\in \{-1,0,1\}$ if the $j$-th $\tau$-equivalence class is of type (I), $\nu_{j}\in \{-1,0\}$ if the $j$-th $\tau$-equivalent class is of type (II) or (III), $\nu_{j}\in \{0,1\}$ if the $j$-th $\tau$-equivalence class is of type (IV). Given a sequence $\nu_=(\nu_{1},\cdots,\nu_{n^{\tau}(T,\rho)})\in S_{T,\rho}$, let $\mathcal P^{\tau}_{\nu}(G_{T,\rho})$ containing all perfect matching $P$ which contains $\nu_{j}+1$ edges of the $j$-th $\tau$-equivalence class if the $j$-th $\tau$-equivalence class is of type (I) (II) or (III), and $\nu_{j}$ edges of the $j$-th $\tau$-equivalence class if the $j$-th $\tau$-equivalence class is of type (IV).

\medskip

It should be noted that if $\nu_j=-1$, then the $j$-th $\tau$-equivalence class corresponds to a diagonal of a tile $G(p)$, moreover, any $P\in \mathcal P^{\tau}_{\nu}(G_{T,\rho})$ can twist on $G(p)$.

\medskip

One can easily see the following result holds.

 \medskip

\begin{Lemma}\label{localdec}

$\mathcal P(G_{T,\rho})=\bigsqcup_{\nu\in S_{T,\rho}}\mathcal P^{\tau}_{\nu}(G_{T,\rho})$.

\end{Lemma}

\medskip

Assume $a_1,a_4,\tau$ and $a_2,a_3,\tau$ are two triangles in $T$ such that $a_1,a_3$ are clockwise to $\tau$ and $a_2,a_4$ are counterclockwise to $\tau$, as shown in the figure below. Since $\mathcal O$ without punctures, $T$ does not have self-folded triangle, $a_2\neq a_3$ and $a_1\neq a_4$. Similarly, do flip at $\tau$, we know $a_1\neq a_2$ and $a_3\neq a_4$.

\centerline{\begin{tikzpicture}
\draw[-] (1,0) -- (2,0);
\draw[-] (1,-1) -- (2,-1);
\draw[-] (1,0) -- (2,-1);
\node[right] at (1.4,-0.45) {$\tau$};
\node[above] at (1.5,-0.05) {$a_2$};
\node[left] at (1,-0.5) {$a_1$};
\node[below] at (1.5,-0.95) {$a_4$};
\node[right] at (2,-0.5) {$a_3$};
%\node[right] at (1.2,-0.5) {$p_{i_1}$};
\draw [-] (1, 0) -- (1,-1);
\draw [-] (2, 0) -- (2,-1);
\draw [fill] (1,0) circle [radius=.05];
\draw [fill] (2,0) circle [radius=.05];
\draw [fill] (1,-1) circle [radius=.05];
\draw [fill] (2,-1) circle [radius=.05];
\end{tikzpicture}}

\subsection{When $a_1\neq a_3$ and $a_2\neq a_4$.}

By the construction of $\rho$ and $\tau'\neq \gamma\notin T$, we have the following possibilities (the addition operation below is taken in $\mathbb Z_4$):

\begin{enumerate}[$(1)$]

  \item $\rho$ crosses $\tau,a_i$ for $i=1,2,3,4$.

  \item $\rho$ crosses $a_i,\tau, a_{i+1}$ for $i=1,3$.

  \item $\rho$ crosses $a_i,\tau,\tau',a_{i+2}$ for $i=1,2$.

  \item $\rho$ crosses $\tau',a_i$ for $i=1,2,3,4$.

  \item $\rho$ crosses $a_i,\tau',a_{i+1}$ for $i=2,4$.

\end{enumerate}

Since case (1) and case (4) are dual, case (2) and case (5) are dual, we shall only consider cases (1), (2) and (3).

\medskip

In case (1), $\rho$ and $\gamma$ have the same starting point. We may assume $\rho$ crosses $\tau, a_2$, then up to a difference of relative orientation, $G_{T,\rho}$ and $G_{T',\rho}$ are the following graphs, respectively,

\centerline{\begin{tikzpicture}
\draw[-] (1,0) -- (2,0);
\draw[-] (1,-1) -- (2,-1);
\draw[-] (2,0) -- (3,0);
\draw[-] (2,-1) -- (3,-1);
\draw[-] (3,0) -- (3,-1);
\draw[dashed] (1,0) -- (2,-1);
\draw[dashed] (2,0) -- (3,-1);
\node[right] at (1.4,-0.45) {$\tau$};
\node[above] at (1.5,-0.05) {$a_2$};
\node[left] at (1,-0.5) {$a_1$};
\node[below] at (1.5,-1) {$a_4$};
\node[right] at (1.75,-0.5) {$a_3$};
\node[right] at (2.4,-0.5) {$a_2$};
\node[above] at (2.5,-0.05) {$a_6$};
%\node[left] at (3,-0.5) {$a_5$};
\node[below] at (2.5,-1) {$\tau$};
\node[right] at (3,-0.5) {$a_5$};
%\node[right] at (1.2,-0.5) {$p_{i_1}$};
\draw [-] (1, 0) -- (1,-1);
\draw [-] (2, 0) -- (2,-1);
\draw [fill] (1,0) circle [radius=.05];
\draw [fill] (2,0) circle [radius=.05];
\draw [fill] (1,-1) circle [radius=.05];
\draw [fill] (2,-1) circle [radius=.05];
\draw [fill] (3,0) circle [radius=.05];
\draw [fill] (3,-1) circle [radius=.05];
\draw[-] (5,0) -- (6,0);
\draw[-] (5,-1) -- (6,-1);
\draw[dashed] (5,0) -- (6,-1);
\node[right] at (5.4,-0.5) {$a_2$};
\node[above] at (5.5,-0.05) {$a_5$};
\node[left] at (5,-0.5) {$a_1$};
\node[below] at (5.5,-0.95) {$\tau'$};
\node[right] at (6,-0.5) {$a_6$};
%\node[right] at (1.2,-0.5) {$p_{i_1}$};
\draw [-] (5, 0) -- (5,-1);
\draw [-] (6, 0) -- (6,-1);
\draw [fill] (5,0) circle [radius=.05];
\draw [fill] (6,0) circle [radius=.05];
\draw [fill] (5,-1) circle [radius=.05];
\draw [fill] (6,-1) circle [radius=.05];
\end{tikzpicture}}

In this case, $G_{T,\rho}$ (respectively $G_{T',\rho}$) has one $\tau$-(respectively $\tau'$-)equivalence class of type (III) (respectively (IV)). Thus $n^{\tau}(T,\rho)=n^{\tau'}(T',\rho)$. Moreover, we have a partition bijection from $\mathcal P(G_{T,\rho})$ to $\mathcal P(G_{T',\rho})$, as shown in the figure below.

\centerline{\includegraphics{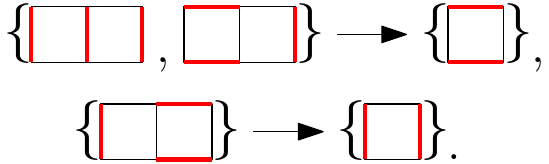}}

It is easy to see the right upper edge labeled $a_5$/$a_6$ of $G_{T,\rho}$ is in the left set if and only if the right upper edge labeled $a_5$/$a_6$ of $G_{T',\rho}$ is in the right set. Thus, the last gluing edge of $G_{T,\rho}$ is in the left set if and only if the last gluing edge of $G_{T',\rho}$ in the right set. Since $\rho$ and $\gamma$ have the same starting point, $G_{T,\rho}$ and $G_{T',\rho}$ do not have first gluing edges.

\medskip
%
%In this case, $P_{\pm}(G_{T',\rho})$ have no $\tau'$-mutable edges pair.
%
%\medskip

In case (2), we may assume $\rho$ crosses $a_1, \tau, a_2$, then up to a difference of relative orientation, $G_{T,\rho}$ and $G_{T',\rho}$ are the following graphs, respectively,

\centerline{\begin{tikzpicture}
\draw[-] (1,0) -- (2,0);
\draw[-] (1,-1) -- (2,-1);
\draw[-] (2,0) -- (3,0);
\draw[-] (2,-1) -- (3,-1);
\draw[-] (3,0) -- (3,-1);
\draw[-] (1,-1) -- (1,-2);
\draw[-] (2,-1) -- (2,-2);
\draw[-] (1,-2) -- (2,-2);
%\draw[-] (1,-1) -- (2,-1);
\draw[dashed] (1,0) -- (2,-1);
\draw[dashed] (2,0) -- (3,-1);
\draw[dashed] (1,-1) -- (2,-2);
\node[right] at (1.4,-0.45) {$\tau$};
\node[right] at (1.4,-1.5) {$a_1$};
\node[above] at (1.5,-0.05) {$a_2$};
\node[left] at (1.05,-0.5) {$a_1$};
\node[left] at (1.05,-1.5) {$a_8$};
\node[below] at (1.5,-0.8) {$a_4$};
\node[below] at (1.5,-1.95) {$a_7$};
\node[right] at (1.75,-0.5) {$a_3$};
\node[right] at (1.95,-1.5) {$\tau$};
\node[right] at (2.4,-0.5) {$a_2$};
\node[above] at (2.5,-0.05) {$a_6$};
%\node[left] at (3,-0.5) {$a_5$};
\node[below] at (2.5,-1) {$\tau$};
\node[right] at (2.95,-0.5) {$a_5$};
%\node[right] at (1.2,-0.5) {$p_{i_1}$};
\draw [-] (1, 0) -- (1,-1);
\draw [-] (2, 0) -- (2,-1);
\draw [fill] (1,0) circle [radius=.05];
\draw [fill] (2,0) circle [radius=.05];
\draw [fill] (1,-1) circle [radius=.05];
\draw [fill] (2,-1) circle [radius=.05];
\draw [fill] (1,-2) circle [radius=.05];
\draw [fill] (2,-2) circle [radius=.05];
\draw [fill] (3,0) circle [radius=.05];
\draw [fill] (3,-1) circle [radius=.05];
\draw[-] (5,0) -- (6,0);
\draw[-] (5,-1) -- (6,-1);
\draw[-] (5,-2) -- (6,-2);
\draw[-] (5,-1) -- (5,-2);
\draw[-] (6,-1) -- (6,-2);
\draw[dashed] (5,-1) -- (6,-2);
\draw[dashed] (5,0) -- (6,-1);
\node[right] at (5.4,-1.5) {$a_1$};
\node[right] at (5.4,-0.5) {$a_2$};
\node[above] at (5.5,-0.05) {$a_5$};
\node[left] at (5.05,-0.5) {$a_1$};
\node[below] at (5.5,-0.7) {$\tau'$};
\node[right] at (5.95,-0.5) {$a_6$};
\node[right] at (5.95,-1.5) {$a_2$};
\node[left] at (5.05,-1.5) {$a_8$};
%\node[below] at (1.5,-0.8) {$a_4$};
\node[below] at (5.5,-1.95) {$a_7$};
%\node[right] at (1.2,-0.5) {$p_{i_1}$};
\draw [-] (5, 0) -- (5,-1);
\draw [-] (6, 0) -- (6,-1);
\draw [fill] (5,0) circle [radius=.05];
\draw [fill] (6,0) circle [radius=.05];
\draw [fill] (5,-1) circle [radius=.05];
\draw [fill] (6,-1) circle [radius=.05];
\draw [fill] (5,-2) circle [radius=.05];
\draw [fill] (6,-2) circle [radius=.05];
\end{tikzpicture}}

In this case, $G_{T,\rho}$ (respectively $G_{T',\rho}$) has one $\tau$-(respectively $\tau'$-)equivalence class of type (II) (respectively (IV)). Thus $n^{\tau}(T,\rho)=n^{\tau'}(T',\rho)$. Moreover, we have a partition bijection from $\mathcal P(G_{T,\rho})$ to $\mathcal P(G_{T',\rho})$, as shown in the following.

\centerline{\includegraphics{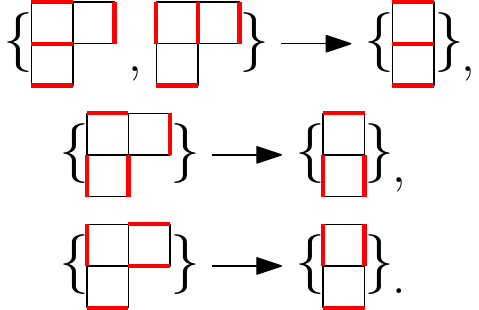}}

As case (1), we can see that the first/last gluing edge of $G_{T,\rho}$ is in the left set if and only if the first/last gluing edge of $G_{T',\rho}$ in the right set.

\medskip
%
%In this case, $P_{\pm}(G_{T,\rho})$ have no $\tau$-mutable edges pair. $P_{\pm}(G_{T',\rho})$ have no $\tau'$-mutable edges pair.
%
%\medskip

In case (3), we may assume $\rho$ crosses $a_2, \tau, \tau', a_4$, then up to a difference of relative orientation, $G_{T,\rho}$ and $G_{T',\rho}$ are the following graphs, respectively,

\centerline{\begin{tikzpicture}
\draw[-] (0,0) -- (1,0);
\draw[-] (0,0) -- (0,-1);
\draw[-] (0,-1) -- (1,-1);
\draw[-] (1,0) -- (2,0);
\draw[-] (1,-1) -- (2,-1);
\draw[-] (2,0) -- (3,0);
\draw[-] (2,-1) -- (3,-1);
\draw[-] (3,0) -- (3,-1);
%\draw[-] (1,-1) -- (1,-2);
%\draw[-] (2,-1) -- (2,-2);
%\draw[-] (1,-2) -- (2,-2);
%\draw[-] (1,-1) -- (2,-1);
\draw[dashed] (1,0) -- (2,-1);
\draw[dashed] (0,0) -- (1,-1);
\draw[dashed] (2,0) -- (3,-1);
%\draw[dashed] (1,-1) -- (2,-2);
\node[right] at (1.4,-0.45) {$\tau$};
\node[above] at (1.5,0) {$a_2$};
\node[above] at (0.5,0) {$\tau$};
\node[left] at (1.35,-0.5) {$a_1$};
\node[below] at (1.5,-1) {$a_4$};
\node[below] at (0.5,-1) {$a_7$};
\node[right] at (1.75,-0.5) {$a_3$};
\node[right] at (2.4,-0.5) {$a_2$};
\node[left] at (0.8,-0.5) {$a_4$};
\node[above] at (2.5,0) {$a_6$};
\node[below] at (2.5,-1) {$\tau$};
\node[right] at (3,-0.5) {$a_5$};
\node[left] at (0,-0.5) {$a_8$};
\draw [-] (1, 0) -- (1,-1);
\draw [-] (2, 0) -- (2,-1);
\draw [fill] (0,0) circle [radius=.05];
\draw [fill] (0,-1) circle [radius=.05];
\draw [fill] (1,0) circle [radius=.05];
\draw [fill] (2,0) circle [radius=.05];
\draw [fill] (1,-1) circle [radius=.05];
\draw [fill] (2,-1) circle [radius=.05];
%\draw [fill] (1,-2) circle [radius=.05];
%\draw [fill] (2,-2) circle [radius=.05];
\draw [fill] (3,0) circle [radius=.05];
\draw [fill] (3,-1) circle [radius=.05];
\draw[-] (5,0) -- (6,0);
\draw[-] (5,-1) -- (6,-1);
\draw [-] (5, 0) -- (5,-1);
\draw [-] (6, 0) -- (6,-1);
\draw[-] (5,0) -- (5,1);
\draw[-] (5,-1) -- (5,-2);
\draw[-] (6,0) -- (6,1);
\draw[-] (6,-1) -- (6,-2);
\draw[-] (5,1) -- (6,1);
\draw[-] (5,-2) -- (6,-2);
\draw[dashed] (5,0) -- (6,-1);
\draw[dashed] (5,1) -- (6,0);
\draw[dashed] (5,-1) -- (6,-2);
\node[left] at (5.9,0.5) {$a_2$};
\node[above] at (5.5,-0.25) {$a_1$};
\node[right] at (6,-0.5) {$a_2$};
\node[right] at (6,0.5) {$a_5$};
\node[right] at (6,-1.5) {$\tau'$};
\node[right] at (5.3,-1.5) {$a_4$};
\node[below] at (5.5,-0.8) {$a_3$};
\node[left] at (5.9,-0.4) {$\tau'$};
\node[left] at (5,-0.5) {$a_4$};
\node[left] at (5,-1.5) {$a_8$};
\node[left] at (5,0.5) {$\tau'$};
\node[below] at (5.5,-1.95) {$a_7$};
\node[above] at (5.5,0.95) {$a_6$};
\draw [fill] (5,0) circle [radius=.05];
\draw [fill] (6,0) circle [radius=.05];
\draw [fill] (5,-1) circle [radius=.05];
\draw [fill] (6,-1) circle [radius=.05];
\draw [fill] (6,-2) circle [radius=.05];
\draw [fill] (5,-2) circle [radius=.05];
\draw [fill] (5,1) circle [radius=.05];
\draw [fill] (6,1) circle [radius=.05];
\end{tikzpicture}}

In this case, $G_{T,\rho}$ (respectively $G_{T',\rho}$) has one $\tau$-(respectively $\tau'$-)equivalence class of type (I). Thus $n^{\tau}(T,\rho)=n^{\tau'}(T',\rho)$. Moreover, we have a partition bijection from $\mathcal P(G_{T,\rho})$ to $\mathcal P(G_{T',\rho})$, as shown in the following figure.

\centerline{\includegraphics{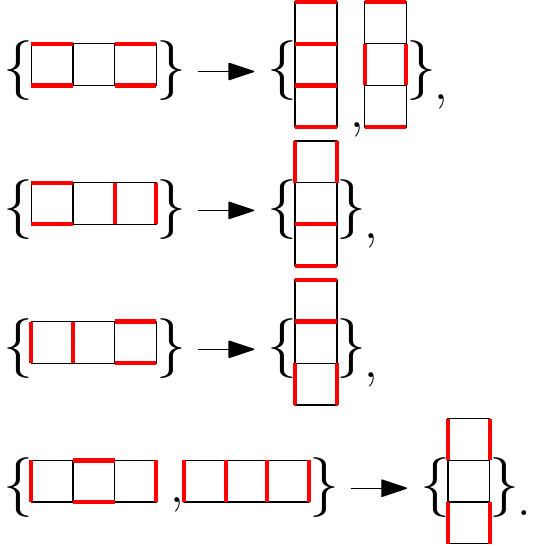}}

As cases (1) and (2), we can see that the first/last gluing edge of $G_{T,\rho}$ is in the left set if and only if the first/last gluing edge of $G_{T',\rho}$ in the right set.
%
%\medskip
%
%In this case, clearly the labels of $\tau$-mutable edges pair in $P_{\pm}(G_{T,\rho})$ and the labels of $\tau'$-mutable edges pair in $P_{\pm}(G_{T',\rho})$ are the same.

\subsection{When $a_1= a_3$ or $a_2= a_4$.}

\begin{Lemma}

If $a_1=a_3$ and $a_2=a_4$, then $\mathcal O$ is a torus with one marked point.

\end{Lemma}

\begin{proof}

Denote the triangles formed by $a_1,a_4,\tau$ and $a_2,a_3,\tau$ by $\Delta_1$ and $\Delta_2$, respectively. If $a_1=a_3$ and $a_2=a_4$, gluing $a_1,a_3$ and gluing $a_2,a_4$, we get the subspace $\mathcal O'$ of $\mathcal O$ formed by the union of triangles $\Delta_1$ and $\Delta_2$ is homeomorphism to the torus, the Klein bottle or the real prjective plan with one marked point, which is a closed surface without boundary. Since $\mathcal O$ is a connected manifold, $\mathcal O=\mathcal O'$. Moreover, since $\mathcal O$ is oriented, it is a torus with one marked point.
\end{proof}

\medskip

In view of the above lemma, we know $a_1=a_3$ and $a_2=a_4$ can not hold at the same time since $\mathcal O$ without punctures. We may assume that $a_1=a_3$, thus $a_2\neq a_4$.

\medskip

Denote by $\Delta_1$ the triangle formed by arcs $\tau,a_1,a_2$ and by $\Delta_2$ the triangle formed by arcs $\tau,a_1,a_4$. Since $\rho$ does not cross itself and $\tau'\neq\gamma\notin T$, consider the universal covering of subspace $\Delta_1\cup \Delta_2$, we have the following possibilities (we illustrate cases (1)-(7) in the figure below, the other cases are dual):

\medskip

\centerline{\includegraphics{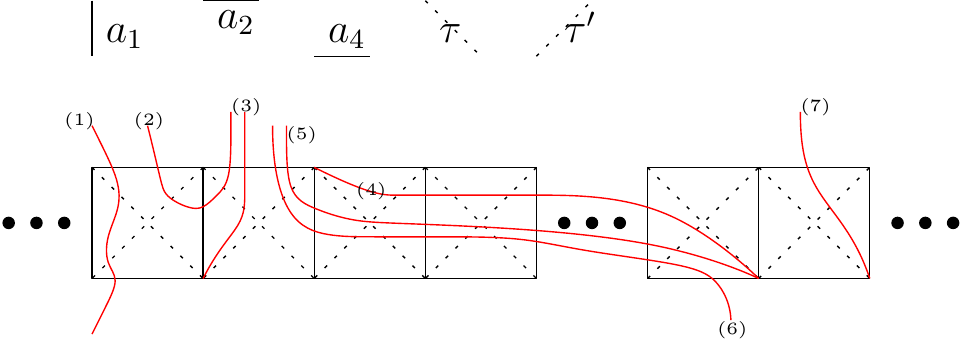}}

\begin{enumerate}[$(1)$]

  \item $\rho$ crosses $a_2,\tau,\tau',a_4$.

  \item $\rho$ crosses $a_i,\tau',a_1,\tau,a_i$ for $i=2,4$.

  \item $\rho$ crosses $a_i,\tau$ for $i=2,4$.

  \item $\rho$ crosses $\tau',a_1,(\tau',\overset{s}{\overbrace{\tau,a_1),\cdots,(\tau',\tau}},a_1),\tau'$ for $s\geq 0$.

  \item $\rho$ crosses $a_i,\tau',a_1,(\tau',\overset{s}{\overbrace{\tau,a_1),\cdots,(\tau',\tau}},a_1),\tau'$ for $i=2,4$ and $s\geq 0$.

  \item $\rho$ crosses $a_2,\tau',a_1,(\tau',\overset{s}{\overbrace{\tau,a_1),\cdots,(\tau',\tau}},a_1),\tau',a_4$ for $s\geq 0$.

  \item $\rho$ crosses $a_i,\tau'$ for $i=2,4$.

  \item $\rho$ crosses $\tau,a_1,(\tau,\overset{s}{\overbrace{\tau',a_1),\cdots,(\tau,\tau'}},a_1),\tau$ for $s\geq 0$.

  \item $\rho$ crosses $a_i,\tau,a_1,(\tau,\overset{s}{\overbrace{\tau',a_1),\cdots,(\tau,\tau'}},a_1),\tau$ for $i=2,4$ and $s\geq 0$.

  \item $\rho$ crosses $a_2,\tau,a_1,(\tau,\overset{s}{\overbrace{\tau',a_1),\cdots,(\tau,\tau'}},a_1),\tau,a_4$ for $s\geq 0$.

\end{enumerate}

\medskip

Since cases (3) and (7), (4) and (8), (5) and (9), (6) and (10) are dual respectively, we shall only discuss cases (1)-(6).

\medskip

In case (1), up to a difference of relative orientation, $G_{T,\rho}$ and $G_{T',\rho}$ are the following graphs, respectively,

\centerline{\begin{tikzpicture}
\draw[-] (0,0) -- (1,0);
\draw[-] (0,0) -- (0,-1);
\draw[-] (0,-1) -- (1,-1);
\draw[-] (1,0) -- (2,0);
\draw[-] (1,-1) -- (2,-1);
\draw[-] (2,0) -- (3,0);
\draw[-] (2,-1) -- (3,-1);
\draw[-] (3,0) -- (3,-1);
%\draw[-] (1,-1) -- (1,-2);
%\draw[-] (2,-1) -- (2,-2);
%\draw[-] (1,-2) -- (2,-2);
%\draw[-] (1,-1) -- (2,-1);
\draw[dashed] (1,0) -- (2,-1);
\draw[dashed] (0,0) -- (1,-1);
\draw[dashed] (2,0) -- (3,-1);
%\draw[dashed] (1,-1) -- (2,-2);
\node[right] at (1.4,-0.45) {$\tau$};
\node[above] at (1.5,0) {$a_2$};
\node[above] at (0.5,0) {$\tau$};
\node[left] at (1.35,-0.5) {$a_1$};
\node[below] at (1.5,-1) {$a_4$};
\node[below] at (0.5,-1) {$a_7$};
\node[right] at (1.75,-0.5) {$a_1$};
\node[right] at (2.4,-0.5) {$a_2$};
\node[left] at (0.8,-0.5) {$a_4$};
\node[above] at (2.5,0) {$a_6$};
\node[below] at (2.5,-1) {$\tau$};
\node[right] at (3,-0.5) {$a_5$};
\node[left] at (0,-0.5) {$a_8$};
\draw [-] (1, 0) -- (1,-1);
\draw [-] (2, 0) -- (2,-1);
\draw [fill] (0,0) circle [radius=.05];
\draw [fill] (0,-1) circle [radius=.05];
\draw [fill] (1,0) circle [radius=.05];
\draw [fill] (2,0) circle [radius=.05];
\draw [fill] (1,-1) circle [radius=.05];
\draw [fill] (2,-1) circle [radius=.05];
%\draw [fill] (1,-2) circle [radius=.05];
%\draw [fill] (2,-2) circle [radius=.05];
\draw [fill] (3,0) circle [radius=.05];
\draw [fill] (3,-1) circle [radius=.05];
\draw[-] (5,0) -- (6,0);
\draw[-] (5,-1) -- (6,-1);
\draw [-] (5, 0) -- (5,-1);
\draw [-] (6, 0) -- (6,-1);
\draw[-] (5,0) -- (5,1);
\draw[-] (5,-1) -- (5,-2);
\draw[-] (6,0) -- (6,1);
\draw[-] (6,-1) -- (6,-2);
\draw[-] (5,1) -- (6,1);
\draw[-] (5,-2) -- (6,-2);
\draw[dashed] (5,0) -- (6,-1);
\draw[dashed] (5,1) -- (6,0);
\draw[dashed] (5,-1) -- (6,-2);
\node[left] at (5.9,0.5) {$a_2$};
\node[above] at (5.5,-0.25) {$a_1$};
\node[right] at (6,-0.5) {$a_2$};
\node[right] at (6,0.5) {$a_5$};
\node[right] at (6,-1.5) {$\tau'$};
\node[right] at (5.3,-1.5) {$a_4$};
\node[below] at (5.5,-0.8) {$a_1$};
\node[left] at (5.9,-0.4) {$\tau'$};
\node[left] at (5,-0.5) {$a_4$};
\node[left] at (5,-1.5) {$a_8$};
\node[left] at (5,0.5) {$\tau'$};
\node[below] at (5.5,-1.95) {$a_7$};
\node[above] at (5.5,0.95) {$a_6$};
\draw [fill] (5,0) circle [radius=.05];
\draw [fill] (6,0) circle [radius=.05];
\draw [fill] (5,-1) circle [radius=.05];
\draw [fill] (6,-1) circle [radius=.05];
\draw [fill] (6,-2) circle [radius=.05];
\draw [fill] (5,-2) circle [radius=.05];
\draw [fill] (5,1) circle [radius=.05];
\draw [fill] (6,1) circle [radius=.05];
\end{tikzpicture}}

In this case, $G_{T,\rho}$ (respectively $G_{T',\rho}$) has one $\tau$-(respectively $\tau'$-)equivalence class of type (I). Thus $n^{\tau}(T,\rho)=n^{\tau'}(T',\rho)$. Moreover, we have a partition bijection from $\mathcal P(G_{T,\rho})$ to $\mathcal P(G_{T',\rho})$, as shown in the following figure.

\centerline{\includegraphics{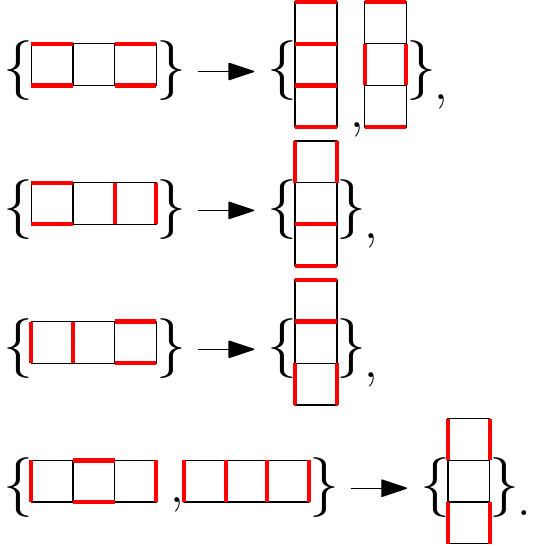}}

One can  see that the first/last gluing edge of $G_{T,\rho}$ is in the left set if and only if the first/last gluing edge of $G_{T',\rho}$ in the right set.

\medskip
%
%In this case, clearly the labels of $\tau$-mutable edges pair in $P_{\pm}(G_{T,\rho})$ and the labels of $\tau'$-mutable edges pair in $P_{\pm}(G_{T',\rho})$ are the same.
%
%\medskip

In case (2), we may assume $i=2$, then up to a difference of relative orientation, $G_{T,\rho}$ and $G_{T',\rho}$ are the following graphs, respectively,

\medskip

\centerline{\begin{tikzpicture}
\draw[-] (1,2) -- (2,2);
\draw[-] (1,3) -- (2,3);
\draw[-] (2,2) -- (2,3);
\draw[-] (3,2) -- (3,3);
\draw[-] (2,2) -- (3,2);
\draw[-] (2,4) -- (2,3);
\draw[-] (3,4) -- (3,3);
\draw[-] (2,4) -- (3,4);
\draw[-] (2,3) -- (3,3);
\draw[-] (4,4) -- (4,3);
\draw[-] (4,4) -- (3,4);
\draw[-] (4,3) -- (3,3);
\draw[dashed] (1,3) -- (2,2);
\draw[dashed] (2,3) -- (3,2);
\draw[dashed] (2,4) -- (3,3);
\draw[dashed] (3,4) -- (4,3);
\node[right] at (1.2,2.55) {$a_2$};
\node[above] at (1.5,2.95) {$a_1$};
\node[left] at (1.05,2.5) {$a_6$};
\node[below] at (1.5,2.05) {$a_5$};
\node[right] at (2.3,2.5) {$a_1$};
\node[right] at (2.3,3.5) {$\tau$};
\node[right] at (3.3,3.5) {$a_2$};
\node[above] at (2.5,2.8) {$a_4$};
\node[below] at (3.5,3.05) {$\tau$};
\node[above] at (2.5,3.95) {$a_2$};
\node[above] at (3.5,3.95) {$a_6$};
\node[right] at (3.95,3.5) {$a_5$};
\node[below] at (2.5,2.05) {$a_2$};
\node[left] at (2.2,2.5) {$\tau$};
\node[left] at (2.05,3.5) {$a_1$};
\node[right] at (2.95,2.5) {$\tau$};
\node[right] at (2.8,3.5) {$a_1$};
\draw [-] (1, 2) -- (1,3);
\draw [-] (2, 2) -- (2,3);
\draw [fill] (1,2) circle [radius=.05];
\draw [fill] (1,3) circle [radius=.05];
\draw [fill] (3,2) circle [radius=.05];
\draw [fill] (3,3) circle [radius=.05];
\draw [fill] (4,4) circle [radius=.05];
\draw [fill] (4,3) circle [radius=.05];
\draw [fill] (3,4) circle [radius=.05];
\draw [fill] (2,2) circle [radius=.05];
\draw [fill] (2,3) circle [radius=.05];
\draw [fill] (2,4) circle [radius=.05];
\draw[-] (6,2) -- (7,2);
\draw[-] (6,3) -- (7,3);
\draw[-] (7,2) -- (7,3);
\draw[-] (8,2) -- (8,3);
\draw[-] (7,2) -- (8,2);
\draw[-] (7,4) -- (7,3);
\draw[-] (8,4) -- (8,3);
\draw[-] (7,4) -- (8,4);
\draw[-] (7,3) -- (8,3);
\draw[-] (9,4) -- (9,3);
\draw[-] (9,4) -- (8,4);
\draw[-] (9,3) -- (8,3);
\draw[dashed] (6,3) -- (7,2);
\draw[dashed] (7,3) -- (8,2);
\draw[dashed] (7,4) -- (8,3);
\draw[dashed] (8,4) -- (9,3);
\node[right] at (6.2,2.5) {$a_2$};
\node[above] at (6.5,2.95) {$\tau'$};
\node[left] at (6.05,2.5) {$a_6$};
\node[below] at (6.5,2.05) {$a_5$};
\node[right] at (6.75,2.5) {$a_1$};
\node[right] at (7.3,2.5) {$\tau'$};
\node[right] at (7.3,3.5) {$a_1$};
\node[right] at (8.3,3.5) {$a_2$};
\node[above] at (7.5,2.8) {$a_4$};
\node[below] at (8.5,3.05) {$a_1$};
\node[above] at (7.5,3.95) {$a_2$};
\node[above] at (8.5,3.95) {$a_6$};
\node[right] at (8.95,3.5) {$a_5$};
\node[below] at (7.5,2.05) {$a_2$};
\node[left] at (7.05,3.5) {$\tau'$};
\node[right] at (7.95,2.5) {$a_1$};
\node[right] at (7.8,3.5) {$\tau'$};
\draw [-] (6, 3) -- (6,2);
\draw [fill] (6,2) circle [radius=.05];
\draw [fill] (7,2) circle [radius=.05];
\draw [fill] (6,3) circle [radius=.05];
\draw [fill] (7,3) circle [radius=.05];
\draw [fill] (8,2) circle [radius=.05];
\draw [fill] (8,3) circle [radius=.05];
\draw [fill] (9,4) circle [radius=.05];
\draw [fill] (9,3) circle [radius=.05];
\draw [fill] (8,4) circle [radius=.05];
\draw [fill] (7,2) circle [radius=.05];
\draw [fill] (7,3) circle [radius=.05];
\draw [fill] (7,4) circle [radius=.05];
\end{tikzpicture}}

In this case, $G_{T,\rho}$ (respectively $G_{T',\rho}$) has one $\tau$-(respectively $\tau'$-)equivalence class of type (II), (IV) (respectively (IV), (II)). Thus $n^{\tau}(T,\rho)=n^{\tau'}(T',\rho)$. Moreover, we have a partition bijection from $\mathcal P(G_{T,\rho})$ to $\mathcal P(G_{T',\rho})$, as shown in the following figure.

\centerline{\includegraphics{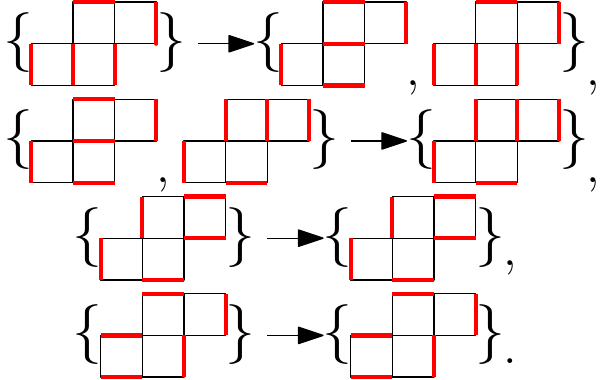}}

We can see that the first/last gluing edge of $G_{T,\rho}$ is in the left set if and only if the first/last gluing edge of $G_{T',\rho}$ in the right set.

\medskip
%
%In this case, $P_{\pm}(G_{T,\rho})$ have no $\tau$-mutable edges pair. $P_{\pm}(G_{T',\rho})$ have no $\tau'$-mutable edges pair.
%
%\medskip

In case (3), $\rho$ and $\gamma$ have the same starting point. We may assume $\rho$ crosses $\tau, a_2$, then up to a difference of relative orientation, $G_{T,\rho}$ and $G_{T',\rho}$ are the following graphs, respectively,

\centerline{\begin{tikzpicture}
\draw[-] (1,0) -- (2,0);
\draw[-] (1,-1) -- (2,-1);
\draw[-] (2,0) -- (3,0);
\draw[-] (2,-1) -- (3,-1);
\draw[-] (3,0) -- (3,-1);
\draw[dashed] (1,0) -- (2,-1);
\draw[dashed] (2,0) -- (3,-1);
\node[right] at (1.4,-0.45) {$\tau$};
\node[above] at (1.5,-0.05) {$a_2$};
\node[left] at (1,-0.5) {$a_1$};
\node[below] at (1.5,-0.95) {$a_4$};
\node[right] at (1.75,-0.5) {$a_1$};
\node[right] at (2.4,-0.5) {$a_2$};
\node[above] at (2.5,-0.05) {$a_6$};
%\node[left] at (3,-0.5) {$a_5$};
\node[below] at (2.5,-0.95) {$\tau$};
\node[right] at (3,-0.5) {$a_5$};
%\node[right] at (1.2,-0.5) {$p_{i_1}$};
\draw [-] (1, 0) -- (1,-1);
\draw [-] (2, 0) -- (2,-1);
\draw [fill] (1,0) circle [radius=.05];
\draw [fill] (2,0) circle [radius=.05];
\draw [fill] (1,-1) circle [radius=.05];
\draw [fill] (2,-1) circle [radius=.05];
\draw [fill] (3,0) circle [radius=.05];
\draw [fill] (3,-1) circle [radius=.05];
\draw[-] (5,0) -- (6,0);
\draw[-] (5,-1) -- (6,-1);
\draw[dashed] (5,0) -- (6,-1);
\node[right] at (5.4,-0.5) {$a_2$};
\node[above] at (5.5,-0.05) {$a_5$};
\node[left] at (5,-0.5) {$a_1$};
\node[below] at (5.5,-0.95) {$\tau'$};
\node[right] at (6,-0.5) {$a_6$};
%\node[right] at (1.2,-0.5) {$p_{i_1}$};
\draw [-] (5, 0) -- (5,-1);
\draw [-] (6, 0) -- (6,-1);
\draw [fill] (5,0) circle [radius=.05];
\draw [fill] (6,0) circle [radius=.05];
\draw [fill] (5,-1) circle [radius=.05];
\draw [fill] (6,-1) circle [radius=.05];
\end{tikzpicture}}

In this case, $G_{T,\rho}$ (respectively $G_{T',\rho}$) has one $\tau$-(respectively $\tau'$-)equivalence class of type (III) (respectively (IV)). Thus $n^{\tau}(T,\rho)=n^{\tau'}(T',\rho)$. Moreover, we have a partition bijection from $\mathcal P(G_{T,\rho})$ to $\mathcal P(G_{T',\rho})$, as shown in the following.

\centerline{\includegraphics{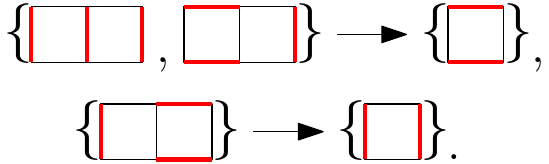}}

We can see that the last gluing edge of $G_{T,\rho}$ is in the left set if and only if the last gluing edge of $G_{T',\rho}$ in the right set. Since $\rho$ and $\gamma$ have the same starting point, $G_{T,\rho}$ and $G_{T',\rho}$ do not have first gluing edges.

\medskip

%In this case, $P_{\pm}(G_{T',\rho})$ have no $\tau'$-mutable edges pair.
%
%\medskip

To consider the remaining cases, we should introduce more notations. If both $a$ and $b$ are edges of two different triangles in $T$, for any non-negative integer $s$, we define the graph $H_s(a,b)$ to be $G_{T,\varsigma}$, where $\varsigma$ is the curve in $\mathcal O$ crossing $a$ $s+1$ times and $b$ $s$ times alternately. See the graph below.

\centerline{\begin{tikzpicture}
\draw[-] (0,0) -- (1,0);
\draw[-] (0,0) -- (0,-1);
\draw[-] (0,-1) -- (1,-1);
\draw[-] (1,0) -- (2,0);
\draw[-] (1,-1) -- (2,-1);
\draw[-] (2,0) -- (3,0);
\draw[-] (2,-1) -- (3,-1);
\draw[-] (3,0) -- (3,-1);
\draw[-] (3,0) -- (4,0);
\draw[-] (3,-1) -- (4,-1);
\draw[-] (4,0) -- (4,-1);
\draw[-] (5,0) -- (6,0);
\draw[-] (5,0) -- (5,-1);
\draw[-] (6,0) -- (6,-1);
\draw[-] (5,-1) -- (6,-1);
\draw[-] (6,0) -- (7,0);
\draw[-] (6,-1) -- (7,-1);
\draw[-] (7,0) -- (7,-1);
\draw[-] (7,0) -- (8,0);
\draw[-] (7,-1) -- (8,-1);
\draw[-] (8,0) -- (8,-1);
\draw[dashed] (1,0) -- (2,-1);
\draw[dashed] (0,0) -- (1,-1);
\draw[dashed] (2,0) -- (3,-1);
\draw[dashed] (3,0) -- (4,-1);
\draw[dashed] (5,0) -- (6,-1);
\draw[dashed] (6,0) -- (7,-1);
\draw[dashed] (7,0) -- (8,-1);
\node[above] at (0.5,-0.05) {$b$};
\node[above] at (1.5,-0.05) {$a$};
\node[above] at (2.5,-0.05) {$b$};
\node[above] at (3.5,-0.05) {$a$};
\node[above] at (5.5,-0.05) {$b$};
\node[above] at (6.5,-0.05) {$a$};
\node[above] at (7.5,-0.05) {$b$};
\node[below] at (0.5,-0.95) {$b$};
\node[below] at (1.5,-0.95) {$a$};
\node[below] at (2.5,-0.95) {$b$};
\node[below] at (3.5,-0.95) {$a$};
\node[below] at (5.5,-0.95) {$b$};
\node[below] at (6.5,-0.95) {$a$};
\node[below] at (7.5,-0.95) {$b$};
\node[right] at (1.4,-0.45) {$b$};
\node[right] at (3.4,-0.45) {$b$};
\node[right] at (2.4,-0.5) {$a$};
\node[right] at (6.4,-0.45) {$b$};
\node[right] at (5.4,-0.5) {$a$};
\node[right] at (7.4,-0.5) {$a$};
\node[left] at (0.8,-0.5) {$a$};
\draw [-] (1, 0) -- (1,-1);
\draw [-] (2, 0) -- (2,-1);
\draw [fill] (0,0) circle [radius=.05];
\draw [fill] (0,-1) circle [radius=.05];
\draw [fill] (1,0) circle [radius=.05];
\draw [fill] (2,0) circle [radius=.05];
\draw [fill] (1,-1) circle [radius=.05];
\draw [fill] (2,-1) circle [radius=.05];
\draw [fill] (3,0) circle [radius=.05];
\draw [fill] (3,-1) circle [radius=.05];
\draw [fill] (4,0) circle [radius=.05];
\draw [fill] (4,-1) circle [radius=.05];
\draw [fill] (5,0) circle [radius=.05];
\draw [fill] (5,-1) circle [radius=.05];
\draw [fill] (6,0) circle [radius=.05];
\draw [fill] (6,-1) circle [radius=.05];
\draw [fill] (7,0) circle [radius=.05];
\draw [fill] (8,0) circle [radius=.05];
\draw [fill] (8,-1) circle [radius=.05];
\draw [fill] (4.3,-0.5) circle [radius=.02];
\draw [fill] (4.5,-0.5) circle [radius=.02];
\draw [fill] (4.7,-0.5) circle [radius=.02];
\draw [fill] (7,-1) circle [radius=.05];
\end{tikzpicture}}

%Denote the $2$-nd tile of $G_s(a,b)$ by $G_1$, $4$-th tile by $G_2$, $\cdots$, $2s$-th tile by $G_s$. Given a $0$-$1$ sequence $(\lambda_1,\cdots,\lambda_s)\in \{0,1\}^{s}$, we set $\mathcal G_{(\lambda_1,\cdots,\lambda_s)}(a,b)$ to be the set of perfect matching of $G_s(a,b)$ such that each perfect matching $P$ in it satisfies the following conditions: if $\lambda_i=1$, then the two edges of $G_i$ labeled $a$ are both in $P$; if $\lambda_i=0$, then the two edges of $G_i$ labeled $a$ are both not in $P$. See Figure??? for an example with $s=2$ and $(\lambda_1,\lambda_2)=(0,1)$.
%
%\medskip

Denote the $i$-th tile of $H_s(a,b)$ by $H_i$. For a sequence $(\lambda_1,\cdots,\lambda_s)\in \{0,1\}^{s}$, let $\mathcal H_{(\lambda_1,\cdots,\lambda_s)}(a,b)$ contains all perfect matching $P$ of $H_s(a,b)$ satisfying the following conditions: if $\lambda_i=1$, then the two edges of $H_{2i}$ labeled $a$ are both in $P$; if $\lambda_i=0$, then the two edges of $H_{2i}$ labeled $a$ are both not in $P$; for a sequence $(\lambda_1,\cdots,\lambda_{s+1})\in \{0,1\}^{s+1}$, let $\mathcal H'_{(\lambda_1,\cdots,\lambda_{s+1})}(a,b)$ contains all the perfect matching $P$ of $H_s(a,b)$ satisfying the following conditions: if $\lambda_i=1$, then the two edges of $H_{2i-1}$ labeled $b$ are both in $P$; if $\lambda_i=0$, then the two edges of $H_{2i-1}$ labeled $b$ are both not in $P$. See the following figure for example $\mathcal H_{(1,0)}(a,b)$ and $\mathcal H'_{(1,0,1)}(a,b)$.

\centerline{\includegraphics{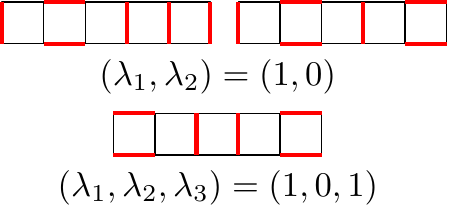}}

We have the easy observations.

\medskip

\begin{Lemma}\label{basicdecom}

\begin{enumerate}[$(1)$]

  %\item $\mathcal P(G_s(a,b))=\bigsqcup_{(\lambda_1,\cdots,\lambda_s)\in \{0,1\}^{s}}\mathcal G_{(\lambda_1,\cdots,\lambda_s)}(a,b)$.

  %\item $\mathcal G_{(\lambda_1,\cdots,\lambda_s)}(a,b) \subset \mathcal P^{b}_{(\lambda_1,\lambda_1+\lambda_2-1,\cdots,\lambda_{s-1}+\lambda_s-1,\lambda_s)}(G_s(a,b))$.

  \item $\mathcal H_{(\lambda_1,\cdots,\lambda_s)}(a,b) = \mathcal P^{a}_{(\lambda_1-1,\lambda_1+\lambda_2-1,\cdots,\lambda_{s-1}+\lambda_s-1,\lambda_s-1)}(H_s(a,b))$.

  \item $\mathcal H'_{(\lambda_1,\cdots,\lambda_{s+1})}(a,b) = \mathcal P^{b}_{(\lambda_1,\lambda_1+\lambda_2-1,\cdots,\lambda_{s}+\lambda_{s+1}-1,\lambda_{s+1})}(H_s(a,b))$.

  \item  \[\begin{array}{ccl} \mathcal P(H_s(a,b))
  & = & \bigsqcup_{(\lambda_1,\cdots,\lambda_s)\in \{0,1\}^{s}}\mathcal H_{(\lambda_1,\cdots,\lambda_s)}(a,b)\\
  & = & \bigsqcup_{(\lambda_1,\cdots,\lambda_{s+1})\in \{0,1\}^{s+1}}\mathcal H'_{(\lambda_1,\cdots,\lambda_{s+1})}(a,b).
  \end{array}\]

\end{enumerate}

\end{Lemma}

\begin{proof}

(1) and (2) follow by definition. (3) follows immediately by Lemma \ref{inone}.
%We shall only prove (1). It is equivalent to say, for each $i$ satisfies $1\leq i\leq s$ and a perfect matching $P$, the edges of $G_i$ labeled $a$ are either in $P$ or not in $P$ at the same time. Otherwise, choose the maximal $i$ such that one edge of $G_i$ labeled $a$ in $P$ and the other one not in $P$ for some $P$. Clearly $i\neq s$. One can see immediately from the definition of perfect matching, one edge of $G_{i+1}$ labeled $a$ in $P$ and the other one not in $P$ for some $P$. This contracts to the maximality of $i$.
\end{proof}

\medskip

\begin{Lemma}\label{diff}
\begin{enumerate}[$(1)$]

  \item If $P\neq Q\in \mathcal H_{(\lambda_1,\cdots,\lambda_s)}(a,b)$, then there exists $i\in [1,s+1]$ such that $P$ and $Q$ can twist on the tile $H_{2i-1}$ and $P\cap E(H_{2i-1})\neq Q\cap E(H_{2i-1})$.

  \item If $P'\neq Q'\in \mathcal H'_{(\lambda_1,\cdots,\lambda_{s+1})}(a,b)$, then there exists $i\in [1,s]$ such that $P'$ and $Q'$ can twist on the tile $H_{2i}$ and $P'\cap E(H_{2i})\neq Q'\cap E(H_{2i})$.

\end{enumerate}

\end{Lemma}

\begin{proof}

We shall only prove (1) because (2) can be proved similarly. If we assume $\lambda_0=\lambda_{s+1}=0$ in convention, according to the definition of $\mathcal H_{(\lambda_1,\cdots,\lambda_s)}(a,b)$, $P$ and $Q$ can twist on tiles $H_{2i-1}$ for $i\in [1,s+1]$ satisfying $\lambda_{i-1}=\lambda_i=0$, the other edges in $P$ and $Q$ are the same. Therefore our result follows at once.
\end{proof}

\medskip

Now we turn to the comparison of the perfect matchings of $G_{T,\rho}$ and $G_{T',\rho}$.

\medskip

In case (4), $\rho=\gamma$. Up to a difference of relative orientation, $G_{T,\rho}$ and $G_{T',\rho}$ are the following graphs, respectively,

\centerline{\begin{tikzpicture}
\draw[-] (0,0) -- (1,0);
\draw[-] (0,0) -- (0,-1);
\draw[-] (0,-1) -- (1,-1);
\draw[-] (1,0) -- (2,0);
\draw[-] (1,-1) -- (2,-1);
\draw[-] (2,0) -- (3,0);
\draw[-] (2,-1) -- (3,-1);
\draw[-] (3,0) -- (3,-1);
\draw[-] (3,0) -- (4,0);
\draw[-] (3,-1) -- (4,-1);
\draw[-] (4,0) -- (4,-1);
\draw[-] (5,0) -- (6,0);
\draw[-] (5,0) -- (5,-1);
\draw[-] (6,0) -- (6,-1);
\draw[-] (5,-1) -- (6,-1);
\draw[-] (6,0) -- (7,0);
\draw[-] (6,-1) -- (7,-1);
\draw[-] (7,0) -- (7,-1);
\draw[-] (7,0) -- (8,0);
\draw[-] (7,-1) -- (8,-1);
\draw[-] (8,0) -- (8,-1);
\draw[dashed] (1,0) -- (2,-1);
\draw[dashed] (0,0) -- (1,-1);
\draw[dashed] (2,0) -- (3,-1);
\draw[dashed] (3,0) -- (4,-1);
\draw[dashed] (5,0) -- (6,-1);
\draw[dashed] (6,0) -- (7,-1);
\draw[dashed] (7,0) -- (8,-1);
\node[left] at (0.05,-0.5) {$a_2$};
\node[left] at (2.25,-0.5) {$a_2$};
\node[left] at (4.25,-0.5) {$a_2$};
\node[left] at (5.25,-0.5) {$a_2$};
\node[left] at (7.25,-0.5) {$a_2$};
\node[left] at (1.25,-0.5) {$a_4$};
\node[left] at (3.25,-0.5) {$a_4$};
\node[left] at (6.25,-0.5) {$a_4$};
\node[right] at (7.95,-0.5) {$a_4$};
\node[above] at (0.5,-0.05) {$\tau$};
\node[above] at (1.5,-0.05) {$a_1$};
\node[above] at (2.5,-0.05) {$\tau$};
\node[above] at (3.5,-0.05) {$a_1$};
\node[above] at (5.5,-0.05) {$\tau$};
\node[above] at (6.5,-0.05) {$a_1$};
\node[above] at (7.5,-0.05) {$\tau$};
\node[below] at (0.5,-0.95) {$\tau$};
\node[below] at (1.5,-0.95) {$a_1$};
\node[below] at (2.5,-0.95) {$\tau$};
\node[below] at (3.5,-0.95) {$a_1$};
\node[below] at (5.5,-0.95) {$\tau$};
\node[below] at (6.5,-0.95) {$a_1$};
\node[below] at (7.5,-0.95) {$\tau$};
\node[right] at (1.3,-0.45) {$\tau$};
\node[right] at (3.3,-0.45) {$\tau$};
\node[right] at (2.3,-0.5) {$a_1$};
\node[right] at (6.3,-0.45) {$\tau$};
\node[right] at (5.3,-0.5) {$a_1$};
\node[right] at (7.3,-0.5) {$a_1$};
\node[left] at (0.8,-0.5) {$a_1$};
\draw [-] (1, 0) -- (1,-1);
\draw [-] (2, 0) -- (2,-1);
\draw [fill] (0,0) circle [radius=.05];
\draw [fill] (0,-1) circle [radius=.05];
\draw [fill] (1,0) circle [radius=.05];
\draw [fill] (2,0) circle [radius=.05];
\draw [fill] (1,-1) circle [radius=.05];
\draw [fill] (2,-1) circle [radius=.05];
\draw [fill] (3,0) circle [radius=.05];
\draw [fill] (3,-1) circle [radius=.05];
\draw [fill] (4,0) circle [radius=.05];
\draw [fill] (4,-1) circle [radius=.05];
\draw [fill] (5,0) circle [radius=.05];
\draw [fill] (5,-1) circle [radius=.05];
\draw [fill] (6,0) circle [radius=.05];
\draw [fill] (6,-1) circle [radius=.05];
\draw [fill] (7,0) circle [radius=.05];
\draw [fill] (8,0) circle [radius=.05];
\draw [fill] (8,-1) circle [radius=.05];
\draw [fill] (4.3,-0.5) circle [radius=.02];
\draw [fill] (4.5,-0.5) circle [radius=.02];
\draw [fill] (4.7,-0.5) circle [radius=.02];
\draw [fill] (7,-1) circle [radius=.05];
\end{tikzpicture}}

\medskip

\centerline{\begin{tikzpicture}
\draw[-] (-1,0) -- (0,0);
\draw[-] (-1,-1) -- (0,-1);
\draw[-] (-1,0) -- (-1,-1);
\draw[-] (0,0) -- (1,0);
\draw[-] (0,0) -- (0,-1);
\draw[-] (0,-1) -- (1,-1);
\draw[-] (1,0) -- (2,0);
\draw[-] (1,-1) -- (2,-1);
\draw[-] (2,0) -- (3,0);
\draw[-] (2,-1) -- (3,-1);
\draw[-] (3,0) -- (3,-1);
\draw[-] (3,0) -- (4,0);
\draw[-] (3,-1) -- (4,-1);
\draw[-] (4,0) -- (4,-1);
\draw[-] (5,0) -- (6,0);
\draw[-] (5,0) -- (5,-1);
\draw[-] (6,0) -- (6,-1);
\draw[-] (5,-1) -- (6,-1);
\draw[-] (6,0) -- (7,0);
\draw[-] (6,-1) -- (7,-1);
\draw[-] (7,0) -- (7,-1);
\draw[-] (7,0) -- (8,0);
\draw[-] (7,-1) -- (8,-1);
\draw[-] (8,0) -- (8,-1);
\draw[-] (8,0) -- (9,0);
\draw[-] (8,-1) -- (9,-1);
\draw[-] (9,0) -- (9,-1);
\draw[dashed] (-1,0) -- (0,-1);
\draw[dashed] (1,0) -- (2,-1);
\draw[dashed] (0,0) -- (1,-1);
\draw[dashed] (2,0) -- (3,-1);
\draw[dashed] (3,0) -- (4,-1);
\draw[dashed] (5,0) -- (6,-1);
\draw[dashed] (6,0) -- (7,-1);
\draw[dashed] (7,0) -- (8,-1);
\draw[dashed] (8,0) -- (9,-1);
\node[left] at (-0.95,-0.5) {$a_2$};
\node[left] at (2.25,-0.5) {$a_4$};
\node[left] at (4.25,-0.5) {$a_4$};
\node[left] at (5.25,-0.5) {$a_4$};
\node[left] at (7.25,-0.5) {$a_4$};
\node[left] at (1.25,-0.5) {$a_2$};
\node[left] at (0.25,-0.5) {$a_4$};
\node[left] at (3.25,-0.5) {$a_2$};
\node[left] at (6.25,-0.5) {$a_2$};
\node[left] at (8.25,-0.5) {$a_2$};
\node[right] at (8.95,-0.5) {$a_4$};
\node[above] at (0.5,-0.05) {$\tau'$};
\node[above] at (1.5,-0.05) {$a_1$};
\node[above] at (-0.5,-0.05) {$a_1$};
\node[above] at (2.5,-0.05) {$\tau'$};
\node[above] at (3.5,-0.05) {$a_1$};
\node[above] at (5.5,-0.05) {$\tau'$};
\node[above] at (6.5,-0.05) {$a_1$};
\node[above] at (8.5,-0.05) {$a_1$};
\node[above] at (7.5,-0.05) {$\tau'$};
\node[below] at (0.5,-0.95) {$\tau'$};
\node[below] at (1.5,-0.95) {$a_1$};
\node[below] at (-0.5,-0.95) {$a_1$};
\node[below] at (2.5,-0.95) {$\tau'$};
\node[below] at (3.5,-0.95) {$a_1$};
\node[below] at (5.5,-0.95) {$\tau'$};
\node[below] at (6.5,-0.95) {$a_1$};
\node[below] at (8.5,-0.95) {$a_1$};
\node[below] at (7.5,-0.95) {$\tau'$};
\node[right] at (1.3,-0.45) {$\tau'$};
\node[right] at (-0.7,-0.45) {$\tau'$};
\node[right] at (3.3,-0.45) {$\tau'$};
\node[right] at (2.3,-0.5) {$a_1$};
\node[right] at (6.3,-0.45) {$\tau'$};
\node[right] at (8.3,-0.45) {$\tau'$};
\node[right] at (5.3,-0.5) {$a_1$};
\node[right] at (7.3,-0.5) {$a_1$};
\node[left] at (0.8,-0.5) {$a_1$};
\draw [-] (1, 0) -- (1,-1);
\draw [-] (2, 0) -- (2,-1);
\draw [fill] (0,0) circle [radius=.05];
\draw [fill] (0,-1) circle [radius=.05];
\draw [fill] (-1,0) circle [radius=.05];
\draw [fill] (-1,-1) circle [radius=.05];
\draw [fill] (1,0) circle [radius=.05];
\draw [fill] (2,0) circle [radius=.05];
\draw [fill] (1,-1) circle [radius=.05];
\draw [fill] (2,-1) circle [radius=.05];
\draw [fill] (3,0) circle [radius=.05];
\draw [fill] (3,-1) circle [radius=.05];
\draw [fill] (4,0) circle [radius=.05];
\draw [fill] (4,-1) circle [radius=.05];
\draw [fill] (5,0) circle [radius=.05];
\draw [fill] (5,-1) circle [radius=.05];
\draw [fill] (6,0) circle [radius=.05];
\draw [fill] (6,-1) circle [radius=.05];
\draw [fill] (7,0) circle [radius=.05];
\draw [fill] (8,0) circle [radius=.05];
\draw [fill] (8,-1) circle [radius=.05];
\draw [fill] (9,0) circle [radius=.05];
\draw [fill] (9,-1) circle [radius=.05];
\draw [fill] (4.3,-0.5) circle [radius=.02];
\draw [fill] (4.5,-0.5) circle [radius=.02];
\draw [fill] (4.7,-0.5) circle [radius=.02];
\draw [fill] (7,-1) circle [radius=.05];
\end{tikzpicture}}

\medskip

Clearly we have $n^{\tau}(T,\rho)=n^{\tau'}(T',\rho)$.

\medskip

\huang{change $s+1$ to $s$}Since $G_{T,\rho}$ is isomorphic to $H_{s+1}(a_1,\tau)$ and $G_{T',\rho}$ is isomorphic to $H_{s+2}(\tau',a_1)$. By Lemma \ref{basicdecom}, as sets, we have $\mathcal P(G_{T,\rho})\cong\bigsqcup_{(\lambda_1,\cdots,\lambda_{s+2})\in \{0,1\}^{s+2}}\mathcal H'_{(\lambda_1,\cdots,\lambda_{s+2})}(a_1,\tau)$, and  $\mathcal P(G_{T',\rho})\cong\bigsqcup_{(\lambda_1,\cdots,\lambda_{s+2})\in \{0,1\}^{s+2}}\mathcal H_{(\lambda_1,\cdots,\lambda_{s+2})}(\tau',a_1)$.

%Under the isomorphisms, we have $$\mathcal H'_{(\lambda_1,\cdots,\lambda_{s+2})}(a_1,\tau)\subset \mathcal P^{\tau}_{(\lambda_1,\lambda_1+\lambda_2-1,\cdots,\lambda_{s+1}+\lambda_{s+2}-1,\lambda_{s+2})}(G,\rho),$$
%$$\mathcal H_{(\lambda_1,\cdots,\lambda_{s+2})}(\tau',a_1)\subset \mathcal P^{\tau'}_{(\lambda_1-1,\lambda_1+\lambda_2-1,\cdots,\lambda_{s+1}+\lambda_{s+2}-1,\lambda_{s+2}-1)}(G,\rho)$$

\medskip

We have the following observation.

\medskip

\begin{Lemma}\label{non-tau-mu1}

Let $P\in \mathcal H'_{(\lambda_1,\cdots,\lambda_{s+2})}(a_1,\tau)$ for some sequence $(\lambda_1,\cdots,\lambda_{s+2})$. We assume $\lambda_{0}=1$ in convention, then for any $i\in [1,s+2]$,

 \begin{enumerate}[$(1)$]

   \item the edge labeled $a_2$ of the $(2i-1)$-th tile is in $P$ but non-$\tau$-mutable if and only if $\lambda_{i-1}=1$ and $\lambda_i=0$;

   \item the edge labeled $a_4$ of the $(2i-1)$-th tile is in $P$ but non-$\tau$-mutable if and only if $\lambda_{i-1}=0$ and $\lambda_{i}=1$.

 \end{enumerate}

\end{Lemma}

\begin{proof}

By the symmetry, we shall only prove (1).

``Only If Part:"
Since the edge labeled $a_2$ of the $(2i-1)$-th tile is in $P$ but non-$\tau$-mutable, the edges labeled $a_1,a_4$ of the $(2i-2)$-th tile and the edges labeled $\tau$ of the $(2i-1)$-th tile are not in $P$. Thus the edges labeled $\tau$ of the $(2i-3)$-th tile are in $P$. Therefore, $\lambda_{i-1}=1$ and $\lambda_i=0$.

``If Part:" Since $\lambda_{i-1}=1$, the edges labeled $\tau$ of the $(2i-3)$-th tile are in $P$, and hence the edges labeled $a_1$ and $a_4$ of the $(2i-2)$-th tile are not in $P$. Since $\lambda_i=0$, the edges labeled $\tau$ of the $(2i-1)$-th tile are not in $P$. Therefore, the edge labeled $a_2$ of the $(2i-1)$-th tile is in $P$ but non-$\tau$-mutable.
\end{proof}

\medskip

Similarly, we have the following result. The proof is similar to the proof of Lemma \ref{non-tau-mu1}, so we omit it.

\medskip

\begin{Lemma}\label{non-tau-mu2}

Let $P\in \mathcal H_{(\lambda_1,\cdots,\lambda_{s+2})}(\tau',a_1)$ for some sequence $(\lambda_1,\cdots,\lambda_{s+2})$. We assume $\lambda_{0}=\lambda_{s+3}=0$ in convention, then for any $i\in [1,s+3]$,

 \begin{enumerate}[$(1)$]

   \item the edge labeled $a_2$ of the $(2i-1)$-th tile is in $P$ but non-$\tau'$-mutable if and only if $\lambda_{i-1}=0$ and $\lambda_i=1$;

   \item the edge labeled $a_4$ of the $(2i-1)$-th tile is in $P$ but non-$\tau'$-mutable if and only if $\lambda_{i}=1$ and $\lambda_{i+1}=0$.

 \end{enumerate}

\end{Lemma}

\medskip

\begin{Lemma}\label{same-num}

Let $P\in \mathcal H'_{(\lambda_1,\cdots,\lambda_{s+2})}(a_1,\tau)$ and $P'\in\mathcal H_{(1-\lambda_1,\cdots,1-\lambda_{s+2})}(\tau',a_1)$. For each edge $\tau \neq a\in T$, the number of the non-$\tau$-mutable edges labeled $a$ in $P$ equals to the number of the non-$\tau'$-mutable edges labeled $a$ in $P'$.

\end{Lemma}

\begin{proof}

By Lemma \ref{inone}, if $a\neq a_2,a_4$, then the two numbers both are equal to $0$. If $a=a_2$, by Lemma \ref{non-tau-mu1} and Lemma \ref{non-tau-mu2}, for each $i\in [1,s+2]$, the edge labeled $a_2$ of the $(2i-1)$-th tile of $G_{T,\rho}$ is non-$\tau$-mutable in $P$ if and only if the edge labeled $a_2$ of the $(2i-1)$-th tile of $G_{T',\rho}$ is non-$\tau'$-mutable in $P'$. By Lemma \ref{non-tau-mu2}, the edge labeled $a_2$ of the $(2(s+3)-1)$-th tile of $G_{T,\rho}$ is not a non-$\tau$-mutable edge in $P$. Therefore, the result holds for $a=a_2$. Similarly, the result holds for $a=a_4$.
\end{proof}

\medskip

Thus, we associate $P\in \mathcal H'_{(\lambda_1,\cdots,\lambda_{s+2})}(a_1,\tau)$ with $\mathcal H_{(1-\lambda_1,\cdots,1-\lambda_{s+2})}(\tau',a_1)$, denote as
$$\psi_{\rho}(P)=\mathcal H_{(1-\lambda_1,\cdots,1-\lambda_{s+2})}(\tau',a_1).$$

On the other hand, for each $P'\in \mathcal H_{(\lambda_1,\cdots,\lambda_{s+2})}(\tau',a_1)$, we associate $P'$ with $$\psi'_{\rho}(P')=\mathcal H'_{(1-\lambda_1,\cdots,1-\lambda_{s+2})}(a_1,\tau).$$

Clearly $P'\in \psi_{\rho}(P)$ if and only if $P\in \psi'_{\rho}(P')$, in this case, by Lemma \ref{same-num} the number of non-$\tau$-mutable edges labeled $a$ in $P$ equals to the number of non-$\tau'$-mutable edges labeled $a$ in $P'$ for any $\tau \neq a\in T$.

\medskip

In this case, since $\rho=\gamma$, $G_{T,\rho}$ and $G_{T',\rho}$ have no first or last gluing edge.
%
% Clearly the labels of $\tau$-mutable edges pair in $P_{\pm}(G_{T,\rho})$ and the labels of $\tau'$-mutable edges pair in $P_{\pm}(G_{T',\rho})$ are the same.

\medskip

In case (5), $\rho$ and $\gamma$ have the same starting point. We may assume $i=4$, then up to a difference of relative orientation, $G_{T,\rho}$ and $G_{T',\rho}$ are the following graphs, respectively,

\centerline{\begin{tikzpicture}
\draw[-] (0,0) -- (1,0);
\draw[-] (0,0) -- (0,-1);
\draw[-] (0,-1) -- (1,-1);
\draw[-] (1,0) -- (2,0);
\draw[-] (1,-1) -- (2,-1);
\draw[-] (2,0) -- (3,0);
\draw[-] (2,-1) -- (3,-1);
\draw[-] (3,0) -- (3,-1);
\draw[-] (3,0) -- (4,0);
\draw[-] (3,-1) -- (4,-1);
\draw[-] (4,0) -- (4,-1);
\draw[-] (5,0) -- (6,0);
\draw[-] (5,0) -- (5,-1);
\draw[-] (6,0) -- (6,-1);
\draw[-] (5,-1) -- (6,-1);
\draw[-] (6,0) -- (7,0);
\draw[-] (6,-1) -- (7,-1);
\draw[-] (7,0) -- (7,-1);
\draw[-] (7,0) -- (8,0);
\draw[-] (7,-1) -- (8,-1);
\draw[-] (8,0) -- (8,-1);
\draw[-] (7,1) -- (8,1);
\draw[-] (7,0) -- (7,1);
\draw[-] (8,0) -- (8,1);
\draw[dashed] (1,0) -- (2,-1);
\draw[dashed] (0,0) -- (1,-1);
\draw[dashed] (2,0) -- (3,-1);
\draw[dashed] (3,0) -- (4,-1);
\draw[dashed] (5,0) -- (6,-1);
\draw[dashed] (6,0) -- (7,-1);
\draw[dashed] (7,0) -- (8,-1);
\draw[dashed] (7,1) -- (8,0);
\node[left] at (0.05,-0.5) {$a_2$};
\node[left] at (2.25,-0.5) {$a_2$};
\node[left] at (4.25,-0.5) {$a_2$};
\node[left] at (5.25,-0.5) {$a_2$};
\node[left] at (7.25,-0.5) {$a_2$};
\node[left] at (1.25,-0.5) {$a_4$};
\node[left] at (3.25,-0.5) {$a_4$};
\node[left] at (6.25,-0.5) {$a_4$};
\node[right] at (7.95,-0.5) {$a_4$};
\node[above] at (0.5,-0.05) {$\tau$};
\node[above] at (1.5,-0.05) {$a_1$};
\node[above] at (2.5,-0.05) {$\tau$};
\node[above] at (3.5,-0.05) {$a_1$};
\node[above] at (5.5,-0.05) {$\tau$};
\node[above] at (6.5,-0.05) {$a_1$};
\node[above] at (7.5,-0.2) {$\tau$};
\node[below] at (0.5,-0.95) {$\tau$};
\node[below] at (1.5,-0.95) {$a_1$};
\node[below] at (2.5,-0.95) {$\tau$};
\node[below] at (3.5,-0.95) {$a_1$};
\node[below] at (5.5,-0.95) {$\tau$};
\node[below] at (6.5,-0.95) {$a_1$};
\node[below] at (7.5,-0.95) {$\tau$};
\node[right] at (1.3,-0.45) {$\tau$};
\node[right] at (3.3,-0.45) {$\tau$};
\node[right] at (2.3,-0.5) {$a_1$};
\node[right] at (7.3,0.5) {$a_4$};
\node[right] at (6.3,-0.45) {$\tau$};
\node[right] at (5.3,-0.5) {$a_1$};
\node[right] at (7.3,-0.5) {$a_1$};
\node[left] at (0.8,-0.5) {$a_1$};
\node[left] at (7.05,0.5) {$a_1$};
\node[right] at (7.95,0.5) {$a_8$};
\node[above] at (7.5,0.95) {$a_7$};
\draw [-] (1, 0) -- (1,-1);
\draw [-] (2, 0) -- (2,-1);
\draw [fill] (0,0) circle [radius=.05];
\draw [fill] (0,-1) circle [radius=.05];
\draw [fill] (1,0) circle [radius=.05];
\draw [fill] (2,0) circle [radius=.05];
\draw [fill] (1,-1) circle [radius=.05];
\draw [fill] (2,-1) circle [radius=.05];
\draw [fill] (3,0) circle [radius=.05];
\draw [fill] (3,-1) circle [radius=.05];
\draw [fill] (4,0) circle [radius=.05];
\draw [fill] (4,-1) circle [radius=.05];
\draw [fill] (5,0) circle [radius=.05];
\draw [fill] (5,-1) circle [radius=.05];
\draw [fill] (6,0) circle [radius=.05];
\draw [fill] (6,-1) circle [radius=.05];
\draw [fill] (7,0) circle [radius=.05];
\draw [fill] (8,0) circle [radius=.05];
\draw [fill] (8,-1) circle [radius=.05];
\draw [fill] (7,1) circle [radius=.05];
\draw [fill] (8,1) circle [radius=.05];
\draw [fill] (4.3,-0.5) circle [radius=.02];
\draw [fill] (4.5,-0.5) circle [radius=.02];
\draw [fill] (4.7,-0.5) circle [radius=.02];
\draw [fill] (7,-1) circle [radius=.05];
\end{tikzpicture}}

\centerline{\begin{tikzpicture}
\draw[-] (-1,0) -- (0,0);
\draw[-] (-1,-1) -- (0,-1);
\draw[-] (-1,0) -- (-1,-1);
\draw[-] (0,0) -- (1,0);
\draw[-] (0,0) -- (0,-1);
\draw[-] (0,-1) -- (1,-1);
\draw[-] (1,0) -- (2,0);
\draw[-] (1,-1) -- (2,-1);
\draw[-] (2,0) -- (3,0);
\draw[-] (2,-1) -- (3,-1);
\draw[-] (3,0) -- (3,-1);
\draw[-] (3,0) -- (4,0);
\draw[-] (3,-1) -- (4,-1);
\draw[-] (4,0) -- (4,-1);
\draw[-] (5,0) -- (6,0);
\draw[-] (5,0) -- (5,-1);
\draw[-] (6,0) -- (6,-1);
\draw[-] (5,-1) -- (6,-1);
\draw[-] (6,0) -- (7,0);
\draw[-] (6,-1) -- (7,-1);
\draw[-] (7,0) -- (7,-1);
\draw[-] (7,0) -- (8,0);
\draw[-] (7,-1) -- (8,-1);
\draw[-] (8,0) -- (8,-1);
\draw[-] (8,0) -- (9,0);
\draw[-] (8,-1) -- (9,-1);
\draw[-] (9,0) -- (9,-1);
\draw[-] (8,1) -- (9,1);
\draw[-] (9,0) -- (9,1);
\draw[-] (8,0) -- (8,1);
\draw[dashed] (-1,0) -- (0,-1);
\draw[dashed] (1,0) -- (2,-1);
\draw[dashed] (0,0) -- (1,-1);
\draw[dashed] (2,0) -- (3,-1);
\draw[dashed] (3,0) -- (4,-1);
\draw[dashed] (5,0) -- (6,-1);
\draw[dashed] (6,0) -- (7,-1);
\draw[dashed] (7,0) -- (8,-1);
\draw[dashed] (8,0) -- (9,-1);
\draw[dashed] (8,1) -- (9,0);
\node[left] at (-0.95,-0.5) {$a_2$};
\node[left] at (2.25,-0.5) {$a_4$};
\node[left] at (4.25,-0.5) {$a_4$};
\node[left] at (5.25,-0.5) {$a_4$};
\node[left] at (7.25,-0.5) {$a_4$};
\node[left] at (1.25,-0.5) {$a_2$};
\node[left] at (0.25,-0.5) {$a_4$};
\node[left] at (3.25,-0.5) {$a_2$};
\node[left] at (6.25,-0.5) {$a_2$};
\node[left] at (8.25,-0.5) {$a_2$};
\node[right] at (8.95,-0.5) {$a_4$};
\node[above] at (0.5,-0.05) {$\tau'$};
\node[above] at (1.5,-0.05) {$a_1$};
\node[above] at (-0.5,-0.05) {$a_1$};
\node[above] at (2.5,-0.05) {$\tau'$};
\node[above] at (3.5,-0.05) {$a_1$};
\node[above] at (5.5,-0.05) {$\tau'$};
\node[above] at (6.5,-0.05) {$a_1$};
\node[above] at (8.5,-0.2) {$a_1$};
\node[above] at (7.5,-0.05) {$\tau'$};
\node[below] at (0.5,-0.95) {$\tau'$};
\node[below] at (1.5,-0.95) {$a_1$};
\node[below] at (-0.5,-0.95) {$a_1$};
\node[below] at (2.5,-0.95) {$\tau'$};
\node[below] at (3.5,-0.95) {$a_1$};
\node[below] at (5.5,-0.95) {$\tau'$};
\node[below] at (6.5,-0.95) {$a_1$};
\node[below] at (8.5,-0.95) {$a_1$};
\node[below] at (7.5,-0.95) {$\tau'$};
\node[right] at (1.3,-0.45) {$\tau'$};
\node[right] at (-0.7,-0.45) {$\tau'$};
\node[right] at (3.3,-0.45) {$\tau'$};
\node[right] at (2.3,-0.5) {$a_1$};
\node[right] at (6.3,-0.45) {$\tau'$};
\node[right] at (8.3,-0.45) {$\tau'$};
\node[right] at (8.3,0.55) {$a_4$};
\node[right] at (5.3,-0.5) {$a_1$};
\node[right] at (7.3,-0.5) {$a_1$};
\node[left] at (0.8,-0.5) {$a_1$};
\node[left] at (8.05,0.5) {$\tau'$};
\node[right] at (8.95,0.5) {$a_8$};
\node[above] at (8.5,0.95) {$a_7$};
\draw [-] (1, 0) -- (1,-1);
\draw [-] (2, 0) -- (2,-1);
\draw [fill] (0,0) circle [radius=.05];
\draw [fill] (0,-1) circle [radius=.05];
\draw [fill] (-1,0) circle [radius=.05];
\draw [fill] (-1,-1) circle [radius=.05];
\draw [fill] (1,0) circle [radius=.05];
\draw [fill] (2,0) circle [radius=.05];
\draw [fill] (1,-1) circle [radius=.05];
\draw [fill] (2,-1) circle [radius=.05];
\draw [fill] (3,0) circle [radius=.05];
\draw [fill] (3,-1) circle [radius=.05];
\draw [fill] (4,0) circle [radius=.05];
\draw [fill] (4,-1) circle [radius=.05];
\draw [fill] (5,0) circle [radius=.05];
\draw [fill] (5,-1) circle [radius=.05];
\draw [fill] (6,0) circle [radius=.05];
\draw [fill] (6,-1) circle [radius=.05];
\draw [fill] (7,0) circle [radius=.05];
\draw [fill] (8,0) circle [radius=.05];
\draw [fill] (8,-1) circle [radius=.05];
\draw [fill] (9,0) circle [radius=.05];
\draw [fill] (9,-1) circle [radius=.05];
\draw [fill] (9,1) circle [radius=.05];
\draw [fill] (8,1) circle [radius=.05];
\draw [fill] (4.3,-0.5) circle [radius=.02];
\draw [fill] (4.5,-0.5) circle [radius=.02];
\draw [fill] (4.7,-0.5) circle [radius=.02];
\draw [fill] (7,-1) circle [radius=.05];
\end{tikzpicture}}

Clearly we have $n^{\tau}(T,\rho)=n^{\tau'}(T',\rho)$.

\medskip

Since $G_{T,\rho}$ can be obtained by gluing a tile with a graph which is isomorphic to $H_s(a_1,\tau)$, we have $\mathcal P(G_{T,\rho})=\mathcal P_1\sqcup \mathcal P_2$, where $\mathcal P_1$ contains all $P$ which contains the upper right edge labeled $a_7$, $\mathcal P_2$ contains all $P$ which contains the upper right edge labeled $a_8$. By Lemma \ref{basicdecom}, as sets, we have $$\mathcal P_1\cong\textstyle\bigsqcup_{(\lambda_1,\cdots,\lambda_{s+1})\in \{0,1\}^{s+1}}\mathcal H'_{(\lambda_1,\cdots,\lambda_{s+1})}(a_1,\tau),$$ $$\mathcal P_2\cong\textstyle\bigsqcup_{(\lambda_1,\cdots,\lambda_{s})\in \{0,1\}^{s}}\mathcal H'_{(\lambda_1,\cdots,\lambda_{s})}(a_1,\tau).$$

%Clearly, $\mathcal P_1$ is isomorphic to $\mathcal P(G_s(a_1,\tau))$, $\mathcal P_2$ is isomorphic to $\mathcal P(G_{s-2}(a_1,\tau))$, $\mathcal P_3$ is isomorphic to $\mathcal P(G_{s-1}(a_1,\tau))$ and $\mathcal P_4$ is isomorphic to $\mathcal P(G_{s-1}(a_1,\tau))$ as sets.
\medskip

By the same reason, $\mathcal P(G_{T',\rho})=\mathcal P'_1\sqcup \mathcal P'_2$, where $\mathcal P'_1$ contains all $P$ which contains the upper right edge labeled $a_7$, $\mathcal P'_2$ contains all $P$ which contains the upper right edge labeled $a_8$. As sets, we have $$\mathcal P'_1\cong\textstyle\bigsqcup_{(\lambda_1,\cdots,\lambda_{s+1})\in \{0,1\}^{s+1}}\mathcal H_{(\lambda_1,\cdots,\lambda_{s+1})}(\tau',a_1),$$ $$\mathcal P'_2\cong\textstyle\bigsqcup_{(\lambda_1,\cdots,\lambda_{s})\in \{0,1\}^{s}}\mathcal H_{(\lambda_1,\cdots,\lambda_{s})}(\tau',a_1).$$

\medskip

As case (4), under the above isomorphisms, for each $P\in \mathcal H'_{(\lambda_1,\cdots,\lambda_{s+1})}(a_1,\tau)$, we associate $P$ with
$$\psi_{\rho}(P)=\mathcal H_{(1-\lambda_1,\cdots,1-\lambda_{s+1})}(\tau',a_1).$$ For each $P\in \mathcal H'_{(\lambda_1,\cdots,\lambda_{s})}(a_1,\tau)$, we associate $P$ with
$$\psi_{\rho}(P)=\mathcal H_{(1-\lambda_1,\cdots,1-\lambda_{s})}(\tau',a_1).$$

On the other hand, for each $P'\in \mathcal H_{(\lambda_1,\cdots,\lambda_{s+1})}(\tau',a_1)$, we associate $P'$ with $$\psi'_{\rho}(P')=\mathcal H'_{(1-\lambda_1,\cdots,1-\lambda_{s+1})}(a_1,\tau).$$ For each $P'\in \mathcal H_{(\lambda_1,\cdots,\lambda_{s})}(\tau',a_1)$, we associate $P'$ with $$\psi'_{\rho}(P')=\mathcal H'_{(1-\lambda_1,\cdots,1-\lambda_{s})}(a_1,\tau).$$

We can see $P'\in \psi_{\rho}(P)$ if and only if $P\in \psi'_{\rho}(P')$, in this case, by Lemma \ref{same-num} the number of non-$\tau$-mutable edges labeled $a$ in $P$ equals to the number of non-$\tau'$-mutable edges labeled $a$ in $P'$ for each $\tau\neq a\in T$.

\medskip

For each $P\in \mathcal P(G_{T,\rho})$ and $P'\in \mathcal P(G_{T',\rho})$ with $P'\in \psi_{\rho}(P)$, since $\psi_{\rho}(\mathcal P_i)\subset \mathcal P'_i$ for $i=1,2$, the last gluing edge of $G_{T,\rho}$ is in $P$ if and only if the last gluing edge of $G_{T',\rho}$ in $P'$. Since $\rho$ and $\gamma$ have the same starting point, $G_{T,\rho}$ and $G_{T',\rho}$ do not have first gluing edges.

%\medskip
%
%In this case, clearly the labels of $\tau$-mutable edges pair in $P_{\pm}(G_{T,\rho})$ and the labels of $\tau'$-mutable edges pair in $P_{\pm}(G_{T',\rho})$ are the same.

\medskip

In case (6), up to a difference of relative orientation, $G_{T,\rho}$ and $G_{T',\rho}$ are the following graphs, respectively,

\centerline{\begin{tikzpicture}
\draw[-] (0,-1) -- (0,-2);
\draw[-] (1,-1) -- (1,-2);
\draw[-] (0,-2) -- (1,-2);
\draw[-] (0,0) -- (1,0);
\draw[-] (0,0) -- (0,-1);
\draw[-] (0,-1) -- (1,-1);
\draw[-] (1,0) -- (2,0);
\draw[-] (1,-1) -- (2,-1);
\draw[-] (2,0) -- (3,0);
\draw[-] (2,-1) -- (3,-1);
\draw[-] (3,0) -- (3,-1);
\draw[-] (3,0) -- (4,0);
\draw[-] (3,-1) -- (4,-1);
\draw[-] (4,0) -- (4,-1);
\draw[-] (5,0) -- (6,0);
\draw[-] (5,0) -- (5,-1);
\draw[-] (6,0) -- (6,-1);
\draw[-] (5,-1) -- (6,-1);
\draw[-] (6,0) -- (7,0);
\draw[-] (6,-1) -- (7,-1);
\draw[-] (7,0) -- (7,-1);
\draw[-] (7,0) -- (8,0);
\draw[-] (7,-1) -- (8,-1);
\draw[-] (8,0) -- (8,-1);
\draw[-] (7,1) -- (8,1);
\draw[-] (7,0) -- (7,1);
\draw[-] (8,0) -- (8,1);
\draw[dashed] (1,0) -- (2,-1);
\draw[dashed] (0,0) -- (1,-1);
\draw[dashed] (0,-1) -- (1,-2);
\draw[dashed] (2,0) -- (3,-1);
\draw[dashed] (3,0) -- (4,-1);
\draw[dashed] (5,0) -- (6,-1);
\draw[dashed] (6,0) -- (7,-1);
\draw[dashed] (7,0) -- (8,-1);
\draw[dashed] (7,1) -- (8,0);
\node[left] at (0.05,-0.5) {$a_2$};
\node[left] at (2.25,-0.5) {$a_2$};
\node[left] at (4.25,-0.5) {$a_2$};
\node[left] at (5.25,-0.5) {$a_2$};
\node[left] at (7.25,-0.5) {$a_2$};
\node[left] at (1.25,-0.5) {$a_4$};
\node[left] at (3.25,-0.5) {$a_4$};
\node[left] at (6.25,-0.5) {$a_4$};
\node[right] at (7.95,-0.5) {$a_4$};
\node[above] at (0.5,-0.05) {$\tau$};
\node[above] at (1.5,-0.05) {$a_1$};
\node[above] at (2.5,-0.05) {$\tau$};
\node[above] at (3.5,-0.05) {$a_1$};
\node[above] at (5.5,-0.05) {$\tau$};
\node[above] at (6.5,-0.05) {$a_1$};
\node[above] at (7.5,-0.2) {$\tau$};
\node[below] at (0.5,-0.8) {$\tau$};
\node[below] at (1.5,-0.95) {$a_1$};
\node[below] at (2.5,-0.95) {$\tau$};
\node[below] at (3.5,-0.95) {$a_1$};
\node[below] at (0.5,-1.95) {$a_6$};
\node[below] at (5.5,-0.95) {$\tau$};
\node[below] at (6.5,-0.95) {$a_1$};
\node[below] at (7.5,-0.95) {$\tau$};
\node[right] at (1.3,-0.45) {$\tau$};
\node[right] at (3.3,-0.45) {$\tau$};
\node[right] at (2.3,-0.5) {$a_1$};
\node[right] at (7.3,0.5) {$a_4$};
\node[right] at (0.95,-1.5) {$a_1$};
\node[right] at (6.3,-0.45) {$\tau$};
\node[right] at (5.3,-0.5) {$a_1$};
\node[right] at (7.3,-0.5) {$a_1$};
\node[left] at (0.8,-0.5) {$a_1$};
\node[left] at (0.05,-1.5) {$a_5$};
\node[left] at (0.8,-1.5) {$a_2$};
\node[left] at (7.05,0.5) {$a_1$};
\node[right] at (7.95,0.5) {$a_8$};
\node[above] at (7.5,0.95) {$a_7$};
\draw [-] (1, 0) -- (1,-1);
\draw [-] (2, 0) -- (2,-1);
\draw [fill] (0,0) circle [radius=.05];
\draw [fill] (0,-1) circle [radius=.05];
\draw [fill] (0,-2) circle [radius=.05];
\draw [fill] (1,-2) circle [radius=.05];
\draw [fill] (1,0) circle [radius=.05];
\draw [fill] (2,0) circle [radius=.05];
\draw [fill] (1,-1) circle [radius=.05];
\draw [fill] (2,-1) circle [radius=.05];
\draw [fill] (3,0) circle [radius=.05];
\draw [fill] (3,-1) circle [radius=.05];
\draw [fill] (4,0) circle [radius=.05];
\draw [fill] (4,-1) circle [radius=.05];
\draw [fill] (5,0) circle [radius=.05];
\draw [fill] (5,-1) circle [radius=.05];
\draw [fill] (6,0) circle [radius=.05];
\draw [fill] (6,-1) circle [radius=.05];
\draw [fill] (7,0) circle [radius=.05];
\draw [fill] (8,0) circle [radius=.05];
\draw [fill] (8,-1) circle [radius=.05];
\draw [fill] (7,1) circle [radius=.05];
\draw [fill] (8,1) circle [radius=.05];
\draw [fill] (4.3,-0.5) circle [radius=.02];
\draw [fill] (4.5,-0.5) circle [radius=.02];
\draw [fill] (4.7,-0.5) circle [radius=.02];
\draw [fill] (7,-1) circle [radius=.05];
\end{tikzpicture}}

\centerline{\begin{tikzpicture}
\draw[-] (-1,0) -- (0,0);
\draw[-] (-1,-1) -- (0,-1);
\draw[-] (-1,-2) -- (0,-2);
\draw[-] (-1,-1) -- (-1,-2);
\draw[-] (0,-1) -- (0,-2);
\draw[-] (-1,0) -- (-1,-1);
\draw[-] (0,0) -- (1,0);
\draw[-] (0,0) -- (0,-1);
\draw[-] (0,-1) -- (1,-1);
\draw[-] (1,0) -- (2,0);
\draw[-] (1,-1) -- (2,-1);
\draw[-] (2,0) -- (3,0);
\draw[-] (2,-1) -- (3,-1);
\draw[-] (3,0) -- (3,-1);
\draw[-] (3,0) -- (4,0);
\draw[-] (3,-1) -- (4,-1);
\draw[-] (4,0) -- (4,-1);
\draw[-] (5,0) -- (6,0);
\draw[-] (5,0) -- (5,-1);
\draw[-] (6,0) -- (6,-1);
\draw[-] (5,-1) -- (6,-1);
\draw[-] (6,0) -- (7,0);
\draw[-] (6,-1) -- (7,-1);
\draw[-] (7,0) -- (7,-1);
\draw[-] (7,0) -- (8,0);
\draw[-] (7,-1) -- (8,-1);
\draw[-] (8,0) -- (8,-1);
\draw[-] (8,0) -- (9,0);
\draw[-] (8,-1) -- (9,-1);
\draw[-] (9,0) -- (9,-1);
\draw[-] (8,1) -- (9,1);
\draw[-] (9,0) -- (9,1);
\draw[-] (8,0) -- (8,1);
\draw[dashed] (-1,0) -- (0,-1);
\draw[dashed] (-1,-1) -- (0,-2);
\draw[dashed] (1,0) -- (2,-1);
\draw[dashed] (0,0) -- (1,-1);
\draw[dashed] (2,0) -- (3,-1);
\draw[dashed] (3,0) -- (4,-1);
\draw[dashed] (5,0) -- (6,-1);
\draw[dashed] (6,0) -- (7,-1);
\draw[dashed] (7,0) -- (8,-1);
\draw[dashed] (8,0) -- (9,-1);
\draw[dashed] (8,1) -- (9,0);
\node[left] at (-0.95,-0.5) {$a_2$};
\node[left] at (2.25,-0.5) {$a_4$};
\node[left] at (4.25,-0.5) {$a_4$};
\node[left] at (5.25,-0.5) {$a_4$};
\node[left] at (7.25,-0.5) {$a_4$};
\node[left] at (1.25,-0.5) {$a_2$};
\node[left] at (0.25,-0.5) {$a_4$};
\node[left] at (3.25,-0.5) {$a_2$};
\node[left] at (6.25,-0.5) {$a_2$};
\node[left] at (8.25,-0.5) {$a_2$};
\node[right] at (8.95,-0.5) {$a_4$};
\node[above] at (0.5,-0.05) {$\tau'$};
\node[above] at (1.5,-0.05) {$a_1$};
\node[above] at (-0.5,-0.05) {$a_1$};
\node[above] at (2.5,-0.05) {$\tau'$};
\node[above] at (3.5,-0.05) {$a_1$};
\node[above] at (5.5,-0.05) {$\tau'$};
\node[above] at (6.5,-0.05) {$a_1$};
\node[above] at (8.5,-0.2) {$a_1$};
\node[above] at (7.5,-0.05) {$\tau'$};
\node[below] at (0.5,-0.95) {$\tau'$};
\node[below] at (1.5,-0.95) {$a_1$};
\node[below] at (-0.5,-0.8) {$a_1$};
\node[below] at (2.5,-0.95) {$\tau'$};
\node[below] at (3.5,-0.95) {$a_1$};
\node[below] at (5.5,-0.95) {$\tau'$};
\node[below] at (6.5,-0.95) {$a_1$};
\node[below] at (8.5,-0.95) {$a_1$};
\node[below] at (7.5,-0.95) {$\tau'$};
\node[below] at (-0.5,-1.95) {$a_6$};
\node[right] at (1.3,-0.45) {$\tau'$};
\node[right] at (-0.05,-1.5) {$\tau'$};
\node[right] at (-0.7,-0.45) {$\tau'$};
\node[right] at (-0.8,-1.45) {$a_2$};
\node[right] at (3.3,-0.45) {$\tau'$};
\node[right] at (2.3,-0.5) {$a_1$};
\node[right] at (6.3,-0.45) {$\tau'$};
\node[right] at (8.3,-0.45) {$\tau'$};
\node[right] at (8.3,0.55) {$a_4$};
\node[right] at (5.3,-0.5) {$a_1$};
\node[right] at (7.3,-0.5) {$a_1$};
\node[left] at (0.8,-0.5) {$a_1$};
\node[left] at (-0.95,-1.5) {$a_5$};
\node[left] at (8.05,0.5) {$\tau'$};
\node[right] at (8.95,0.5) {$a_8$};
\node[above] at (8.5,0.95) {$a_7$};
\draw [-] (1, 0) -- (1,-1);
\draw [-] (2, 0) -- (2,-1);
\draw [fill] (0,0) circle [radius=.05];
\draw [fill] (0,-1) circle [radius=.05];
\draw [fill] (0,-2) circle [radius=.05];
\draw [fill] (-1,-2) circle [radius=.05];
\draw [fill] (-1,0) circle [radius=.05];
\draw [fill] (-1,-1) circle [radius=.05];
\draw [fill] (1,0) circle [radius=.05];
\draw [fill] (2,0) circle [radius=.05];
\draw [fill] (1,-1) circle [radius=.05];
\draw [fill] (2,-1) circle [radius=.05];
\draw [fill] (3,0) circle [radius=.05];
\draw [fill] (3,-1) circle [radius=.05];
\draw [fill] (4,0) circle [radius=.05];
\draw [fill] (4,-1) circle [radius=.05];
\draw [fill] (5,0) circle [radius=.05];
\draw [fill] (5,-1) circle [radius=.05];
\draw [fill] (6,0) circle [radius=.05];
\draw [fill] (6,-1) circle [radius=.05];
\draw [fill] (7,0) circle [radius=.05];
\draw [fill] (8,0) circle [radius=.05];
\draw [fill] (8,-1) circle [radius=.05];
\draw [fill] (9,0) circle [radius=.05];
\draw [fill] (9,-1) circle [radius=.05];
\draw [fill] (9,1) circle [radius=.05];
\draw [fill] (8,1) circle [radius=.05];
\draw [fill] (4.3,-0.5) circle [radius=.02];
\draw [fill] (4.5,-0.5) circle [radius=.02];
\draw [fill] (4.7,-0.5) circle [radius=.02];
\draw [fill] (7,-1) circle [radius=.05];
\end{tikzpicture}}

Clearly we have $n^{\tau}(T,\rho)=n^{\tau'}(T',\rho)$.

\medskip

Since we can obtain $G_{T,\rho}$ by gluing two tiles with a graph which is isomorphic to $H_s(a_1,\tau)$, we have $\mathcal P(G_{T,\rho})=\mathcal P_1\sqcup \mathcal P_2\sqcup \mathcal P_3\sqcup \mathcal P_4$, where $\mathcal P_1$ contains all $P$ which contains the lower left edge labeled $a_6$ and the upper right edge labeled $a_7$, $\mathcal P_2$ contains all $P$ which contains the lower left edge labeled $a_6$ and the upper right edge labeled $a_8$, $\mathcal P_3$ contains all $P$ which contains the lower left edge labeled $a_5$ and the upper right edge labeled $a_7$, $\mathcal P_4$ contains all $P$ which contains the lower left edge labeled $a_5$ and the upper right edge labeled $a_8$. By Lemma \ref{basicdecom}, as sets, we have
$$\mathcal P_1\cong\textstyle\bigsqcup_{(\lambda_1,\cdots,\lambda_{s+1})\in \{0,1\}^{s+1}}\mathcal H'_{(\lambda_1,\cdots,\lambda_{s+1})}(a_1,\tau),$$ $$\mathcal P_2\cong\textstyle\bigsqcup_{(\lambda_1,\cdots,\lambda_{s})\in \{0,1\}^{s}}\mathcal H'_{(\lambda_1,\cdots,\lambda_{s})}(a_1,\tau),$$ $$\mathcal P_3\cong\textstyle\bigsqcup_{(\lambda_2,\cdots,\lambda_{s+1})\in \{0,1\}^{s}}\mathcal H'_{(\lambda_2,\cdots,\lambda_{s+1})}(a_1,\tau),$$ $$\mathcal P_4\cong\textstyle\bigsqcup_{(\lambda_2,\cdots,\lambda_{s})\in \{0,1\}^{s-1}}\mathcal H'_{(\lambda_2,\cdots,\lambda_{s})}(a_1,\tau).$$

Similarly, $\mathcal P(G_{T',\rho})=\mathcal P'_1\sqcup \mathcal P'_2\sqcup \mathcal P'_3\sqcup \mathcal P'_4$, where $\mathcal P'_1$ contains all $P$ which contains the lower left edge labeled $a_6$ and the upper right edge labeled $a_7$, $\mathcal P'_2$ contains all $P$ which contains the lower left edge labeled $a_6$ and the upper right edge labeled $a_8$, $\mathcal P'_3$ contains all $P$ which contains the lower left edge labeled $a_5$ and the upper right edge labeled $a_7$, $\mathcal P'_4$ contains all $P$ which contains the lower left edge labeled $a_5$ and the upper right edge labeled $a_8$. As sets, we have $$\mathcal P'_1\cong\textstyle\bigsqcup_{(\lambda_1,\cdots,\lambda_{s+1})\in \{0,1\}^{s+1}}\mathcal H_{(\lambda_1,\cdots,\lambda_{s+1})}(\tau',a_1),$$ $$\mathcal P'_2\cong\textstyle\bigsqcup_{(\lambda_1,\cdots,\lambda_s)\in \{0,1\}^{s}}\mathcal H_{(\lambda_1,\cdots,\lambda_s)}(\tau',a_1),$$ $$\mathcal P'_3\cong\textstyle\bigsqcup_{(\lambda_2,\cdots,\lambda_{s+1})\in \{0,1\}^{s}}\mathcal H_{(\lambda_2,\cdots,\lambda_{s+1})}(\tau',a_1),$$ $$\mathcal P'_4\cong\textstyle\bigsqcup_{(\lambda_2,\cdots,\lambda_{s})\in \{0,1\}^{s-1}}\mathcal H_{(\lambda_2,\cdots,\lambda_{s})}(\tau',a_1).$$

\medskip

As cases (4) and (5), under the above isomorphisms, for each $P\in \mathcal H'_{(\lambda_1,\cdots,\lambda_{s+1})}(a_1,\tau)$, we associate $P$ with
$$\psi_{\rho}(P)=\mathcal H_{(1-\lambda_1,\cdots,1-\lambda_{s+1})}(\tau',a_1).$$ For each $P\in \mathcal H'_{(\lambda_1,\cdots,\lambda_{s})}(a_1,\tau)$, we associate $P$ with
$$\psi_{\rho}(P)=\mathcal H_{(1-\lambda_1,\cdots,1-\lambda_{s})}(\tau',a_1).$$ For each $P\in \mathcal H'_{(\lambda_2,\cdots,\lambda_{s+1})}(a_1,\tau)$, we associate $P$ with
$$\psi_{\rho}(P)=\mathcal H_{(1-\lambda_2,\cdots,1-\lambda_{s+1})}(\tau',a_1).$$ For each $P\in \mathcal H'_{(\lambda_2,\cdots,\lambda_{s})}(a_1,\tau)$, we associate $P$ with
$$\psi_{\rho}(P)=\mathcal H_{(1-\lambda_2,\cdots,1-\lambda_{s})}(\tau',a_1).$$

On the other hand, for each $P'\in \mathcal H_{(\lambda_1,\cdots,\lambda_{s+1})}(\tau',a_1)$, we associate $P'$ with $$\psi'_{\rho}(P')=\mathcal H'_{(1-\lambda_1,\cdots,1-\lambda_{s+1})}(a_1,\tau).$$ For each $P'\in \mathcal H_{(\lambda_1,\cdots,\lambda_{s})}(\tau',a_1)$, we associate $P'$ with $$\psi'_{\rho}(P')=\mathcal H'_{(1-\lambda_1,\cdots,1-\lambda_{s})}(a_1,\tau).$$ For each $P'\in \mathcal H_{(\lambda_2,\cdots,\lambda_{s+1})}(\tau',a_1)$, we associate $P'$ with $$\psi'_{\rho}(P')=\mathcal H'_{(1-\lambda_2,\cdots,1-\lambda_{s+1})}(a_1,\tau).$$ For each $P'\in \mathcal H_{(\lambda_2,\cdots,\lambda_{s})}(\tau',a_1)$, we associate $P'$ with $$\psi'_{\rho}(P')=\mathcal H'_{(1-\lambda_2,\cdots,1-\lambda_{s})}(a_1,\tau).$$

We can see $P'\in \psi_{\rho}(P)$ if and only if $P\in \psi'_{\rho}(P')$, in this case, by Lemma \ref{same-num} the number of non-$\tau$-mutable edges labeled $a$ in $P$ equals to the number of non-$\tau'$-mutable edges labeled $a$ in $P'$ for each $\tau\neq a\in T$.

\medskip

For each $P\in \mathcal P(G_{T,\rho})$ and $P'\in \mathcal P(G_{T',\rho})$ with $P'\in \psi_{\rho}(P)$, since $\psi_{\rho}(\mathcal P_i)\subset \mathcal P'_i$ for $i=1,2,3,4$, the first/last gluing edge of $G_{T,\rho}$ is in $P$ if and only if the first/last gluing edge of $G_{T',\rho}$ in $P'$.

\medskip

%In this case, clearly the labels of $\tau$-mutable edges pair in $P_{\pm}(G_{T,\rho})$ and the labels of $\tau'$-mutable edges pair in $P_{\pm}(G_{T',\rho})$ are the same.
%
%\medskip

We summarize above discussions as the following proposition.

\begin{Proposition}\label{compare-loc}

Keep the foregoing notations. Let $T$ be a triangulation of $\mathcal O$ and $\tau$ be an arc in $T$. Let $\tau'\neq\gamma\notin T$ be an arc in $\mathcal O$. Let $\rho$ be a subcurve of $\gamma$
satisfying the assumptions at the beginning of this section. Then

\begin{enumerate}[$(1)$]

  \item $n^{\tau}(T,\rho)=n^{\tau'}(T',\rho)$.

%  \item The labels of $\tau$-mutable edges pair in $P_{\pm}(G_{T,\rho})$ and the labels of $\tau'$-mutable edges pair in $P_{\pm}(G_{T',\rho})$ are the same.

  \item We can associate a subset $\psi_{\rho}(P)$ of $\mathcal P(G_{T'},\rho)$ for each $P\in \mathcal P(G_{T,\rho})$ and a subset $\psi'_{\rho}(P')$ of $\mathcal P(G_{T,\rho})$ for each $P'\in \mathcal P(G_{T',\rho})$, which satisfy

      \begin{enumerate}[$(a)$]

        \item $P'\in \psi_{\rho}(P)$ if and only if $P\in \psi'_{\rho}(P')$. In such case,

        \item $P\in \mathcal P^{\tau}_\nu(G_{T,\rho})$ for some $\nu$ if and only if $P'\in \mathcal P^{\tau'}_{-\nu}(G_{T',\rho})$;

        \item the number of non-$\tau$-mutable edges labeled $a$ in $P$ equals to the number of non-$\tau'$-mutable edges labeled $a$ in $P'$ for each $\tau\neq a\in T$;

        \item the first/last gluing edge of $G_{T,\rho}$ is in $P$ if and only if the first/last gluing edge of $G_{T',\rho}$ is in $P'$.

     \end{enumerate}

\end{enumerate}

\end{Proposition}

\begin{proof}

We already proved (1) for each case.

When $a_1\neq a_3, a_2\neq a_4$, we have a partition bijection between $\mathcal P(G_{T,\rho})$ and $\mathcal P(G_{T',\rho})$. For any $P\in\mathcal P(G_{T,\rho})$, there is a unique $S\subset \mathcal P(G_{T,\rho})$ such that $P\in S$ according to the partition bijection. We define $\psi_{\rho}(P)$ be the subset of $\mathcal P(G_{T',\rho})$ which corresponds to $S$ under the partition bijection. Use the same method, we can define $\psi'_{\rho}(P')$ for any $P'\in \mathcal P(G_{T',\rho})$. It is easy to check that $\psi_{\rho}$ and $\psi'_{\rho}$ satisfy the conditions of (2).

When $a_1=a_3$ or $a_2=a_4$, we may assume $a_1=a_3$. We have already proved (2) for cases (4,5,6). Dually, (2) holds for cases (8,9,10). In the remaining cases, we have a partition bijection between $\mathcal P(G_{T,\rho})$ and $\mathcal P(G_{T',\rho})$. The results can be proved similarly as the case $a_1\neq a_3,a_2\neq a_4$.
\end{proof}

\medskip

The following observations would help us to prove theorem \ref{partition bi}.

\medskip

\begin{Lemma}\label{diff2}

Keep the foregoing notations. Let $P,Q\in \mathcal P^{\tau}_{\nu}(G_{T,\rho})$ with $P\neq Q$ for some $\nu$. Then either $\psi_{\rho}(P)\cap \psi_{\rho}(Q)=\emptyset$ or there exists a tile $G$ of $G_{T,\rho}$ with diagonal labeled $\tau$ such that $P$ and $Q$ can twist on $G$ and $P\cap E(G)\neq Q\cap E(G)$.

\end{Lemma}

\begin{proof}

When $a_1\neq a_3, a_2\neq a_4$, we have a partition bijection between $\mathcal P(G_{T,\rho})$ and $\mathcal P(G_{T',\rho})$. By the definition of $\psi_{\rho}$, it is easy to see the following fact. If $P$ and $Q$ are not in the same set $S$ which in the domain of the partition bijection, then $\psi_{\rho}(P)\cap \psi_{\rho}(Q)=\emptyset$. Otherwise, the latter possibility happens.

When $a_1= a_3$ or $a_2= a_4$, we may assume $a_1=a_3$. In cases (4,5,6,8,9,10), by Lemma \ref{basicdecom} and Lemma \ref{diff}, there exists a tile $G$ with diagonal labeled $\tau$ such that $P$ and $Q$ can twist on $G$ and $P\cap E(G)\neq Q\cap E(G)$. In the remaining cases, the result can be proved similarly as the case $a_1\neq a_3, a_2\neq a_4$.
\end{proof}

\medskip

\begin{Lemma}\label{2^m}

Keep the foregoing notations.

\begin{enumerate}[$(1)$]

  \item For any $P\in \mathcal P^{\tau}_{\nu}(G_{T,\rho})$, we have $|\psi_{\rho}(P)|=2^{\sum_{\nu_l=1}\nu_l}$. Precisely, for each $t$ such that $\nu_t=1$, there are two possibilities to choose the edges of $G'_{l_t}$ as $\tau'$-mutable edges pair in $P'$ for each $P'\in \psi_{\rho}(P)$, the other remaining edges are the same for $P'\neq P''\in \psi_{\rho}(P)$. Here $G'_{l_t}$ is the tile of $G_{T',\rho}$ determined by the $t$-th $\tau'$-equivalence class when $\nu_t=1$.

  \item For any $P'\in \mathcal P^{\tau'}_{\nu}(G_{T',\rho})$, we have $|\psi'_{\rho}(P')|=2^{\sum_{\nu_l=1}\nu_l}$. Precisely, for each $t$ such that $\nu_t=1$, there are two possibilities to choose the edges of $G_{l_t}$ as $\tau$-mutable edges pair in $P$ for each $P\in \psi'_{\rho}(P')$, the other remaining edges are the same for $P\neq Q\in \psi_{\rho}(P)$. Here $G_{l_t}$ is the tile of $G_{T,\rho}$ determined by the $t$-th $\tau$-equivalence class when $\nu_t=1$.

\end{enumerate}

\end{Lemma}

\begin{proof}

We shall only prove (1) because (2) can be proved dually.

When $a_1\neq a_3, a_2\neq a_4$. It is easy to verify the result by checking the partition bijection between $\mathcal P(G_{T,\rho})$ and $\mathcal P(G_{T',\rho})$.

When $a_1=a_3$ or $a_2=a_4$, we may assume $a_1=a_3$. In cases (4,5,6), by Lemma \ref{basicdecom} (1), $\psi_{\rho}(P)=\mathcal H_{(\lambda_1,\cdots,\lambda_s)}(\tau',a_1)$ for some $s$ and $(\lambda_1,\cdots,\lambda_s)$. It is easy to see $\mathcal H_{(\lambda_1,\cdots,\lambda_s)}(\tau',a_1)$ satisfies the required condition. In cases (8,9,10), by Lemma \ref{basicdecom} (2), $\psi_{\rho}(P)=\mathcal H'_{(\lambda_1,\cdots,\lambda_s)}(a_1,\tau)$ for some $s$ and $(\lambda_1,\cdots,\lambda_s)$. It is easy to see $\mathcal H'_{(\lambda_1,\cdots,\lambda_s)}(a_1,\tau)$ satisfies the required condition. In the remaining cases, the result can be proved similarly as the case $a_1\neq a_3, a_2\neq a_4$.
\end{proof}

\medskip

We need the following observations for the proof of Theorem \ref{mainthm}.

\medskip

\begin{Lemma}\label{maxtomax}
Keep the foregoing notations.
\begin{enumerate}[$(1)$]

  \item $P_{+}(G_{T',\rho})\subset \psi_{\rho}(P_{+}(G_{T,\rho}))$,\;\;  $P_{-}(G_{T',\rho})\subset \psi_{\rho}(P_{-}(G_{T,\rho}))$.

  \item $P_{+}(G_{T,\rho})\subset \psi'_{\rho}(P_{+}(G_{T',\rho}))$,\;\;  $P_{-}(G_{T,\rho})\subset \psi'_{\rho}(P_{-}(G_{T',\rho}))$.

\end{enumerate}

\end{Lemma}

\begin{proof}

We shall only prove (1) because (2) can be proved dually.

Firstly, suppose that $a_1\neq a_3, a_2\neq a_4$. It is easy to see the result holds by checking the partition bijection between $\mathcal P(G_{T,\rho})$ and $\mathcal P(G_{T',\rho})$.

Then, suppose that $a_1=a_3$ or $a_2=a_4$, we may assume $a_1=a_3$. In cases (4,5,6), since $P_{+}(H_s(a_1,\tau))\in \mathcal H'_{(0,\cdots,0)}(a_1,\tau)$, $P_{-}(H_s(a_1,\tau))\in \mathcal H'_{(1,\cdots,1)}(a_1,\tau)$ and  $P_{+}(H_{s+1}(\tau',a_1))\in \mathcal H_{(1,\cdots,1)}(\tau',a_1)$, $P_{-}(H_{s+1}(\tau',a_1))\in \mathcal H_{(0,\cdots,0)}(\tau',a_1)$, our result follows in these cases. Dually, our result holds in cases (8,9,10).
In the remainding cases, the result can be verified similarly as the case $a_1\neq a_3, a_2\neq a_4$.
\end{proof}

\medskip

\begin{Lemma}\label{com}

Keep the foregoing notations.

\begin{enumerate}[$(1)$]

  \item Suppose that $P\in \mathcal P(G_{T,\rho})$ can twist on a tile $G(p)$ with diagonal labeled $a=a_q,q=1,2,3,4$. If all $\tau$-mutable edges pairs in $P$ are labeled $a_{q-1},a_{q+1}$, then there exists $P'\in \psi_{\rho}(P)$ such that $P'$ satisfies

       \begin{enumerate}[$(a)$]

         \item $P'$ can twist on $G'(p)$, here $G'(p)$ is the tile of $G_{T',\rho}$ determined by $p$.

         \item $\mu_pP'\in \psi_{\rho}(\mu_{p}P)$.

         \item All $\tau'$-mutable edges pairs in $P'$ are labeled $a_{q-1}, a_{q+1}$.

         \item $n^+_p(a,P)-m^+_p(a,\rho)=n^+_p(a,P')-m^+_p(a,\rho)$ and $n^-_p(a,P)-m^-_p(a,\rho)=n^-_p(a,P')-m^-_p(a,\rho)$.

       \end{enumerate}

  \item Suppose that $P'\in \mathcal P(G_{T',\rho})$ can twist on a tile $G'(p)$ with diagonal labeled $a=a_q,q=1,2,3,4$. If all $\tau'$-mutable edges pairs in $P'$ are labeled $a_{q-1},a_{q+1}$, then there exists $P\in \psi'_{\rho}(P')$ such that $P$ satisfies

      \begin{enumerate}[$(a)$]

         \item $P$ can twist on $G(p)$, here $G(p)$ is the tile of $G_{T,\rho}$ determined by $p$.

         \item $\mu_pP\in \psi'_{\rho}(\mu_{p}P')$.

         \item all $\tau$-mutable edges pairs in $P$ are labeled $a_{q-1}, a_{q+1}$.

         \item $n^+_p(a,P)-m^+_p(a,\rho)=n^+_p(a,P')-m^+_p(a,\rho)$ and $n^-_p(a,P)-m^-_p(a,\rho)=n^-_p(a,P')-m^-_p(a,\rho)$.

       \end{enumerate}

\end{enumerate}

\end{Lemma}

\begin{proof}

We shall only prove (1) because (2) can be proved dually.

When $a_1\neq a_3, a_2\neq a_4$, by checking the partition bijection between $\mathcal P(G_{T,\rho})$ and $\mathcal P(G_{T',\rho})$, we can see the results hold.

When $a_1= a_3$ or $a_2= a_4$, we may assume $a_1=a_3$. The cases (1,2,3,7) can be verified similarly as the case $a_1\neq a_3,a_2\neq a_4$.

In case (4), $a=a_1$, $P\in \mathcal H'_{(\lambda_1,\cdots,\lambda_{s+2})}(a_1,\tau)$ for some $\lambda\in \{0,1\}^{s+2}$. Thus, $\psi_{\rho}(P)=\mathcal H_{(1-\lambda_1,\cdots,1-\lambda_{s+2})}(\tau',a_1)$. Assume $G(p)$ is the $(2i-1)$-th tile, we have $\mu_{p}P\in \mathcal H'_{(\lambda_1,\cdots,\lambda_{i-1},1-\lambda_i,\lambda_{i+1},\cdots,\lambda_{s+2})}(a_1,\tau)$. For any perfect matching $P'$ in $\psi_{\rho}(P)$, $P'$ can twist on $G'(p)$, the $2i$-th tile of $G_{T',\rho}$, and $\mu_pP'\in \mathcal H_{(1-\lambda_1,\cdots,\lambda_i,\cdots,1-\lambda_s)}(\tau',a_1)=\psi_{\rho}(\mu_pP)$. We can choose $P'\in \psi_{\rho}(P)$ such that all the $\tau'$-mutation edges pairs are labeled $a_2,a_4$. Then $n^+_p(a,P)-m^+_p(a,\rho)$ and $n^+_p(a,P')-m^+_p(a,\rho)$ are both equal to $s+2-i$ and $n^-_p(a,P)-m^-_p(a,\rho)$ and $n^-_p(a,P')-m^-_p(a,\rho)$ are both equal to $i-1$.

In case (5), if $a=a_1$, then the results can be proved similarly as case (4). Now suppose that $a\neq a_1$, then $a=a_4$ and $G(p)$ is the last tile of $G_{T,\rho}$.

If the upper right edge labeled $a_8$ of $G(p)$ is in $P$, then $P\in \mathcal H'_{(\lambda_1,\cdots,\lambda_{s+1})}(a_1,\tau)$ for some $\lambda\in \{0,1\}^{s+1}$. Thus $\psi_{\rho}(P)=\mathcal H_{(1-\lambda_1,\cdots,1-\lambda_{s+1})}(\tau',a_1)$. We have the upper right edge labeled $a_8$ of $G'(p)$ is in $P'$ for any $P'\in \psi_{\rho}(P)$, where $G'(p)$ is the last tile of $G_{T',\rho}$. Thus, $P'$ can twist on the tile $G'(p)$ for any $P'\in \psi_{\rho}(P)$. We can choose $P'\in \psi_{\rho}(P)$ such that all the $\tau'$-mutation edges pairs are labeled $a_1$. Since $\mu_pP\in \mathcal H'_{(\lambda_1,\cdots,\lambda_{s+1},1)}(a_1,\tau)$ and $\mu_pP'\in \mathcal H_{(1-\lambda_1,\cdots,1-\lambda_{s+1},1)}(\tau',a_1)$, $\mu_pP'\in \psi_{\rho}(\mu_pP)$.

If the upper right edge labeled $a_7$ of $G(p)$ is in $P$, then $P\in \mathcal H'_{(\lambda_1,\cdots,\lambda_{s+2})}(a_1,\tau)$ for some $\lambda\in \{0,1\}^{s+2}$. Since $P$ can twist on $G(p)$, the edge labeled $\tau$ of $G(p)$ is in $P$, equivalently, $\lambda_{s+2}=1$. Then the upper right edge labeled $a_8$ of $G(p)$ is in $\mu_pP$. Change the roles of $P$ and $\mu_pP$, this case can be proved as the above case.

In both cases, by Lemma \ref{non-tau-mu1} and Lemma \ref{non-tau-mu2}, the edge labeled $a_4$ of the $(2i-1)$-th tile of $G_{T,\rho}$ is non-$\tau$-mutable in $P$ if and only if the edge labeled
$a_4$ of the $(2i+1)$-th tile of $G_{T',\rho}$ is non-$\tau'$-mutable in $P'$, we have $n^+_p(a,P)-m^+_p(a,\rho)=n^+_p(a,P')-m^+_p(a,\rho)$ and $n^-_p(a,P)-m^-_p(a,\rho)=n^-_p(a,P')-m^-_p(a,\rho)$, note that the edge of the first tile of $G_{T',\rho}$ labeled $a_4$ can not be a non-$\tau'$-mutable edge.

Therefore, our results hold for case (5).

For case (6), the results can be proved similarly to cases (4) and (5). Therefore, the results follow in cases (4,5,6). Dually, the results can be proved in cases (8,9,10).

Similarly, one can check the results hold if we change the roles of $\tau$ and $\tau'$. Our results follow.
\end{proof}

\medskip

\section{Proof of the Theorem \ref{partition bi}}\label{main1}

In this section, we prove Theorem \ref{partition bi}. Let $\gamma_i,i\in [1,k]$ be the subcurves of $\gamma$ defined at the beginning of Section \ref{compare}. Our strategy is the following: first, we show each $P\in \mathcal P(G_{T,\gamma})$ uniquely corresponds to an element $(P_i)\in \coprod \mathcal P(G_{T,\gamma_i})$, see Proposition \ref{decompose}. Secondly, using the functions $\psi_{\gamma_i}, \psi'_{\gamma_i}$ defined in Section \ref{compare} we construct the partition bijection between $\mathcal P(G_{T,\gamma})$ and $\mathcal P(G_{\mu_{\tau}(T),\gamma})$ in Subsection \ref{main11}. We mainly deal with the coefficients in subsection \ref{main12}.

\medskip

\subsection{Proof of the Theorem \ref{partition bi} (1)}\label{main11}

Given a perfect matching $P$ of $G_{T,\gamma}$, for each $j$, we associate $P$ a perfect matching $\iota_j(P)$ of $G_{T,\gamma_j}$ as follows: denote the edges incident to $u_i$ by $v_i,w_i$ and $v'_i,w'_i$, as shown in the figure below, where $v_i,w_i$ are edges of $G_{T,\gamma_i}$ and $v'_i,w'_i$ are edges of $G_{T,\gamma_{i+1}}$. By Lemma \ref{inone}, $u_i\in P$ or $v_i\in P$ or $w'_i\in P$.

\centerline{\begin{tikzpicture}
\draw[-] (1,0) -- (2,0);
\draw[-] (1,-1) -- (2,-1);
\draw[-] (2,0) -- (3,0);
\draw[-] (2,-1) -- (3,-1);
\draw[-] (3,0) -- (3,-1);
\node[above] at (1.5,-0.05) {$v_i$};
\node[below] at (1.5,-0.95) {$w_i$};
\node[right] at (1.75,-0.5) {$u_i$};
\node[above] at (2.5,-0.05) {$v'_i$};
\node[below] at (2.5,-0.95) {$w'_i$};
\draw [-] (1, 0) -- (1,-1);
\draw [-] (2, 0) -- (2,-1);
\draw [fill] (1,0) circle [radius=.05];
\draw [fill] (2,0) circle [radius=.05];
\draw [fill] (1,-1) circle [radius=.05];
\draw [fill] (2,-1) circle [radius=.05];
\draw [fill] (3,0) circle [radius=.05];
\draw [fill] (3,-1) circle [radius=.05];
\end{tikzpicture}}

 \[\begin{array}{ccl} \iota_j(P) &=&

         \left\{\begin{array}{ll}

             P\cap E(G_{T,\gamma_j}), &\mbox{if $v_{j-1},w'_{j}\notin P$}, \\

             P\cap E(G_{T,\gamma_j})\cup\{u_{j-1}\}, &\mbox{if $v_{j-1}\in P$ and $w'_{j}\notin P$}, \\

             P\cap E(G_{T,\gamma_j})\cup\{u_{j}\}, &\mbox{if $v_{j-1}\notin P$ and $w'_{j}\in P$}, \\

             P\cap E(G_{T,\gamma_j})\cup\{u_{j-1},u_j\}, &\mbox{if $v_{j-1},w'_{j}\in P$},

         \end{array}\right.

 \end{array}\]
Specially,

 \[\begin{array}{ccl} \iota_1(P) &=&

         \left\{\begin{array}{ll}

             P\cap E(G_{T,\gamma_1}), &\mbox{if $w'_{1}\notin P$}, \\

             P\cap E(G_{T,\gamma_1})\cup\{u_1\}, &\mbox{if $w'_{1}\in P$}.

         \end{array}\right.

 \end{array}\]

 \[\begin{array}{ccl} \iota_k(P) &=&

         \left\{\begin{array}{ll}

             P\cap E(G_{T,\gamma_1}), &\mbox{if $v_{k-1}\notin P$}, \\

             P\cap E(G_{T,\gamma_1})\cup\{u_{k-1}\}, &\mbox{if $v_{k-1}\in P$}.

         \end{array}\right.

 \end{array}\]

One can see $\iota_j(P)$ is a perfect matching of $G_{T,\gamma_j}$ by definition for $j\in [1, k]$.

\medskip

The following proposition illustrates $\mathcal P(G_{T,\gamma})$ is a subset of $\coprod_{i=1}^{k}\mathcal P(G_{T,\gamma_i})$. Thus we can study $\mathcal P(G_{T,\gamma})$ ``locally".

\medskip

\begin{Proposition}\label{decompose}

Keep the notations as above. Assume $\tau'\neq \gamma \notin T$. Then we have a bijection
$$\iota: \mathcal P(G_{T,\gamma})\rightarrow \{(P_i)_{i} \mid P_i\in \mathcal P(G_{T,\gamma_i}), \forall j, u_j\in P_j\cup P_{j+1}\},\;\;\;P\rightarrow (\iota_i(P))_{i}.$$

\end{Proposition}

\begin{proof}
We first prove $\iota$ is well-defined. For any $1\leq j<k$, by the constructions of $\iota_j(P)$ and $\iota_{j+1}(P)$, $u_j\notin \iota_j(P)$ if and only if $v_j\in P$, $u_j\notin \iota_{j+1}(P)$ if and only if $w'_j\in P$. However, $v_j\in P$ and $w'_j\in P$ can not hold at the same time by Lemma \ref{inone}. Thus for any $1\leq j<k$, $u_j\in \iota_j(P)\cup \iota_{j+1}(P)$.

Next we show $\iota$ is a bijection. Define
$$\iota^{-1}: \{(P_i)_{i}\;|\;P_i\in \mathcal P(G_{T,\gamma_i}), \forall j, u_j\in P_j\cup P_{j+1}\}\rightarrow \mathcal P(G_{T,\gamma}),$$
$$(P_i)_i\rightarrow (\cup_{i} P_i)\setminus \{u_1,\cdots,u_{k-1}\}.$$
Here in the notation $(\cup_{i} P_i)\setminus \{u_1,\cdots,a_{k-1}\}$, if $u_j\in P_j\cap P_{j+1}$, then we delete $u_j$ in $\cup_{i} P_i$ once, this means $u_j\in (\cup_{i} P_i)\setminus \{u_1,\cdots,u_{k-1}\}$ in this case. By Lemma \ref{inone}, for each $1\leq j<k$, either $u_j\in P_j$ or $v_j\in P_j$, and either $u_j\in P_{j+1}$ or $w'_j\in P_{j+1}$. Thus one of the cases $v_j\in \iota((P_i)_i)$, $w'_j\in \iota((P_i)_i)$ and $u_j\in \iota((P_i)_i)$ happens. Hence $(\cup_{i} P_i)\setminus \{u_1,\cdots,u_{k-1}\}$ is a perfect matching of $G_{T,\gamma}$.

Now we prove $\iota^{-1}$ is the inverse of $\iota$.

For any $P\in \mathcal P(G_{T,\gamma})$, for any $1\leq j<k$, if $v_j\in P$, then $v_j\in \iota_j(P)$, and hence $v_j\in \iota^{-1}\iota(P)$; if $u_j\in P$, then $u_j\in \iota_j(P)\cap \iota_{j+1}(P)$, and hence $u_j\in \iota^{-1}\iota(P)$; if $w'_j\in P$, then $w'_j\in \iota_{j+1}(P)$, and hence $w'_j\in \iota^{-1}\iota(P)$. Therefore $\iota^{-1}\iota(P)=P$.

For any  $(P_i)_i \in\{(P_i)_{i}\;|\;P_i\in \mathcal P(G_{T,\gamma_i}), \forall j, u_j\in P_j\cup P_{j+1}\}$, for any $1\leq j<k$, one of the following cases happens: $v_j\in P_j, u_j\in P_{j+1}$ or $u_j\in P_j, w'_j\in P_{j+1}$ or $u_j\in P_j, u_j\in P_{j+1}$. If $v_j\in P_j, u_j\in P_{j+1}$, then $v_j\in \iota^{-1}((p_i)_i)$, and hence $v_j\in (\iota\iota^{-1}((P_i)_i))_j$ and $u_j\in (\iota\iota^{-1}((P_i)_i))_{j+1}$; if $u_j\in P_j, w'_j\in P_{j+1}$, then $w'_j\in \iota^{-1}((p_i)_i)$, and hence $u_j\in (\iota\iota^{-1}((P_i)_i))_{j}$ and $w'_j\in (\iota\iota^{-1}((P_i)_i))_{j+1}$; if $u_j\in P_j\cap P_{j+1}$, then $u_j\in \iota^{-1}((p_i)_i)$, and hence $u_j\in (\iota\iota^{-1}((P_i)_i))_{j}$ and $u_j\in (\iota\iota^{-1}((P_i)_i))_{j+1}$. Therefore $\iota\iota^{-1}((P_i)_i)=(P_i)_i$.

The proof of the proposition is complete.
\end{proof}

\medskip

According to this Proposition, we would write a perfect matching $P\in \mathcal P(G_{T,\gamma})$ as $(P_i)$ in the sequel. We arrange the order of $\tau$-equivalence classes in $G_{T,\gamma_i}$ according to the orientation of $\gamma$. By Lemma \ref{localdec}, for any perfect matching $P$ of $G_{T,\gamma}$, $P\in \mathcal P^{\tau}_{\nu}(G_{T,\gamma})$ for some sequence $\nu$. By the truncation of $\gamma$, we can truncate $\nu$ as $\nu^1,\cdots,\nu^k$ such that $P_i\in \mathcal P^{\tau}_{\nu^i}(G_{T,\gamma_i})$ for any $i\in [1,k]$. Similarly, we would write $P'\in \mathcal P(G_{T',\gamma})$ as $(P'_i)$ with $P'_i\in \mathcal P(G_{T',\gamma_i})$.

\medskip

\begin{Lemma}\label{formglobal}

Let $P=(P_i)\in \mathcal P(G_{T,\gamma})$. For any $(P'_i)$ such that $P'_i\in \psi_{\gamma_i}(P_i), i\in [1,k]$, $(P'_i)$ forms a perfect matching of $G_{T',\gamma}$. Moreover, if $u_j\notin P_j\cap P_{j+1}$ for some $j\in [1,k-1]$, then $u'_j\notin P'_j\cap P'_{j+1}$.

\end{Lemma}

\begin{proof}

For any $j\in [1,k-1]$, since $(P_i)\in \mathcal P(G_{T,\gamma})$, $u_j\in P_j\cup P_{j+1}$ by Proposition \ref{decompose}, and hence $u'_j\in P'_j\cup P'_{j+1}$   by Proposition \ref{compare-loc} (2.d). Therefore, by the dual version of Proposition \ref{decompose}, $(P'_i)$ forms a perfect matching of $G_{T',\gamma}$. By Proposition \ref{compare-loc} (2.d), if $u_j\notin P_j\cap P_{j+1}$ for some $j\in [1,k-1]$, then $u'_j\notin P'_j\cap P'_{j+1}$.
\end{proof}

\medskip

For a sequence $\nu=(\nu_l)\in S_{T,\gamma}$, we choose pairs of elements in the index set $[1,n^{\tau}(T,\gamma)]$ by induction on $n^{\tau}(T,\gamma)$, the length of $\nu$ as follows. Firstly, choose a pair $(\nu_s,\nu_t), s<t$ satisfies (1) $\nu_s\nu_t=-1$; (2) $\nu_l=0$ for all $s<l<t$; (3) $\nu_l\nu_s\geq 0$ for all $l<s$, and we call $\{s,t\}$ a \emph{$\nu$-pair}. Then delete $\nu_s,\nu_t$ from $(\nu_l)$, we get a sequence $(v'_l)$ of length less than $n^{\tau}(T,\gamma)$, and we do the same steps on $(\nu'_l)$ as on $(\nu_l)$. In particular, if $\{s,t\}$ does not exist in the first step, let the set of the $\nu$-pairs to be empty. It is easy to see the set of $\nu$-pairs equals to the set of $(-\nu)$-pairs.

\medskip

Let $P\in \mathcal P^{\tau}_{\nu}(G_{T,\gamma})$ for some $\nu$. When $\nu_s=-1$, the $s$-th $\tau$-equivalence class in $G_{T,\gamma}$ corresponds to a diagonal of a tile, denote as $G(p_{l_s})$, of $G_{T,\gamma}$. Similarly, denote by $G'(p_{l_t})$ the tile of $G_{T',\gamma}$ corresponds to the $t$-th $\tau'$-equivalence class.

\medskip

Let $\pi(P)$ be the subset of $\mathcal P(G_{T',\gamma})$ containing all $P'=(P'_i)$ such that $P'_i\in \psi_{\gamma_i}(P_i)$ and the edges in $P'\cap E(G'(p_{l_t}))$ and the edges in $P\cap E(G(p_{l_s}))$ have the same labels for any $\nu$-pair $\{s,t\}$ with $\nu_t=1$.

\medskip

Dually, change the roles of $\tau$ and $\tau'$, for any $P'\in \mathcal P^{\tau'}_{\nu}(G_{T',\gamma})$, let $\pi'(P')$ be the subset of $\mathcal P(G_{T,\gamma})$ containing all $P=(P_i)$ such that $P_i\in \psi'_{\gamma_i}(P'_i)$ and the edges in $P\cap E(G(p_{l_s}))$ and the edges in $P'\cap E(G'(p_{l_t}))$ have the same labels for any $\nu$-pair $\{s,t\}$ with $\nu_s=1$.

\medskip

We have the following characterization of $\pi(P)$ for $P\in \mathcal P^{\tau}_{\nu}(G_{T,\gamma})$.

\medskip

\begin{Proposition}\label{2^n}

Keep the foregoing notations. Assume $\tau'\neq \gamma \notin T$. Let $P$ be a perfect matching in $\mathcal P^{\tau}_{\nu}(G_{T,\gamma})$. Then $|\pi(P)|=2^{max\{0, \sum\nu_l\}}$. Precisely, let $S=\{s\in [1,\sum\nu_l]\mid \nu_s=1, s \text{ is not in a }\nu\text{-pair}\}$, for each $s\in S$, there are $2$ possibilities to choose the edges of $G'_{l_{s}}$ as $\tau'$-mutable pair in $P'$ for any $P'\in \pi(P)$. The other remaining edges are the same for any $P'\neq P''\in \pi(P)$.

\end{Proposition}

\begin{proof}

It follows by the construction of $\pi(P)$ and Lemma \ref{2^m}.
\end{proof}

\medskip

For any $P\in \mathcal P_{\nu}^{\tau}(G_{T,\gamma})$ with $\sum\nu_l\geq 0$. Denote $d=\sum\nu_l$. Choose $s_1,\cdots,s_d$ be the indices in order such that $\nu_{s_i}=1$ and $s_i$ is not in a $\nu$-pair for $i\in [1,d]$. By Proposition \ref{2^n}, for each $\lambda\in \{0,1\}^{\sum\nu_l}$, there uniquely exists $P'(\lambda)\in \pi(P)$ such that, for any $i\in [1,\sum\nu_l]$, the edges of $G'(p_{l_{s_i}})$ labeled $a_1,a_3$ are in $P'(\lambda)$ when $\lambda_i=1$ and the edges of $G'(p_{l_{s_i}})$ labeled $a_2,a_4$ are in $P'(\lambda)$ when $\lambda_i=0$, where $G'(p_{l_{s_i}})$ is the tile of $G_{T',\gamma}$ corresponding to the $s_i$-th $\tau'$-equivalence class. Dually, let $P'\in \mathcal P^{\tau'}_{\nu}(G_{T',\gamma})$ with $\sum \nu_l\geq 0$. Denote $d=\sum\nu_l$. Choose $s_1,\cdots,s_d$ be the indices in order such that $\nu_{s_i}=1$ and $s_i$ is not in a $\nu$-pair for $i\in [1,d]$. For each $\lambda\in \{0,1\}^d$, there uniquely exists $P(\lambda)\in \pi'(P')$ such that, for any $i\in [1,d]$, the edges of $G(p_{l_{s_i}})$ labeled $a_1,a_3$ are in $P(\lambda)$ when $\lambda_i=1$ and the edges of $G(p_{l_{s_i}})$ labeled $a_2,a_4$ are in $P(\lambda)$ when $\lambda_i=0$.

\medskip

Denote by $\mathcal P^{\tau}_{+}(G_{T,\gamma})$, $\mathcal P^{\tau}_{=0}(G_{T,\gamma})$ and $\mathcal P^{\tau}_{-}(G_{T,\gamma})$ the subset of $\mathcal P(G_{T,\gamma})$ which contain all $P\in \mathcal P^{\tau}_{\nu}(G_{T,\gamma})$ with $\sum \nu_l>0$, $\sum \nu_l=0$ and $\sum \nu_l<0$, respectively. Dually, we have the notions of $\mathcal P^{\tau'}_{+}(G_{T',\gamma})$, $\mathcal P^{\tau'}_{=0}(G_{T',\gamma})$ and $\mathcal P^{\tau'}_{-}(G_{T',\gamma})$. Thus $\pi(P)\subset \mathcal P^{\tau'}_{-}(G_{T',\gamma})$ for any $P\in \mathcal P^{\tau}_{+}(G_{T,\gamma})$, $\pi(P)\subset \mathcal P^{\tau'}_{=0}(G_{T',\gamma})$ for any $P\in \mathcal P^{\tau}_{=0}(G_{T,\gamma})$ and $\pi(P)\subset \mathcal P^{\tau'}_{+}(G_{T',\gamma})$ for any $P\in \mathcal P^{\tau}_{-}(G_{T,\gamma})$.

\medskip

\begin{Lemma}\label{retraction}

Keep the foregoing notations. Assume $\tau'\neq \gamma\notin T$. Let $P'=(P'_i)$ be a perfect matching in $\mathcal P(G_{T',\gamma})$. For any $P=(P_i)\in \pi'(P')$, we have $P'\in \pi(P)$.

\end{Lemma}

\begin{proof}

Assume $P'\in \mathcal P^{\tau'}_{\nu}(G_{T',\gamma})$ for some $\nu$, then $P\in \mathcal P^{\tau}_{-\nu}(G_{T,\gamma})$. By Proposition \ref{compare-loc} (2.b), $P'_i\in \psi_{\gamma_i}(P_i)$ for $i\in [1,k]$. Since $P=(P_i)\in \pi'(P')$, the edges in $P\cap E(G(p_{l_s}))$ and the edges in $P'\cap E(G'(p_{l_t}))$ have the same labels for any $\nu$-pair $\{s,t\}$ with $\nu_s=1$, equivalently, the edges in $P\cap E(G(p_{l_s}))$ and the edges in $P'\cap E(G'(p_{l_t}))$ have the same labels for any $(-\nu)$-pair $\{s,t\}$ with $-\nu_t=1$. Therefore, $P'\in \pi(P)$.
\end{proof}

\medskip

\begin{Proposition}\label{partition1}

Keep the foregoing notations. Assume $\tau'\neq \gamma \notin T$. We have $P\rightarrow \pi(P)$ gives partition bijections from $\mathcal P^{\tau}_{+}(G_{T,\gamma})$ to $\mathcal P^{\tau'}_{-}(G_{T',\gamma})$ and from $\mathcal P^{\tau}_{=0}(G_{T,\gamma})$ to $\mathcal P^{\tau'}_{=0}(G_{T',\gamma})$.

\end{Proposition}

\begin{proof}

For any $P=(P_i), Q=(Q_i)\in \mathcal P^{\tau}_{+}(G_{T,\gamma})\cup \mathcal P^{\tau}_{=0}(G_{T,\gamma})$ with $P\neq Q$. Assume $P\in \mathcal P^{\tau}_{\nu}(G_{T,\gamma})$ and $Q\in \mathcal P^{\tau}_{\nu'}(G_{T,\gamma})$ for some $\nu$ and $\nu'$ with $\sum\nu_l\geq 0$ and $\sum\nu'_l\geq 0$.

If $\nu\neq \nu'$, then $\nu^i\neq \nu'^{i}$ for some $i\in [1,k]$, and hence $\psi_{\gamma_i}(P)\cap \psi_{\gamma_i}(Q)=\emptyset$ by using Proposition \ref{compare-loc} (2.b). Hence $\pi(P)\cap \pi(Q)=\emptyset$.

If $\nu=\nu'$, then $P_i\neq Q_i$ for some $i\in [1,k]$ since $P\neq Q$. By Lemma \ref{diff2}, either $\psi_{\gamma_i}(P_i)\cap \psi_{\gamma_i}(Q_i)=\emptyset$ or there exists a tile $G$ of $G_{T,\gamma_i}$ with diagonal labeled $\tau$ such that $P$ and $Q$ can twist on $G$ and $P\cap E(G)\neq Q\cap E(G)$.

In case $\psi_{\gamma_i}(P_i)\cap \psi_{\gamma_i}(Q_i)=\emptyset$, $\pi_{i}(P_i)\cap \pi_{i}(Q_i)=\emptyset$, and hence $\pi(P)\cap \pi(Q)=\emptyset$. In the latter case, we may assume $G$ corresponds to the $t$-th $\tau$-equivalence class. We have $\nu_t=-1$. Since $\sum\nu_l\geq 0$, there exists $s$ such that $\{s,t\}$ is a $\nu$-pair. Since $P\cap E(G)\neq Q\cap E(G)$, according to the constructions of $\pi(P)$ and $\pi(Q)$, $\pi(P)\cap \pi(Q)=\emptyset$.

Using Lemma \ref{retraction}, we have $$\textstyle\bigcup_{P\in \mathcal P^{\tau}_{+}(G_{T,\gamma})} \pi(P)=\mathcal P^{\tau'}_{-}(G_{T',\gamma}),\;\;\; \textstyle\bigcup_{P\in \mathcal P^{\tau}_{=0}(G_{T,\gamma})} \pi(P)=\mathcal P^{\tau'}_{=0}(G_{T',\gamma}).$$

Therefore, $P\rightarrow \pi(P)$ induces the required partition bijections.
\end{proof}

\medskip

Dually, change the roles of $\tau$ and $\tau'$, we have the following.

\medskip

\begin{Proposition}\label{partition2}

Keep the foregoing notations. Assume $\tau'\neq \gamma \notin T$. We have $P'\rightarrow \pi'(P')$ gives partition bijections from $\mathcal P^{\tau'}_{+}(G_{T',\gamma})$ to $\mathcal P^{\tau}_{-}(G_{T,\gamma})$ and from $\mathcal P^{\tau'}_{=0}(G_{T',\gamma})$ to $\mathcal P^{\tau}_{=0}(G_{T,\gamma})$.

\end{Proposition}

\medskip

By Lemma \ref{retraction}, clearly $\pi'\pi(P)=P$ for all $P\in \mathcal P^{\tau}_{=0}(G_{T,\gamma})$.

\medskip

With the preparations, we now give the proof of Theorem \ref{partition bi} (1).

\medskip

{\bf Proof of Theorem \ref{partition bi} (1):} If $\gamma\in T$ or $\gamma=\tau'$, the result clearly holds. If $\tau'\neq \gamma\notin T$, since $\mathcal P(G_{T,\gamma})=\mathcal P^{\tau}_{+}(G_{T,\gamma})\bigsqcup \mathcal P^{\tau}_{=0}(G_{T,\gamma})\bigsqcup\mathcal P^{\tau}_{-}(G_{T,\gamma})$ and $\mathcal P(G_{T',\gamma})=\mathcal P^{\tau'}_{-}(G_{T',\gamma})\bigsqcup \mathcal P^{\tau'}_{=0}(G_{T',\gamma})\bigsqcup\mathcal P^{\tau'}_{+}(G_{T',\gamma})$, by Propositions \ref{partition1} and \ref{partition2}, our result follows at once. \;\;\; \ \ \ \ \ \ \ \ \ \;\;\; \ \ \ \ \ \ \  \;\;\; \ \ \ \ \ \ \ \ \ \;\;\; \ \ \ \ \ \ \ \ \ \ \;\;\; \ \ \ \ \ \ \ \;\;\; \ \ \ \ \ \ \ \ \ \ \ \ \ \ \ \ \ \ \ \ \ \ \ \  $\square$

\medskip

\subsection{Proof of the Theorem \ref{partition bi} (2)}\label{main12}

In this subsection, we prove Theorem \ref{partition bi} (2).

\begin{Lemma}\label{basic1}

\begin{enumerate}[$(1)$]

  \item Let $Q\in \mathcal P_{\nu}^{\tau}(G_{T,\gamma})$ with $\sum\nu_l\geq 0$. Let $Q'\in \pi(Q)$ corresponding to some $\lambda\in \{0,1\}^{\sum\nu_l}$ as in Proposition \ref{2^n}.

  \begin{enumerate}[$(a)$]

    \item If $\lambda=(0,0,\cdots,0)$, then $\bigoplus_{P'\in \pi(Q)}y^{T'}(P')=y^{T'}(Q')\cdot(1\oplus y^{T'}_{\tau'})^{\sum\nu_l}$ in $\mathbb P$.

    \item If $\lambda=(1,1,\cdots,1)$, then $\bigoplus_{P'\in \pi(Q)}y^{T'}(P')=y^{T'}(Q')\cdot(1\oplus y^{T}_{\tau})^{\sum\nu_l}$ in $\mathbb P$.

  \end{enumerate}

  \item Let $Q'\in \mathcal P_{\nu}^{\tau'}(G_{T',\gamma})$ with $\sum\nu_l\geq 0$. Let $Q\in \pi'(Q')$ corresponding to some $\lambda\in \{0,1\}^{\sum\nu_l}$ as in Proposition \ref{2^n}.

  \begin{enumerate}[$(a)$]

    \item If $\lambda=(0,0,\cdots,0)$, then $\bigoplus_{P\in \pi'(Q')}y^{T}(P)=y^{T}(Q)\cdot(1\oplus y^{T'}_{\tau'})^{\sum\nu_l}$ in $\mathbb P$.

    \item If $\lambda=(1,1,\cdots,1)$, then $\bigoplus_{P\in \pi'(Q')}y^{T}(P)=y^{T}(Q)\cdot(1\oplus y^{T}_{\tau})^{\sum\nu_l}$ in $\mathbb P$.

  \end{enumerate}

\end{enumerate}

\end{Lemma}

\begin{proof}

We shall only prove (1) because (2) can be proved dually. In case $\lambda=(0,0,\cdots,0)$, by Lemma \ref{brick1}, $y^{T'}(P')=y^{T'}(Q')\cdot(y^{T'}_{\tau'})^{\sum\lambda'_i}$ for any $P'\in \pi(Q)$ corresponding to $\lambda'$. Thus our result follows in this case. In case $\lambda=(1,1,\cdots,1)$, by Lemma \ref{brick1}, $y^{T'}(P')=y^{T'}(Q')\cdot(y^{T'}_{\tau'})^{\sum\lambda'_i-\sum\nu_l}$ for any $P'\in \pi(Q)$ corresponding to $\lambda'$. Thus $\bigoplus_{P'\in \pi(Q)}y^{T'}(P')=y^{T'}(Q')\cdot(1\oplus (y^{T'}_{\tau'})^{-1})^{\sum\nu_l}$. Since $y_{\tau'}^{T'}=(y^{T}_{\tau})^{-1}$, our result follows.
\end{proof}

\medskip

\begin{Lemma}\label{non-a-F}

Keep the foregoing notations. Assume $\tau'\neq \gamma\notin T$. Suppose that $P\in \mathcal P^{\tau}_{\nu}(G_{T,\gamma})$ can twist on a tile $G(p)$ with diagonal labeled $\varsigma$. Assume $S_1,S_2\in \mathfrak P$ such that $Q\in S_1$ and $\mu_{p}Q\in S_2$. If $\varsigma\neq a_1,a_2,a_3,a_4,\tau$, then, in $\mathbb P$, we have
$$\frac{\bigoplus_{P\in S_1}y^{T}(P)}{\bigoplus_{P'\in \pi(S_1)}y^{T'}(P')}=\frac{\bigoplus_{P\in S_2}y^{T}(P)}{\bigoplus_{P'\in \pi(S_2)}y^{T'}(P')}.$$

\end{Lemma}

\begin{proof}

Suppose first that $\sum\nu_l\geq 0$. Then $S_1=\{Q\}$ and $S_2=\{\mu_pQ\}$. Let $Q'\in \pi(Q)$ corresponding to $(0,0,\cdots,0)\in \{0,1\}^{\sum\nu_l}$. Then $\mu_pQ'\in \pi(\mu_pQ)$ corresponds to $(0,0,\cdots,0)\in \{0,1\}^{\sum\nu_l}$. By Lemma \ref{basic1},
$$\textstyle\bigoplus_{P'\in \pi(S_1)}y^{T'}(P')=y^{T'}(Q')\cdot(1\oplus y_{\tau'}^{T'})^{\sum\nu_l},$$
and
$$\textstyle\bigoplus_{P'\in \pi(S_2)}y^{T'}(P')=y^{T'}(\mu_{p}Q')\cdot(1\oplus y_{\tau'}^{T'})^{\sum\nu_l}.$$
By Lemma \ref{brick1}, $\frac{y^{T'}(\mu_pQ')}{y^{T'}(Q')}=(y^{T'}_{\varsigma})^{\pm 1}$ and $\frac{y^{T}(\mu_pQ)}{y^{T}(Q)}=(y^{T}_{\varsigma})^{\pm 1}$. Since $\varsigma\neq a_1,a_2,a_3,a_4,\tau$, $b^T_{\tau \varsigma}=0$ and hence $y^{T'}_{\varsigma}=y^{T}_{\varsigma}$. Our result follows in this case.

Now, suppose that $\sum\nu_l<0$. Assume that $\pi(S_1)=\{Q'\}$, then $\pi(S_2)=\{\mu_pQ'\}$. Change the roles of $T$ and $T'$, as the discussion for the case
$\sum\nu_l\geq 0$, our result follows in this case.
\end{proof}

\medskip

\begin{Lemma}\label{tau-F}

Keep the foregoing notations. Assume $\tau'\neq \gamma\notin T$. Suppose that $P\in \mathcal P^{\tau}_{\nu}(G_{T,\gamma})$ can twist on a tile $G(p_{l_t})$ with diagonal labeled $\tau$. Assume $S_1,S_2\in \mathfrak P$ such that $Q\in S_1$ and $\mu_{p_{l_t}}Q\in S_2$. Then, in $\mathbb P$, we have
$$\frac{\bigoplus_{P\in S_1}y^{T}(P)}{\bigoplus_{P'\in \pi(S_1)}y^{T'}(P')}=\frac{\bigoplus_{P\in S_2}y^{T}(P)}{\bigoplus_{P'\in \pi(S_2)}y^{T'}(P')}.$$

\end{Lemma}

\begin{proof}

Suppose first that $\sum\nu_l\geq 0$. Then $S_1=\{Q\}$ and $S_2=\{\mu_{p_{l_t}}Q\}$. Since $\sum\nu_l\geq 0$, there exists $s$ such that $\{s,t\}$ is a $\nu$-pair. Thus, $\nu_s=1$. Let $Q'\in \pi(Q)$ corresponding to $(0,0,\cdots,0)\in \{0,1\}^{\sum\nu_l}$. Since $Q'\in \mathcal P^{\tau'}_{-\nu}(G_{T',\gamma})$ and $-\nu_s=-1$, the $s$-th $\tau'$-equivalence class in $G_{T',\gamma}$ corresponds to the tile $G'(p_{l_s})$ with diagonal labeled $\tau'$. Then $Q'$ can twist on $G'_{p_{l_s}}$, and $\mu_{p_{l_s}}Q'\in \pi(\mu_{p_{l_t}}Q)$ corresponds to $(0,0,\cdots,0)\in \{0,1\}^{\sum\nu_l}$. By Lemma \ref{basic1},
$$\textstyle\bigoplus_{P'\in \pi(S_1)}y^{T'}(P')=y^{T'}(Q')\cdot(1\oplus y_{\tau'}^{T'})^{\sum\nu_l},$$
and
$$\textstyle\bigoplus_{P'\in \pi(S_2)}y^{T'}(P')=y^{T'}(\mu_{p_{l_s}}Q')\cdot(1\oplus y_{\tau'}^{T'})^{\sum\nu_l}.$$
By Lemma \ref{brick1}, $\frac{y^{T'}(\mu_{p_{l_s}}Q')}{y^{T'}(Q')}=(y^{T'}_{\tau'})^{\pm 1}$ and $\frac{y^{T}(\mu_{p_{l_t}}Q)}{y^{T}(Q)}=(y^{T}_{\tau})^{\mp 1}$. Since $y^{T'}_{\tau'}=(y^{T}_{\tau})^{-1}$. Our result follows in this case.

Now, suppose that $\sum\nu_l<0$. Assume that $\pi(S_1)=\{Q'\}$. If $t$ is not in a $\nu$-pair, $\mu_{p_{l_t}}Q\in S_1$ and $S_1=S_2$. Our result clearly holds in this case. If $t$ is in a $\nu$-pair $\{s,t\}$, then $\nu_s=1$. Since $Q'\in \mathcal P^{\tau'}_{-\nu}(G_{T',\gamma})$ and $-\nu_s=-1$, the $s$-th $\tau'$-equivalence class in $G_{T',\gamma}$ corresponds to the tile $G'(p_{l_s})$ with diagonal labeled $\tau'$. We have $S_2=\{\mu_{p_{l_s}}Q'\}$. Change the roles of $T$ and $T'$, as the discussion for the case $\sum\nu_l\geq 0$, our result follows in this case.
\end{proof}

\medskip

\begin{Lemma}\label{special-a-F}

Keep the foregoing notations. Assume $\tau'\neq \gamma\notin T$. Let $Q$ be a perfect matching of $G_{T,\gamma}$ which can twist on a tile $G(p)$ with diagonal labeled $a=a_q$ for $q=1,2,3,4$. Assume $S_1,S_2\in \mathfrak P$ such that $Q\in S_1$ and $\mu_{p}Q\in S_2$. If all the $\tau$-mutable edges pairs in $Q$ are labeled $a_{q-1},a_{q+1}$, then, in $\mathbb P$, we have
$$\frac{\bigoplus_{P\in S_1}y^{T}(P)}{\bigoplus_{P'\in \pi(S_1)}y^{T'}(P')}=\frac{\bigoplus_{P\in S_2}y^{T}(P)}{\bigoplus_{P'\in \pi(S_2)}y^{T'}(P')}.$$

\end{Lemma}

\begin{proof}

We may assume the edge labeled $\tau$ of $G(p)$ is in $Q$. Write $Q$ as $(Q_i)$ and assume $G(p)$ is a tile of $G_{T,\gamma_j}$ for some $j$. Thus $\mu_pQ=(\mu_pQ_i)$ with $\mu_pQ_i=Q_i$ for $i\neq j$. By Lemma \ref{com}, there exists $Q'_j\in \psi_{\gamma_j}(Q_j)$ satisfies (a) $Q'_j$ can twist on $G'(p)$, (b) $\mu_pQ'_j\in \psi_{\gamma_j}(\mu_pQ_j)$, (c) all $\tau'$-mutable edges pairs in $Q'_j$ are labeled $a_{q-1},a_{q+1}$. For each $i\neq j$, since all the $\tau$-mutable edges pairs in $Q$ are labeled $a_{q-1},a_{q+1}$, by Lemma \ref{2^m}, we can choose $Q_i'\in \psi_{\gamma_i}(Q_i)$ such that all $\tau'$-mutation edges pairs in $Q'_i$ are labeled $a_{q-1},a_{q+1}$. Since all $\tau$-mutable edges pairs in $Q$ are labeled $a_{q-1},a_{q+1}$, all $\tau$-mutable edges pairs in $\mu_pQ$ are labeled $a_{q-1},a_{q+1}$. By the constructions of $\pi(Q)$ and $\pi(\mu_pQ)$, we have $Q'=(Q'_i)\in \pi(Q)$ and $Q''=(\mu_pQ'_i)\in \pi(\mu_pQ)$, here we mean $\mu_pQ'_i=Q'_i$ for $i\neq j$. Clearly, $\mu_{p}Q'=Q''$. Assume $\mu_pQ\in \mathcal P^{\tau}_{\nu'}(G_{T,\gamma})$. Clearly, $\sum\nu_l>\sum\nu'_l$ and $|b^{T}_{\tau a_q}|=\sum\nu_l-\sum\nu'_l$.

Suppose that $\sum\nu_l\geq 0$ and $\sum\nu'_l\geq 0$. Then $S_1=\{Q\}$ and $S_2=\{\mu_{p}Q\}$.

If $a=a_1,a_3$, since all mutation edges pairs in $Q'$ and $Q''$ are labeled $a_2,a_4$, $Q'$ and $Q''$ corresponds to $(0,0,\cdots,0)\in \{0,1\}^{\sum\nu_l}$ and $(0,0,\cdots,0)\in \{0,1\}^{\sum\nu'_l}$, respectively. By Lemma \ref{basic1},
$$\textstyle\bigoplus_{P'\in \pi(S_1)}y^{T'}(P')=y^{T'}(Q')\cdot(1\oplus y_{\tau'}^{T'})^{\sum\nu_l},$$
and
$$\textstyle\bigoplus_{P'\in \pi(S_2)}y^{T'}(P')=y^{T'}(Q'')\cdot(1\oplus y_{\tau'}^{T'})^{\sum\nu'_l}.$$
By Lemma \ref{brick1}, $\frac{y^{T'}(Q')}{y^{T'}(Q'')}=y^{T'}_{a_q}$ and $\frac{y^{T}(Q)}{y^{T}(\mu_p Q)}=y^{T}_{a_q}$. Since $y^{T'}_{\tau'}=(y^{T}_{\tau})^{-1}$ and
$$y^{T'}_{a_q}=y^{T}_{a_q}\cdot (y^T_{\tau})^{\sum\nu_l-\sum\nu'_l}\cdot(1\oplus y^{T}_{\tau})^{\sum\nu'_l-\sum\nu_l},$$
$y_{a_q}^{T}=y_{a_q}^{T'}\cdot(1\oplus y^{T'}_{\tau'})^{\sum\nu_l-\sum\nu'_l}$.
Our result follows in this case.

If $a=a_2,a_4$, since all mutation edges pairs in $Q'$ and $Q''$ are labeled $a_1,a_3$, $Q'$ and $Q''$ corresponds to $(1,1,\cdots,1)\in \{0,1\}^{\sum\nu_l}$ and $(1,1,\cdots,1)\in \{0,1\}^{\sum\nu'_l}$, respectively. By Lemma \ref{basic1},
$$\textstyle\bigoplus_{P'\in \pi(S_1)}y^{T'}(P')=y^{T'}(Q')\cdot(1\oplus y_{\tau}^{T})^{\sum\nu_l},$$
and
$$\textstyle\bigoplus_{P'\in \pi(S_2)}y^{T'}(P')=y^{T'}(Q'')\cdot(1\oplus y_{\tau}^{T})^{\sum\nu'_l}.$$
By Lemma \ref{brick1}, $\frac{y^{T'}(Q'')}{y^{T'}(Q')}=y^{T'}_{a_q}$ and $\frac{y^{T}(\mu_p Q)}{y^{T}(Q)}=y^{T}_{a_q}$. Since $y^{T'}_{a_q}=y^{T}_{a_q}\cdot(1\oplus y^{T}_{\tau})^{\sum\nu_l-\sum\nu'_l},$ our result follows in this case.

Suppose that $\sum\nu_l\geq 0$ and $\sum\nu'_l< 0$. Then $S_1=\{Q\}$ and $\pi(S_2)=\{Q''\}$.

If $q=1,3$, since all mutation edges pairs in $Q'$ and $\mu_pQ$ are labeled $a_2,a_4$, $Q'$ and $\mu_pQ$ corresponds to $(0,0,\cdots,0)\in \{0,1\}^{\sum\nu_l}$ and $(0,0,\cdots,0)\in \{0,1\}^{-\sum\nu'_l}$, respectively. By Lemma \ref{basic1},
$$\textstyle\bigoplus_{P'\in \pi(S_1)}y^{T'}(P')=y^{T'}(Q')\cdot(1\oplus y_{\tau'}^{T'})^{\sum\nu_l},$$
and
$$\textstyle\bigoplus_{P\in S_2}y^{T}(P)=y^{T}(\mu_pQ)\cdot(1\oplus y_{\tau'}^{T'})^{-\sum\nu'_l}.$$
By Lemma \ref{brick1}, $\frac{y^{T'}(Q')}{y^{T'}(Q'')}=y^{T'}_{a_q}$ and $\frac{y^{T}(Q)}{y^{T}(\mu_p Q)}=y^{T}_{a_q}$. Since $y^{T'}_{\tau'}=(y^{T}_{\tau})^{-1}$ and
$$y^{T'}_{a_q}=y^{T}_{a_q} \cdot(y^T_{\tau})^{\sum\nu_l-\sum\nu'_l}\cdot(1\oplus y^{T}_{\tau})^{\sum\nu'_l-\sum\nu_l},$$
$y_{a_q}^{T}=y_{a_q}^{T'}\cdot(1\oplus y^{T'}_{\tau'})^{\sum\nu_l-\sum\nu'_l}$, and hence
$$\frac{y^{T}(\mu_pQ)\cdot(1\oplus y_{\tau'}^{T'})^{-\sum\nu'_l}}{y^{T}(Q)}=\frac{y^{T'}(Q'')}{y^{T'}(Q')\cdot(1\oplus y^{T'}_{\tau'})^{\sum\nu_l}}.$$
Our result follows in this case.

If $q=2,4$, since all mutation edges pairs in $Q'$ and $\mu_pQ$ are labeled $a_1,a_3$, $Q'$ and $\mu_pQ$ corresponds to $(1,1,\cdots,1)\in \{0,1\}^{\sum\nu_l}$ and $(1,1,\cdots,1)\in \{0,1\}^{-\sum\nu'_l}$, respectively. By Lemma \ref{basic1},
$$\textstyle\bigoplus_{P'\in \pi(S_1)}y^{T'}(P')=y^{T'}(Q')\cdot(1\oplus y_{\tau}^{T})^{\sum\nu_l},$$
and
$$\textstyle\bigoplus_{P\in S_2}y^{T}(P)=y^{T}(\mu_pQ)\cdot(1\oplus y_{\tau}^{T})^{-\sum\nu'_l}.$$
By Lemma \ref{brick1}, $\frac{y^{T'}(Q'')}{y^{T'}(Q')}=y^{T'}_{a_q}$ and $\frac{y^{T}(\mu_p Q)}{y^{T}(Q)}=y^{T}_{a_q}$. Since $y^{T'}_{a_q}=y^{T}_{a_q}\cdot(1\oplus y^{T}_{\tau})^{\sum\nu_l-\sum\nu'_l},$ our result follows in this case.

Suppose that $\sum\nu_l< 0$ and $\sum\nu'_l< 0$. Change the roles of $T$ and $T'$, as the discussion for the case $\sum\nu_l,\sum\nu'_l\geq 0$, our result follows in this case.
\end{proof}

\medskip

\begin{Lemma}\label{all-a-F}

Keep the foregoing notations. Assume $\tau'\neq \gamma\notin T$. Let $Q$ be a perfect matching of $G_{T,\gamma}$ which can twist on a tile $G(p)$ with diagonal labeled $a_q$ for $q=1,2,3,4$. Assume $S_1,S_2\in \mathfrak P$ such that $Q\in S_1$ and $\mu_{p}Q\in S_2$. Then, in $\mathbb P$, we have
$$\frac{\bigoplus_{P\in S_1}y^{T}(P)}{\bigoplus_{P'\in \pi(S_1)}y^{T'}(P')}=\frac{\bigoplus_{P\in S_2}y^{T}(P)}{\bigoplus_{P'\in \pi(S_2)}y^{T'}(P')}.$$

\end{Lemma}

\begin{proof}

After do twists from $Q$ on the tiles $G$ with diagonals labeled $\tau$ and the edges labeled $a_{q},a_{q+2}$ are in $Q$, we can obtain a perfect matching $R$ which satisfies the conditions of Lemma \ref{special-a-F}. Then our result follows by Lemmas \ref{special-a-F} and \ref{tau-F}.
\end{proof}

\medskip

In summary, we have the following proposition.

\medskip

\begin{Proposition}\label{all-F}

Keep the foregoing notations. Assume $\tau'\neq \gamma\notin T$. Let $Q$ be a perfect matching of $G_{T,\gamma}$ which can twist on a tile $G(p)$. Assume $S_1,S_2\in \mathfrak P$ such that $Q\in S_1$ and $\mu_{p}Q\in S_2$. Then, in $\mathbb P$, we have
$$\frac{\bigoplus_{P\in S_1}y^{T}(P)}{\bigoplus_{P'\in \pi(S_1)}y^{T'}(P')}=\frac{\bigoplus_{P\in S_2}y^{T}(P)}{\bigoplus_{P'\in \pi(S_2)}y^{T'}(P')}.$$

\end{Proposition}

\begin{proof}

This follows immediately by Lemma \ref{non-a-F}, Lemma \ref{tau-F} and Lemma \ref{all-a-F}.
\end{proof}

\medskip

As a corollary of Proposition \ref{all-F}, we have the following theorem.

\medskip

\begin{Theorem}\label{F}

Keep the foregoing notations. Assume $\tau'\neq \gamma\notin T$.

\begin{enumerate}[$(1)$]

  \item For any $S\in \mathfrak P$, we have
  $$\frac{\bigoplus_{P\in S}y^{T}(P)}{\bigoplus_{P'\in \pi(S)}y^{T'}(P')}=\frac{\bigoplus_{P\in \mathcal P(G_{T,\gamma})}y^{T}(P)}{\bigoplus_{P'\in \mathcal P(G_{T',\gamma})}y^{T'}(P')}.$$

  \item For any $S'\in \mathfrak P'$, we have
  $$\frac{\bigoplus_{P\in \pi'(S')}y^{T}(P)}{\bigoplus_{P'\in S'}y^{T'}(P')}=\frac{\bigoplus_{P\in \mathcal P(G_{T,\gamma})}y^{T}(P)}{\bigoplus_{P'\in \mathcal P(G_{T',\gamma})}y^{T'}(P')}.$$

\end{enumerate}

\end{Theorem}

\begin{proof}

We shall only prove (1) because (2) can be proved dually. By Lemma \ref{transitive} and Proposition \ref{all-F}, $\frac{\bigoplus_{P\in S_1}y^{T}(P)}{\bigoplus_{P'\in \pi(S_1)}y^{T'}(P')}=\frac{\bigoplus_{P\in S_2}y^{T}(P)}{\bigoplus_{P'\in \pi(S_2)}y^{T'}(P')}$ for any $S_1,S_2\in \mathfrak P$. Therefore, our result follows.
\end{proof}

\medskip

{\bf Proof of Theorem \ref{partition bi} (2):}

If $\tau\neq\gamma\in T$, our result clearly holds since $x^T_{\gamma}=x^{T'}_{\gamma}$. If $\gamma=\tau,\tau'$, our result follows by $x^{T'}_{\tau'}=\frac{y^T_{\tau}\cdot x^T_{a_2}\cdot x^T_{a_4}+x^T_{a_1}\cdot x^T_{a_3}}{x^T_{\tau}\cdot(1\oplus y^T_{\tau})}$.

If $\tau'\neq \gamma \notin T$, we may assume that $|\pi(S)|=1$ and $\pi(S)=\{Q'\}$. Then $Q'\in \mathcal P^{\tau'}_{\nu}(G_{T',\gamma})$ for some $\nu$ with $\sum\nu_l\geq 0$. Denote $d=\sum\nu_l$. Then
$$x^{T'}(Q')=\frac{A\cdot(x^{T'}_{\tau'})^d\cdot y^{T'}(Q')}{\bigoplus_{P'\in \mathcal P(G_{T',\gamma})}y^{T'}(P')},$$ where $A$ is the cluster monomial corresponding to the edges not labeled $\tau'$ of $Q'$.

On the other hand, by Propositions \ref{compare-loc} (2.c), \ref{2^n} and Lemma \ref{brick1}, we have

$$
\begin{array}{rcl} & & \textstyle\sum_{P\in \pi'(Q')}x^T(P)\vspace{2pt}\\

& = & \frac{\textstyle\sum_{\lambda\in \{0,1\}^d}A\cdot (x^T_{\tau})^{-d}\cdot (x^T_{a_2}x^T_{a_4})^{d-\sum\lambda_i}\cdot(x^T_{a_1}x^T_{a_3})^{\sum\lambda_i}\cdot y^T(Q(\lambda))}{\bigoplus_{P\in \mathcal P(G_{T,\gamma})}y^T(P)}  \vspace{4pt}  \\

& = & A\cdot (x^T_{\tau})^{-d}\frac{\textstyle\sum_{\lambda\in \{0,1\}^d}y^T(Q_0)\cdot (y^T_{\tau})^{d-\sum\lambda_i}\cdot (x^T_{a_2}x^T_{a_4})^{d-\sum\lambda_i}\cdot(x^T_{a_1}x^T_{a_3})^{\sum\lambda_i}}{\bigoplus_{P\in \mathcal P(G_{T,\gamma})}y^T(P)}  \vspace{4pt}  \\

& = & A\cdot (x^T_{\tau})^{-d}\cdot y^T(Q_0) \frac{\textstyle\sum_{\lambda\in \{0,1\}^d}(y^T_{\tau}x^T_{a_2}x^T_{a_4})^{d-\sum\lambda_i}\cdot(x^T_{a_1}x^T_{a_3})^{\sum\lambda_i}}{\bigoplus_{P\in \mathcal P(G_{T,\gamma})}y^T(P)}  \vspace{4pt}  \\

& = & A\cdot (x^T_{\tau})^{-d}\cdot y^T(Q_0) \frac{(y^T_{\tau}x^T_{a_2}x^T_{a_4}+x^T_{a_1}x^T_{a_3})^d}{\bigoplus_{P\in \mathcal P(G_{T,\gamma})}y^T(P)},  \vspace{2pt}  \\
\end{array}$$
where $Q(\lambda)\in \pi'(Q')$ corresponds to $\lambda\in \{0,1\}^d$ and $Q_0\in \pi'(Q')$ corresponds to $(1,1,\cdots,1)\in \{0,1\}^d$.

Since $x^{T'}_{\tau'}=\frac{y^T_{\tau}\cdot x^T_{a_2}\cdot x^T_{a_4}+x^T_{a_1}\cdot x^T_{a_3}}{x^T_{\tau}\cdot(1\oplus y^T_{\tau})}$, $\sum_{P\in \pi'(Q')}x^T(P)= \frac{A\cdot (x^{T'}_{\tau'})^d \cdot y^T(Q_0)\cdot (1\oplus y^T_{\tau})^d}{\bigoplus_{P\in \mathcal P(G_{T,\gamma})}y^T(P)}$. Thus, by Theorem \ref{F},

$$
\begin{array}{rcl} \frac{\textstyle\sum_{P\in \pi'(Q')}x^T(P)}{x^{T'}(Q')}

& = & \frac{y^T(Q_0)\cdot(1\oplus y^T_{\tau})^d}{y^{T'}(Q')}\cdot \frac{\bigoplus_{P'\in \mathcal P(G_{T',\gamma})}y^{T'}(P')} {\bigoplus_{P\in \mathcal P(G_{T,\gamma})}y^T(P)}  \vspace{4pt}  \\

& = & \frac{y^T(Q_0)\cdot(1\oplus y^T_{\tau})^d}{y^{T'}(Q')}\cdot \frac{y^{T'}(Q')} {\bigoplus_{Q\in \pi'(Q')}y^T(Q)}  \vspace{4pt}  \\

& = & \frac{y^T(Q_0)\cdot(1\oplus y^T_{\tau})^d}{y^{T'}(Q')}\cdot \frac{y^{T'}(Q')} {y^T(Q_0)\cdot(1\oplus y^T_{\tau})^d}  \vspace{4pt}  \\

& = & 1. \vspace{2pt}
\end{array}$$
Our result follows. \;\;\; \ \ \ \ \ \ \ \ \ \;\;\; \ \ \ \ \ \ \  \;\;\; \ \ \ \ \ \ \ \ \ \;\;\; \ \ \ \ \ \ \ \ \ \ \;\;\; \ \ \ \ \ \ \ \;\;\; \ \ \ \ \ \ \ \ \ \ \ \ \ \ \ \ \ \ \ \ \ \ \ \  $\square$

\medskip

\section{Proof of Theorem \ref{mainthm}}\label{main2}

\subsection{Valuation maps on \textbf{$\mathcal P(G_{T,\gamma})$}}

Herein, two valuation maps $v_{+},v_{-}: \mathcal P(G_{T,\gamma})\rightarrow \mathbb Z$ are constructed. We will prove that $v_{+}$ coincide with $v_{-}$ and $v$ as in Theorem \ref{mainthm}. Throughout this section let $\mathcal O$ be an unpunctured surface and $T$ be a triangulation. Let $\gamma$ be an oriented arc in $\mathcal O$ crossing $T$ with points $p_1,\cdots, p_d$ in order. Assume that $p_1,\cdots,p_d$ belong to the arcs $\tau_{i_1},\cdots,\tau_{i_d}$, respectively in $T$.

\medskip

\begin{Lemma}\label{invo}

If $P\in \mathcal P(G_{T,\gamma})$ can twist on $G(p_s)$, then $\mu_{p_s}P$ can twist on $G(p_s)$, and $\mu_{p_s}\mu_{p_s}P=P$.

\end{Lemma}

\begin{proof}

Since $G(p_s)$ has two edges belong to $\mu_{p_s}P$ according to the definition of $\mu_{p_s}P$, $\mu_{p_s}P$ can twist on $G(p_s)$. $\mu_{p_s}\mu_{p_s}P=P$ is clear.
\end{proof}

\medskip

\begin{Lemma}\label{commuta}

If $P\in \mathcal P(G_{T,\gamma})$ can twist on $G(p_s)$ and $G(p_t)$ for some $s,t$ such that $|s-t|>1$, then
\begin{enumerate}[$(1)$]

  \item $\mu_{p_s}P$ and $\mu_{p_t}P$ can twist on $G(p_t)$ and $G(p_s)$, respectively, and $$\mu_{p_s}\mu_{p_t}P=\mu_{p_t}\mu_{p_s}P.$$

  \item $\Omega(p_s,P)+\Omega(p_t,\mu_{p_s}P)=\Omega(p_t,P)+\Omega(p_s,\mu_{p_t}P)$.

\end{enumerate}

\end{Lemma}

\begin{proof}

(1)~ Since $|s-t|>1$, the edges of $G(p_t)$ in $P$ are also in $\mu_{p_s}P$, and hence $\mu_{p_s}P$ can twist on $G(p_t)$. Similarly, $\mu_{p_t}P$ can twist on $G(p_s)$. By definition, $\mu_{p_s}\mu_{p_t}P$ and $\mu_{p_t}\mu_{p_s}P$ are obtained from $P$ via replacing the edges of $G(p_s)$ and $G(p_t)$ by the remaining edges simultaneously. Thus, $\mu_{p_s}\mu_{p_t}P=\mu_{p_t}\mu_{p_s}P$.

(2)~ Since $\Omega(p,Q)=-\Omega(p,\mu_{p}Q)$ for any $Q$ which can twist on $G(p)$, we may assume the edges of $G(p_s)$ and $G(p_t)$ which are labeled $a_{2_s},a_{4_s}$ and $a_{2_t},a_{4_t}$, respectively, in Definition \ref{omega} belong to $P$.

In case $\tau_{i_s}$ and $\tau_{i_t}$ are not in the same triangle in $T$ or $\tau_{i_s}=\tau_{i_t}$, we have $n^{\pm}_{p_s}(\tau_{i_s},P)=n^{\pm}_{p_s}(\tau_{i_s},\mu_{p_t}P)$, and $n^{\pm}_{p_t}(\tau_{i_t},P)=n^{\pm}_{p_t}(\tau_{i_t},\mu_{p_s}P)$.

Therefore,
$$
\begin{array}{rcl} && \Omega(p_s,P)+\Omega(p_t,\mu_{p_s}P) \vspace{2pt}  \\

& = & [n^{+}_{p_s}(\tau_{i_s},P)-m^{+}_{p_s}(\tau_{i_s},\gamma)-n^{-}_{p_s}(\tau_{i_s},P)+m^{-}_{p_s}(\tau_{i_s},\gamma)]d^T \vspace{2pt}  \\

& + & [n^{+}_{p_t}(\tau_{i_t},\mu_{p_s}P)-m^{+}_{p_t}(\tau_{i_t},\gamma)-n^{-}_{p_t}(\tau_{i_t},\mu_{p_s}P)+m^{-}_{p_t}(\tau_{i_t},\gamma)]d^T \vspace{2pt}  \\

& = & [n^{+}_{p_s}(\tau_{i_s},\mu_{p_t}P)-m^{+}_{p_s}(\tau_{i_s},\gamma)-n^{-}_{p_s}(\tau_{i_s},\mu_{p_t}P)+m^{-}_{p_s}(\tau_{i_s},\gamma)]d^T \vspace{2pt}  \\

& + & [n^{+}_{p_t}(\tau_{i_t},P)-m^{+}_{p_t}(\tau_{i_t},\gamma)-n^{-}_{p_t}(\tau_{i_t},P)+m^{-}_{p_t}(\tau_{i_t},\gamma)]d^T \vspace{2pt}  \\

& = & \Omega(p_s,\mu_{p_t}P)+\Omega(p_t,P).\vspace{2pt}\\
\end{array}$$

Since $\Omega(p,Q)=-\Omega(p,\mu_{p}Q)$ for any $Q$ which can twist on $G(p)$,
$$\Omega(p_s,P)+\Omega(p_t,\mu_{p_s}P)=\Omega(p_t,P)+\Omega(p_s,\mu_{p_t}P)$$ holds if and only if $$\Omega(p_s,\mu_{p_s}P)+\Omega(p_t,P)=\Omega(p_t,\mu_{p_s}P)+\Omega(p_s,\mu_{p_s}\mu_{p_t}P)$$ holds if and only if
$$\Omega(p_s,\mu_{p_t}P)+\Omega(p_t,\mu_{p_s}\mu_{p_t}P)=\Omega(p_t,\mu_{p_t}P)+\Omega(p_s,P)$$ holds.

In case $\tau_{i_s}$ and $\tau_{i_t}$ are in a same triangle in $T$, we therefore shall assume that the edges labeled $\tau_{i_t}$ and $\tau_{i_s}$ of $G(p_s)$ and $G(p_t)$, respectively are in $P$, since otherwise we can do twists for $P$ to make our assumption holds. Without loss of generality, assume that $s<t$ and $\tau_{i_t}$ is counterclockwise to $\tau_{i_s}$ in the triangle of $T$. Then we have $n_{p_s}^{+}(\tau_{i_s},\mu_{p_t}P)=n_{p_s}^{+}(\tau_{i_s},P)-c$, $n_{p_s}^{-}(\tau_{i_s},\mu_{p_t}P)=n_{p_s}^{-}(\tau_{i_s},P)$ and $n_{p_t}^{+}(\tau_{i_t},\mu_{p_s}P)=n_{p_t}^{+}(\tau_{i_t},P)$, $n_{p_t}^{-}(\tau_{i_t},\mu_{p_s}P)=n_{p_t}^{-}(\tau_{i_t},P)-c$, where $c=1$ or $2$.
Therefore, $$
\begin{array}{rcl} && \Omega(p_s,P)+\Omega(p_t,\mu_{p_s}P) \vspace{2pt}  \\

& = & [n^{+}_{p_s}(\tau_{i_s},P)-m^{+}_{p_s}(\tau_{i_s},\gamma)-n^{-}_{p_s}(\tau_{i_s},P)+m^{-}_{p_s}(\tau_{i_s},\gamma)]d^T \vspace{2pt}  \\

& - & [n^{+}_{p_t}(\tau_{i_t},\mu_{p_s}P)-m^{+}_{p_t}(\tau_{i_t},\gamma)-n^{-}_{p_t}(\tau_{i_t},\mu_{p_s}P)+m^{-}_{p_t}(\tau_{i_t},\gamma)]d^T \vspace{2pt}  \\

& = & [n^{+}_{p_s}(\tau_{i_s},\mu_{p_t}P)+c-m^{+}_{p_s}(\tau_{i_s},\gamma)-n^{-}_{p_s}(\tau_{i_s},\mu_{p_t}P)+m^{-}_{p_s}(\tau_{i_s},\gamma)]d^T \vspace{2pt}  \\

& - & [n^{+}_{p_t}(\tau_{i_t},P)-m^{+}_{p_t}(\tau_{i_t},\gamma)-n^{-}_{p_t}(\tau_{i_t},P)+c+m^{-}_{p_t}(\tau_{i_t},\gamma)]d^T \vspace{2pt}  \\

& = & \Omega(p_s,\mu_{p_t}P)+\Omega(p_t,P).\vspace{2pt}\\
\end{array}$$
\end{proof}

\medskip

\begin{Lemma}\label{reduction}

Let $P$ be a perfect matching of $G_{T,\gamma}$. Suppose that $(p_{i_1},\cdots,p_{i_r})$ is a sequence such that $\mu_{p_{i_{t-1}}}\cdots\mu_{p_{i_1}}$ can twist on $G(p_{i_t})$ for $1\leq t\leq r$ and $\mu_{p_{i_r}}\cdots\mu_{p_{i_1}}P=P$. Then

\begin{enumerate}[$(1)$]

  \item there exists $2\leq t\leq r$ such that $p_{i_t}=p_{i_1}$.

  \item There exist $t<t'$ such that $p_{i_{t'}}=p_{i_{t}}$ and $p_{i_{s'}}\neq p_{i_s}$ for any $s'\neq s$ with $t<s,s'<t'$. In this case, $|i_s-i_t|> 1$ for any $s$ satisfies $t<s<t'$.

\end{enumerate}

\end{Lemma}

\begin{proof}

(1)~ Denote the edges of $G(p_{i_1})$ by $b_1,b_2,b_3$ and $b_4$ as in the figure below. Since $P$ can twist on $G(p_{i_1})$, without loss of generality, assume that $b_1,b_3\in P$. In case $i_1=1$ or $i_1=d$,  $\mu_{p_{i_1}}P$ has to twist on $G(p_{i_1})$ to obtain $b_1$ or $b_3$. In case $G(p_{i_1-1})$ is on the left of $G(p_{i_1})$ and $G(p_{i_1+1})$ is on the right of $G(p_{i_1})$, $\mu_{p_{i_1}}P$ has to twist on $G(p_{i_1})$ to delete $b_2,b_4$. In case $G(p_{i_1-1})$ is on the left of $G(p_{i_1})$ and $G(p_{i_1+1})$ on top of $G(p_{i_1})$, $\mu_{p_{i_1}}P$ has to twist on $G(p_{i_1})$ to delete $b_4$. In case $G(p_{i_1-1})$ is under $G(p_{i_1})$ and $G(p_{i_1+1})$ is on the right of $G(p_{i_1})$, $\mu_{p_{i_1}}P$ has to twist on $G(p_{i_1})$ to delete $b_2$. In case $G(p_{i_1-1})$ is under $G(p_{i_1})$ and $G(p_{i_1+1})$ on top of $G(p_{i_1})$, $\mu_{p_{i_1}}P$ has to twist on $G(p_{i_1})$ to obtain $b_1,b_3$. Therefore in both cases there exists $2\leq t\leq r$ such that $p_{i_t}=p_{i_1}$.

\centerline{\begin{tikzpicture}
\draw[-] (1,0) -- (2,0);
\draw[-] (1,-1) -- (2,-1);
\node[above] at (1.5,-0.05) {$b_2$};
\node[left] at (1,-0.5) {$b_1$};
\node[below] at (1.5,-0.95) {$b_4$};
\node[right] at (2,-0.5) {$b_3$};
\node[right] at (1.2,-0.5) {$p_{i_1}$};
\draw [-] (1, 0) -- (1,-1);
\draw [-] (2, 0) -- (2,-1);
\draw [fill] (1,0) circle [radius=.05];
\draw [fill] (2,0) circle [radius=.05];
\draw [fill] (1,-1) circle [radius=.05];
\draw [fill] (2,-1) circle [radius=.05];
\draw[-] (4,0) -- (5,0);
\draw[-] (4,-1) -- (5,-1);
\node[above] at (4.5,-0.05) {$b_2$};
\node[left] at (4,-0.5) {$b_1$};
\node[below] at (4.5,-0.95) {$b_4$};
\node[right] at (5,-0.5) {$b_3$};
\node[right] at (4.2,-0.5) {$p_{i_t}$};
\draw [-] (4, 0) -- (4,-1);
\draw [-] (5, 0) -- (5,-1);
\draw [fill] (4,0) circle [radius=.05];
\draw [fill] (5,0) circle [radius=.05];
\draw [fill] (4,-1) circle [radius=.05];
\draw [fill] (5,-1) circle [radius=.05];
\end{tikzpicture}}

(2)~ The existence of $(t,t')$ follows immediately by (1). We may assume $i_s\neq i_t$ for all $s$ satisfies $t<s<t'$. Suppose that $|i_s-i_t|=1$ for some $s$ with $t<s<t'$. Choose the minimal $s$ satisfies such conditions. By the minimality of $s$ and Lemma \ref{commuta}, $\mu_{p_{i_{s-1}}}\cdots \mu_{p_{i_{t+1}}}\mu_{p_{i_{t}}}\cdots \mu_{p_{i_1}}P=\mu_{p_{i_{t}}}\mu_{p_{i_{s-1}}}\cdots \mu_{p_{i_{t+1}}}\mu_{p_{i_{t-1}}}\cdots \mu_{p_{i_1}}P$, denote this perfect matching by $Q$. Denote the edges of $G(p_{i_t})$ by $b_1,b_2,b_3$ and $b_4$ as in the figure below. Since $Q$ can twist on $G_{i_t}$, we may assume $b_1,b_3\in Q$. Since $Q$ can twist on $G(p_{i_s})$ and $|i_s-i_t|=1$, $G(i_s)$ is on the left or the right of $G(i_t)$. In case $G(i_s)$ is on the right of $G(i_t)$, then $b_1,b_3, b_5\in Q$, and thus $b_1,b_6,b_7\in \mu_{p_{i_s}}Q$. Since $\mu_{p_{i_{t'-1}}}\cdots \mu_{p_{i_{s+1}}}\mu_{p_{i_s}}Q$ can twist on $G(p_{i_{t'}})=G(p_{i_t})$, there exists $s'$ with $s+1\leq s'<t'$ such that $p_{i_{s'}}=p_{i_s}$, a contradiction. Similarly, one can prove such $s$ does not exist in case $G(i_s)$ is on the left of $G(i_t)$. Therefore, our result follows.
\end{proof}

\centerline{\begin{tikzpicture}
\draw[-] (1,0) -- (2,0);
\draw[-] (1,0) -- (1,-1);
\draw[-] (2,0) -- (2,-1);
\draw[-] (1,-1) -- (2,-1);
\draw[-] (3,0) -- (2,0);
\draw[-] (3,-1) -- (2,-1);
\draw[-] (3,0) -- (3,-1);
\node[above] at (1.5,-0.05) {$b_2$};
\node[left] at (1,-0.5) {$b_1$};
\node[below] at (2.5,-0.95) {$b_7$};
\node[below] at (1.5,-0.95) {$b_4$};
\node[above] at (2.5,-0.05) {$b_6$};
\node[right] at (2.2,-0.5) {$p_{i_s}$};
\node[right] at (1.75,-0.5) {$b_3$};
\node[right] at (3,-0.5) {$b_5$};
\node[right] at (1.2,-0.5) {$p_{i_t}$};
\draw [fill] (1,0) circle [radius=.05];
\draw [fill] (2,0) circle [radius=.05];
\draw [fill] (1,-1) circle [radius=.05];
\draw [fill] (2,-1) circle [radius=.05];
\draw [fill] (3,0) circle [radius=.05];
\draw [fill] (3,-1) circle [radius=.05];
\end{tikzpicture}}

\medskip

\begin{Theorem}\label{valumap}

Let $(\mathcal O,M)$ be an unpunctured surface and $\gamma$ be an oriented arc in $\mathcal O$. Suppose that $T$ is an indexed triangulation of $\mathcal O$.

\begin{enumerate}[$(1)$]

\item There uniquely exists a \emph{maximal valuation map} $v_{+}:\mathcal P(G_{T,\gamma})\rightarrow \mathbb Z$ such that

\begin{enumerate}[$(a)$]

  \item (initial condition) $v_{+}(P_{+}(G_{T,\gamma}))=0$.

  \item (iterated relation) If $P\in \mathcal P(G_{T,\gamma}))$ can twist on $G(p)$, then $$v_{+}(P)-v_{+}(\mu_pP)=\Omega(p,P).$$

\end{enumerate}

\item There uniquely exists a \emph{minimal valuation map} $v_{-}:\mathcal P(G_{T,\gamma})\rightarrow \mathbb Z$ such that

\begin{enumerate}[$(a)$]

  \item (initial condition) $v_{-}(P_{-}(G_{T,\gamma}))=0$.

  \item (iterated relation) If $P\in \mathcal P(G_{T,\gamma}))$ can twist on $G(p)$, then $$v_{-}(P)-v_{-}(\mu_pP)=\Omega(p,P).$$

\end{enumerate}

\end{enumerate}

\end{Theorem}

\begin{proof}

We shall only prove Statement (1). By Lemma \ref{transitive}, $v_{+}$ is unique if it exists. We have to show that $v_{+}$ is well-defined exist. That is, if $\mu_{p_{i_r}}\cdots\mu_{p_{i_1}}P=P$ for some sequence $(p_{i_1},\cdots,p_{i_r})$, then

$$\Omega(p_{i_1},p_{i_2},\cdots,p_{i_r};P):=\sum_{s=1}^{r}\Omega(p_{i_s},\mu_{p_{i_{s-1}}}\cdots \mu_{p_{i_1}}P)=0.$$

We prove $\Omega(p_{i_1},p_{i_2},\cdots,p_{i_r};P)=0$ by induction on $r$. When $r=0$, nothing need to be proved. It is clear that $r\neq 1$. When $r=2$, then $p_{i_1}=p_{i_2}$ and $\Omega(p_{i_1},p_{i_2};P)=0$ follows by definiton. Assume $\Omega(p_{i_1},p_{i_2},\cdots,p_{i_r};P)=0$ holds for $r<k$. In case $r=k$, according to Lemma \ref{reduction} (2), choose a pair $(t,t')$ with $1\leq t<t'\leq k$ such that $p_{i_{t'}}=p_{i_{t}}$ and $p_{i_{s'}}\neq p_{i_s}$ for any $s'\neq s$ with $t<s,s'<t'$. Thus the sequence of twists $\mu_{i_k}\cdots\mu_{i_{t'+1}}\mu_{i_{t'-1}}\cdots\mu_{i_{t+1}}\mu_{i_{t-1}}\cdots\mu_{i_1}$ is defined for $P$ and $\mu_{i_k}\cdots\mu_{i_{t'+1}}\mu_{i_{t'-1}}\cdots\mu_{i_{t+1}}\mu_{i_{t-1}}\cdots\mu_{i_1}P=P$. We claim that

$$\Omega(p_{i_1},p_{i_2},\cdots,p_{i_k};P)=\Omega(p_{i_1},\cdots,p_{i_{t-1}},p_{i_{t+1}},\cdots,p_{i_{t'-1}},p_{i_{t'+1}},\cdots, p_{i_k};P).$$

Once the claim is true, $\Omega(p_{i_1},\cdots,p_{i_{t-1}},p_{i_{t+1}},\cdots,p_{i_{t'-1}},p_{i_{t'+1}},\cdots, p_{i_k};P)=0$ by induction hypothesis, and thus $\Omega(p_{i_1},p_{i_2},\cdots,p_{i_k};P)=0$. Our result follows.

In the following, we prove our claim. Denote $Q=\mu_{i_{t-1}}\cdots\mu_{i_1}P$, to prove our claim, it suffices to prove that

$$\sum_{t\leq s\leq t'}\Omega(p_{i_s},\mu_{p_{i_{s-1}}}\cdots \mu_{p_{i_t}}Q)=\sum_{t< s< t'}\Omega(p_{i_s},\mu_{p_{i_{s-1}}}\cdots \mu_{p_{i_{t+1}}}Q).$$

By Lemma \ref{reduction} (2) and Lemma \ref{commuta} step by step, we have

$$
\begin{array}{rcl} && \sum_{t\leq s\leq t'}\Omega(p_{i_s},\mu_{p_{i_{s-1}}}\cdots \mu_{p_{i_t}}Q) \vspace{2pt}  \\

& = & \Omega(p_{i_t},Q)+\Omega(p_{i_{t+1}},\mu_{p_{i_t}}Q)+ \sum_{t+2\leq s\leq t'}\Omega(p_{i_s},\mu_{p_{i_{s-1}}}\cdots \mu_{p_{i_t}}Q) \vspace{2pt}  \\

& = & \Omega(p_{i_{t+1}},Q)+\Omega(p_{i_t},\mu_{p_{i_{t+1}}}Q)+ \sum_{t+2\leq s\leq t'}\Omega(p_{i_s},\mu_{p_{i_{s-1}}}\cdots \mu_{p_{i_t}}Q)\vspace{2pt} \\

& = & \Omega(p_{i_{t+1}},Q)+\Omega(p_{i_t},\mu_{p_{i_{t+1}}}Q)+\Omega(p_{i_{t+2}},\mu_{p_{i_{t+1}}}\mu_{p_{i_t}}Q)\vspace{2pt} \\

& & + \sum_{t+3\leq s\leq t'}\Omega(p_{i_s},\mu_{p_{i_{s-1}}}\cdots \mu_{p_{i_t}}Q)\vspace{2pt} \\

& = & \Omega(p_{i_{t+1}},Q)+\Omega(p_{i_t},\mu_{p_{i_{t+1}}}Q)+\Omega(p_{i_{t+2}},\mu_{p_{i_t}}\mu_{p_{i_{t+1}}}Q)\vspace{2pt} \\

& & + \sum_{t+3\leq s\leq t'}\Omega(p_{i_s},\mu_{p_{i_{s-1}}}\cdots \mu_{p_{i_t}}Q)\vspace{2pt} \\

& = & \Omega(p_{i_{t+1}},Q)+\Omega(p_{i_{t+2}},\mu_{p_{i_{t+1}}}Q)+\Omega(p_{i_t},\mu_{p_{i_{t+2}}}\mu_{p_{i_{t+1}}}Q)\vspace{2pt} \\

& & + \sum_{t+3\leq s\leq t'}\Omega(p_{i_s},\mu_{p_{i_{s-1}}}\cdots \mu_{p_{i_t}}Q)\vspace{2pt} \\

& = & \cdots\vspace{2pt} \\

& = & \sum_{t< s< t'}\Omega(p_{i_s},\mu_{p_{i_{s-1}}}\cdots \mu_{p_{i_{t+1}}}Q)\vspace{2pt} \\

& & + \Omega(p_{i_t},\mu_{p_{i_{t'-1}}}\cdots \mu_{p_{i_{t+1}}}Q)+\Omega(p_{i_{t'}},\mu_{p_{i_{t'-1}}}\cdots \mu_{p_{i_{t+1}}}\mu_{p_{i_t}}Q)\vspace{2pt} \\

& = & \sum_{t< s< t'}\Omega(p_{i_s},\mu_{p_{i_{s-1}}}\cdots \mu_{p_{i_{t+1}}}Q)\vspace{2pt} \\

& & + \Omega(p_{i_t},\mu_{p_{i_{t'-1}}}\cdots \mu_{p_{i_{t+1}}}Q)+\Omega(p_{i_t},\mu_{p_{i_t}}\mu_{p_{i_{t'-1}}}\cdots \mu_{p_{i_{t+1}}}Q)\vspace{2pt} \\

& = & \sum_{t< s< t'}\Omega(p_{i_s},\mu_{p_{i_{s-1}}}\cdots \mu_{p_{i_{t+1}}}Q).\vspace{2pt}\\
\end{array}$$

The last equality follows by definition. This verified our claim.
\end{proof}

\medskip

\subsection{$v_{+}=v_{-}$}

In this subsection, we prove $v_+=v_-$ and Theorem \ref{mainthm}. By definition, $X^{T'}_{\tau'}=(X^T)^{-e_{\tau}+(b^{T}_{\tau})_{+}}+(X^T)^{-e_{\tau}+(b^{T}_{\tau})_{-}}$. For convenient, denote $(X^T)^{-e_\tau+(b^{T}_{\tau})_{\pm}}$ by $\prod_{\pm}$. Let $d$ be a non-negative integer. For any sequence $\lambda=(\lambda_1,\cdots,\lambda_d)\in \{0,1\}^d$, denote by $n(\lambda)$ the integer such that
$$Z_1Z_2\cdots Z_d=q^{n(\lambda)/2}(X^T)^{-de_{\tau}+(\sum\lambda_i)(b^{T}_{\tau})_{-}+(d-\sum\lambda_i)(b^{T}_{\tau})_{+}},$$
where
  \[\begin{array}{ccl} Z_i &=&

         \left\{\begin{array}{ll}

             \prod_{-}, &\mbox{if $\lambda_i=1$}, \\

             \prod_{+}, &\mbox{if $\lambda_i=0$}.

         \end{array}\right.

 \end{array}\]
Clearly, $n(\lambda)=0$ if all $\lambda_i=1$ or all $\lambda_i=0$. Using the notation of $n(\lambda)$, we have
$$(X^{T'}_{\tau'})^d=\textstyle\sum_{\lambda}q^{n(\lambda)/2}(X^T)^{-de_{\tau}+(\sum\lambda_i)(b^{T}_{\tau})_{-}+(d-\sum\lambda_i)(b^{T}_{\tau})_{+}}.$$

\medskip

\begin{Lemma}\label{iterate}

Keep the foregoing notations. For any $i\in [1,d]$, if $\lambda$ and $\lambda'$ satisfies $\lambda_j=\lambda'_j$ for $j\neq i$ and $\lambda_i=1$, $\lambda'_i=0$, then $$n(\lambda)-n(\lambda')=(d-2i+1)d^T,$$
where $d^T$ is the integer determined by the compatibility of $(\widetilde B^T,\Lambda^T)$.

\end{Lemma}

\begin{proof}

Assume that $Z_1Z_2\cdots Z_{i-1}=q^{d_1/2}(X^T)^{\vec{a}}$ and $Z_{i+1}\cdots Z_{d}=q^{d_2/2}(X^T)^{\vec{b}}$ for some $d_1,d_2\in \mathbb Z$, where
  \[\begin{array}{ccl} Z_j &=&

         \left\{\begin{array}{ll}

             \prod_{-}, &\mbox{if $\lambda_j=1$}, \\

             \prod_{+}, &\mbox{if $\lambda_j=0$}.

         \end{array}\right.

 \end{array}\]
and we have $\vec{a}=-(i-1)e_{\tau}+(\sum_{j<i}\lambda_j)(b^{T}_{\tau})_{-}+(i-1-\sum_{j<i}\lambda_j)(b^{T}_{\tau})_{+},$ and
$\vec{b}=-(d-i)e_{\tau}+(\sum_{j>i}\lambda_j)(b^{T}_{\tau})_{-}+(d-i-\sum_{j>i}\lambda_j)(b^{T}_{\tau})_{+}$.

By the definitions of $n(\lambda)$ and $n(\lambda')$, we have
$$q^{d_1/2}(X^T)^{\vec{a}}\cdot \textstyle\prod_{-}\cdot q^{d_2/2}(X^T)^{\vec{b}}=q^{n(\lambda)/2}(X^T)^{-de_{\tau}+(\sum\lambda_i)(b^{T}_{\tau})_{-}+(d-\sum\lambda_i)(b^{T}_{\tau})_{+}},$$
$$q^{d_1/2}(X^T)^{\vec{a}}\cdot\textstyle\prod_{+}\cdot q^{d_2/2}(X^T)^{\vec{b}}=q^{n(\lambda')/2}(X^T)^{-de_{\tau}+(\sum\lambda'_i)(b^{T}_{\tau})_{-}+(d-\sum\lambda'_i)(b^{T}_{\tau})_{+}}.$$

Thus,
$$n(\lambda)=d_1+d_2+\Lambda^T(\vec{a}, \vec{b})+\Lambda^T(\vec{a}, (b^{T}_{\tau})_{-})+\Lambda^T((b^{T}_{\tau})_{-},\vec{b}),$$
and
$$n(\lambda')=d_1+d_2+\Lambda^T(\vec{a}, \vec{b})+\Lambda^{T}(\vec{a},(b^{T}_{\tau})_{+})+\Lambda^T((b^{T}_{\tau})_{+},\vec{b}).$$
Hence,
$$
\begin{array}{rcl} n(\lambda)-n(\lambda') & = & \Lambda^T(\vec{a}-\vec{b}, (b^{T}_{\tau})_{-}-(b^{T}_{\tau})_{+}) \vspace{2pt} \\

& = &  \Lambda^T(\vec{a}-\vec{b}, -b^{T}_{\tau})\vspace{2pt} \\

& = &  \Lambda^T(b^{T}_{\tau}, \vec{a}-\vec{b})\vspace{2pt} \\

& = &  [d-i-(i-1)]d^T = (d-2i+1)d^T.
\end{array}$$
here the last equality follows by the compatibility of $(\widetilde B^T,\Lambda^T)$.
\end{proof}

\medskip

\begin{Lemma}\label{minmax}

Keep the foregoing notations. Let $P=P_{\pm}(G_{T,\gamma})\in \mathcal P^{\tau}_{\nu}(G_{T,\gamma})$ for some $\nu$. If $\nu_i<0$ for some $i$, then $\nu_j\leq 0$ for any $j$.

\end{Lemma}

\begin{proof}

Without loss of generality, we may assume $P=P_{+}(G_{T,\gamma})$. Suppose that $\nu_j>0$ for some $j$. Then the $j$-th $\tau$-equivalence class is of type (I,IV). Since $\nu_i<0$ and the edges in $P$ are boundary edges, the $i$-th $\tau$-equivalence class is of type (I,III). Denote the endpoints of the two triangles containing $\tau$ by $o_1,o_2,o_3,o_4$. As shown in the following graph.

 \centerline{\begin{tikzpicture}
\draw[-] (1,0) -- (2,0);
\draw[-] (1,-1) -- (2,-1);
\draw[-] (1,0) -- (2,-1);
\node[right] at (1.4,-0.45) {$\tau$};
\node[above] at (1.5,-0.05) {$a_2$};
\node[left] at (1,-0.5) {$a_1$};
\node[left] at (1,0) {$o_1$};
\node[right] at (2,0) {$o_2$};
\node[left] at (1,-1) {$o_3$};
\node[right] at (2,-1) {$o_4$};
\node[below] at (1.5,-0.95) {$a_4$};
\node[right] at (2,-0.5) {$a_3$};
\draw [-] (1, 0) -- (1,-1);
\draw [-] (2, 0) -- (2,-1);
\draw [fill] (1,0) circle [radius=.05];
\draw [fill] (2,0) circle [radius=.05];
\draw [fill] (1,-1) circle [radius=.05];
\draw [fill] (2,-1) circle [radius=.05];
\end{tikzpicture}}

First, suppose that the $j$-th $\tau$-equivalence class is of type (I). Then $\gamma$ crosses $a_2,\tau,a_4$ consequently by Lemma \ref{max-min}.
If the $i$-th $\tau$-equivalence class is of type (I), then $\gamma$ crosses $a_1,\tau,a_3$ consequently. If $i$-th $\tau$-equivalence class is of type (III), then $\gamma$ starts from $o_3$ and crosses $\tau,a_3$ or starts from $o_2$ and $\tau,a_1$ consequently. In both cases, $\gamma$ crosses itself.

Now, suppose that the $j$-th $\tau$-equivalence class is of type (IV). Then the $j$-th $\tau$ equivalence class contains an edge of a tile $G$ with diagonal labeled $a_2$ or $a_4$ by Lemma \ref{max-min}. We may assume the diagonal is labeled $a_2$. Moreover, since the edge of $G$ labeled $\tau$ is a boundary edge, $\gamma$ starts from $o_4$ and crosses $a_2$. If the $i$-th $\tau$-equivalence classes is of type (I), then $\gamma$ crosses $a_1,\tau,a_3$ consequently. If $i$-th $\tau$-equivalence class is of type (III), then $\gamma$ starts from $o_3$ and crosses $\tau,a_3$ or starts from $o_2$ and $\tau,a_1$ consequently. In both cases, $\gamma$ crosses itself.

Our result follows.
\end{proof}

\medskip

\begin{Lemma}\label{maxtomax1}

Keep the foregoing notations. If $\tau'\neq \gamma\notin T$, then

\begin{enumerate}[$(1)$]

  \item $P_{\pm}(G_{T',\gamma})\subset \pi(P_{\pm}(G_{T,\gamma}))$.

  \item $P_{\pm}(G_{T,\gamma})\subset \pi'(P_{\pm}(G_{T',\gamma}))$

\end{enumerate}

\end{Lemma}

\begin{proof}

We shall only prove (1) and only consider the case $P_{+}(G_{T',\gamma})$. Write $P_{+}(G_{T,\gamma})$ as $(P_i)$, hence $P_i$ contains only boundary edges for $i\in [1,k]$. By Lemma \ref{max-min}, $P_i=P_{+}(G_{T,\gamma_i})$ for $i\in [1,k]$. By Lemma \ref{maxtomax}, we have $P_{+}(G_{T',\gamma_i})\in \psi_{\gamma_i}(P_i)$. Denote $P'_i=P_{+}(G_{T',\gamma_i})$ for $i\in [1,k]$. By Lemma \ref{formglobal}, $(P'_i)$ contains only boundary edges. Thus, $(P'_i)=P_{+}(G_{T,\gamma})$ by Lemma \ref{max-min}. Assume $P_{+}(G_{T,\gamma})\in \mathcal P^{\tau}_{\nu}(G_{T,\gamma})$ for some $\nu$. By Lemma \ref{minmax}, the set of $\nu$-pairs is empty. Thus, $(P'_i)=P_{+}(G_{T',\gamma})\in \pi(P_{+}(G_{T,\gamma}))$.
\end{proof}

\medskip

According to Lemma \ref{maxtomax1}, $P_{\pm}(G_{T',\gamma})\in \pi(P_{\pm}(G_{T,\gamma}))$. Assume $S_{\pm}\in \mathfrak P$ such that $P_{\pm}(G_{T,\gamma})\in S_{\pm}$ and $S'_{\pm}\in \mathfrak P'$ such that $P_{\pm}(G_{T',\gamma})\in S'_{\pm}$. Thus $\pi'(S'_{+})=S_{+}$ and $\pi'(S'_{-})=S_{-}$.

\medskip

Given $P'\in \mathcal P^{\tau'}_{\nu}(G_{T',\gamma})$ with $d=\sum \nu_l\geq 0$, let $P(\lambda)\in \pi'(P')$ corresponding to $\lambda\in \{0,1\}^d$ as Proposition \ref{2^n}. Assume $X^T(P(\lambda))=(X^T)^{\vec{a}(\lambda)}$ for some $\vec{a}(\lambda)\in \mathbb Z^m$. Particularly, the coordinate of $e_\tau$ of $\vec{a}(\lambda)$ is $-d$. For $\lambda,\lambda'\in \{0,1\}^d$ such that $\lambda_j=\lambda'_j$ for $j\neq i$ and $\lambda_i=1,\lambda'_i=0$, by Lemma \ref{brick1} and the constructions of $X^T(P(\lambda))$ and $X^T(P(\lambda'))$, we have $\vec{a}(\lambda')-\vec{a}(\lambda)=b^T_{\tau}$. Thus, if we assume $\vec{a}(1,1,\cdots,1)=\vec{a}_0-de_\tau+d(b^T_\tau)_{-}$ for some $\vec{a}_0\in \mathbb Z^m$, then $\vec{a}(\lambda)=\vec{a}_0-de_\tau+\sum\lambda_i(b^T_\tau)_{-}+(d-\sum\lambda_i)(b^T_{\tau})_{+}$. By Theorem \ref{partition bi} (2), we have $X^{T'}(P')=(X^{T'})^{\vec{a}_0+de_{\tau'}}$. In particular, the coordinates of $e_\tau$ and $e_{\tau'}$ of $\vec{a}_0$ are 0.

\medskip

Denote by $v'_{\pm}$ the maximal and minimal valuation maps on $\mathcal P(G_{T',\gamma})$, which exist by Theorem \ref{valumap}.

\medskip

\begin{Lemma}\label{induction}

Keep the foregoing notations. Let $P'\in \mathcal P^{\tau'}_{\nu}(G_{T',\gamma})$ with $\sum \nu_l\geq 0$.

\begin{enumerate}[$(1)$]

  \item The following are equivalent.

  \begin{enumerate}[$(a)$]

  \item $q^{v'_{+}(P')/2}X^{T'}(P')=\sum_{P\in \pi'(P')}q^{v_{+}(P)/2}X^{T}(P)$;

  \item $v'_{+}(P')=v_{+}(P(0,0,\cdots,0))$;

  \item $v'_{+}(P')=v_{+}(P(1,1,\cdots,1))$.

  \end{enumerate}

  \item The following are equivalent.

  \begin{enumerate}[$(a)$]

  \item $q^{v'_{-}(P')/2}X^{T'}(P')=\sum_{P\in \pi'(P')}q^{v_{-}(P)/2}X^{T}(P)$;

  \item $v'_{-}(P')=v_{-}(P(0,0,\cdots,0))$;

  \item $v'_{-}(P')=v_{-}(P(1,1,\cdots,1))$.

\end{enumerate}

\end{enumerate}

\end{Lemma}

\begin{proof}

We shall only prove (1) since (2) can be proved similarly. Denote $\sum\nu_l$ by $d$. Since $X^{T'}(P')=(X^{T'})^{\vec{a}_0+de_{\tau'}}$,
we have
$$X^{T'}(P')=q^{-\Lambda^{T'}(\vec{a}_0,de_{\tau'})/2}(X^{T'})^{\vec{a}_0}\cdot(X^{T'})^{de_{\tau'}}.$$
Since the coordinates of $e_{\tau}$ and $e_{\tau'}$ in $\vec{a}_0$ are 0, $(X^{T'})^{\vec{a}_0}=(X^{T})^{\vec{a}_0}$. Hence
$$
\begin{array}{rcl} X^{T'}(P')

& = & q^{-\Lambda^{T'}(\vec{a}_0,de_{\tau'})/2}(X^{T'})^{\vec{a}_0}\cdot\textstyle\sum_{\lambda}
q^{n(\lambda)/2}(X^{T})^{-de_{\tau}+\sum\lambda_i(b^T_{\tau})_-+(d-\sum\lambda_i)(b^T_\tau)_+} \vspace{2pt} \\

& = & q^{-\Lambda^{T'}(\vec{a}_0,de_{\tau'})/2} \textstyle\sum_{\lambda}
q^{n(\lambda)/2}(X^{T})^{\vec{a}_0}\cdot (X^{T})^{-de_{\tau}+\sum\lambda_i(b^T_{\tau})_-+(d-\sum\lambda_i)(b^T_\tau)_+}.\\
\end{array}$$

Moreover, by mutation of $\Lambda^T$, we have
$$\Lambda^T(\vec{a}_0, -de_{\tau}+\textstyle\sum\lambda_i(b^T_{\tau})_-+(d-\sum\lambda_i)(b^T_\tau)_{+})=\Lambda^{T'}(\vec{a}_0,de_{\tau'})$$
Therefore,
$$X^{T'}(P')=\textstyle\sum_{\lambda}q^{n(\lambda)/2}(X^T)^{\vec{a}_0-de_{\tau}+\sum\lambda_i(b^T_{\tau})_-+(d-\sum\lambda_i)(b^T_\tau)_{+}}.$$

On the other hand, for any $\lambda\in \{0,1\}^d$,
$$X^T(P(\lambda))=(X^T)^{\vec{a}_0-de_{\tau}+\sum\lambda_i(b^T_{\tau})_-+(d-\sum\lambda_i)(b^T_\tau)_{+}}.$$

Therefore, (a) holds if and only if $v'_{+}(P')+n(\lambda)=v_{+}(P(\lambda))$ for all $\lambda\in \{0,1\}^{d}$.

For any $i\in [1,d]$, if $\lambda$ and $\lambda'$ satisfies $\lambda_j=\lambda_j$ for $j\neq i$ and $\lambda_i=1$, $\lambda'_i=0$, assume $i$ corresponds to the tile $G(p_{l_i})$, by the construction of $\nu$-pairs, we have $n_{p_{l_i}}^{+}(\tau,P(\lambda'))-m_{p_{l_i}}^{+}(\tau,P(\lambda'))=-(d-i)$ and $n_{p_{l_i}}^{-}(\tau,P(\lambda'))-m_{p_{l_i}}^{-}(\tau,P(\lambda'))=-(i-1)$.
Therefore,
$$v_{+}(P(\lambda))-v_{+}(P(\lambda'))=-\Omega(p_{l_i},P(\lambda'))=(d-2i+1)d^T.$$
Moreover, by Lemma \ref{iterate} and $n(0,\cdots,0)=n(1,\cdots,1)=0$, our result follows.
\end{proof}

\medskip

We have the following results as applications of Lemma \ref{induction}.

\medskip

\begin{Proposition}\label{initial}

Keep the foregoing notations. Assume $\tau'\neq\gamma\notin T$, then

\begin{enumerate}[$(1)$]

  \item $\sum_{P'\in S'_{+}}q^{v'_{+}(P')/2}X^{T'}(P')=\sum_{P\in S_{+}}q^{v_{+}(P)/2}X^T(P)$.

  \item $\sum_{P'\in S'_{-}}q^{v'_{-}(P')/2}X^{T'}(P')=\sum_{P\in S_{-}}q^{v_{-}(P)/2}X^T(P)$.

\end{enumerate}

\end{Proposition}

\begin{proof}

We shall only prove (1) because (2) can be proved similarly. We may assume that $P_{+}(G_{T',\gamma})\in \mathcal P_{\nu}(G_{T',\gamma})$ with $\sum\nu_l\geq 0$. By Lemma \ref{maxtomax1} and the dual version of Proposition \ref{2^n}, there exists $\lambda$ such that $P_{+}(G_{T,\gamma})=P(\lambda)$. By Lemma \ref{max-min}, $\lambda=(0,0,\cdots,0)$ or $(1,1,\cdots,1)$. Since $v_{+}(P(\lambda))=v_{+}(P_{+}(G_{T,\gamma}))=0$, by Lemma \ref{induction}, our result follows.
\end{proof}

\medskip

\begin{Proposition}\label{non-a}

Keep the foregoing notations. Assume $\tau'\neq \gamma\notin T$. Let $Q$ be a perfect matching of $G_{T,\gamma}$ which can twist on a tile $G(p)$ with the diagonal is not labeled $a_1,a_2,a_3,a_4$. Assume $S_1,S_2\in \mathfrak P$ such that $Q\in S_1$ and $\mu_{p}Q\in S_2$. Then

\begin{enumerate}[$(1)$]

  \item $\sum_{P'\in \pi(S_1)}q^{v'_{+}(P')/2}X^{T'}(P')=\sum_{P\in S_1}q^{v_{+}(P)/2}X^{T}(P)$ holds if and only if\\
   $\sum_{P'\in \pi(S_2)}q^{v'_{+}(P')/2}X^{T'}(P')=\sum_{P\in S_2}q^{v_{+}(P)/2}X^{T}(P)$ holds.

  \item $\sum_{P'\in \pi(S_1)}q^{v'_{-}(P')/2}X^{T'}(P')=\sum_{P\in S_1}q^{v_{-}(P)/2}X^{T}(P)$ holds if and only if\\
   $\sum_{P'\in \pi(S_2)}q^{v'_{-}(P')/2}X^{T'}(P')=\sum_{P\in S_2}q^{v_{-}(P)/2}X^{T}(P)$ holds.

\end{enumerate}

\end{Proposition}

\begin{proof}

We shall only prove (1). Assume $Q\in \mathcal P^{\tau}_{\nu}(G_{T,\gamma})$ for some $\nu$, then $\mu_{p}Q\in \mathcal P^{\tau}_{\nu}(G_{T,\gamma})$. Denote the label of the diagonal of $G(p)$ by $a$.

We first suppose that $\sum\nu_l\geq 0$. Denote by $Q'$ and $Q''$ the perfect matchings corresponding to $(0,0,\cdots,0)\in \{0,1\}^{\sum \nu_l}$ in $\pi(Q)$ and $\pi(\mu_pQ)$, respectively.

In case $a\neq \tau$, $G(p)$ is a tile of $G_{T,\gamma_j}$ for some $\gamma_j$ which does not cross $a_1,a_2,a_3,a_4,\tau$. Write $Q$ as $(Q_i)$, then $\mu_{p}Q=(\mu_pQ_i)$, here $\mu_pQ_i=Q_i$ if $i\neq j$. Write $Q'$ as $(Q'_i)$ and $Q''$ as $(Q''_i)$. By the constructions of $Q'$ and $Q''$, $Q'_j=Q_j$, $Q''_j=\mu_pQ_j$ and $Q'_i=Q''_i$ for $i\neq j$. Thus, $Q''=\mu_pQ'$. By Proposition \ref{compare-loc} (2.c), the number of edges labeled $a$ in $Q_i$ is the same with that in $Q'_i$ for all $i\in [1,k]$. Further, by Proposition \ref{compare-loc} (2.d) and $Q_j=Q'_j$, we have $\Omega(p,Q)=\Omega(p, Q')$. Therefore,
$$v_{+}(Q)-v_{+}(\mu_{p}Q)=\Omega(p,Q)=\Omega(p,Q')=v'_{+}(Q')-v'_{+}(Q'').$$
Thus $v_{+}(Q)=v'_{+}(Q')$ holds if and only if $v_{+}(\mu_{p}Q)=v'_{+}(Q'')$ holds. Consequently, by the dual version of Lemma \ref{induction}, our result follows in this case.

In case $a=\tau$, assume $G(p)$ corresponds to the $t$-th $\tau$-equivalence class in $G_{T,\gamma}$, and hence $\nu_t=-1$. Since $\sum\nu_l\geq 0$, there exists $s\in [1,n^{\tau}(T,\gamma)]$ such that $\{s,t\}$ is a $\nu$-pair, and hence $\nu_s=1$. Since $Q'\in \mathcal P^{\tau'}_{-\nu}(G_{T',\gamma})$ and $-\nu_s=-1$, the $s$-th $\tau'$-equivalence class in $G_{T',\gamma}$ corresponds to the tile $G'(p_{l_s})$ with diagonal labeled $\tau'$. By the constructions of $Q'$ and $Q''$, we have $Q''=\mu_{p_{l_s}}Q'$. We may assume the edges of $G(p)$ labeled $a_1,a_3$ are in $Q$. Hence the edges labeled $a_1,a_3$ are in $Q'$. Since $a_1,a_3$ are counterclockwise/clockwise to $\tau'$/$\tau$, $\Omega(p_{l_s},Q')=d^T(\sum_{l> s}(-\nu_l)-\sum_{l<s}(-\nu_l))$ and $\Omega(p,Q)=-d^T(\sum_{l> t}\nu_l-\sum_{l<t}\nu_l)$. Since $\{s,t\}$ is a $\nu$-pair, $\sum_{l\in (s,t)}\nu_l=0$, where $(s,t)$ means the integers between $s$ and $t$. Thus, $\Omega(p,Q)=\Omega(p_{l_s},Q')$. Therefore,
$$v_{+}(Q)-v_{+}(\mu_{p}Q)=\Omega(p,Q)=\Omega(p_{l_s},Q')=v'_{+}(Q')-v'_{+}(Q'').$$
Thus $v_{+}(Q)=v'_{+}(Q')$ holds if and only if $v_{+}(\mu_{p}Q)=v'_{+}(Q'')$ holds. Consequently, by the dual version of Lemma \ref{induction}, our result follows in this case.

Now we suppose that $\sum\nu_l\leq 0$. Denote $\pi(Q)=\{Q'\}$ and $\pi(\mu_pQ)=\{Q''\}$.

In case $a\neq \tau$, by the same reason, we have $Q''=\mu_{p}Q'$ and $Q',Q''\in \mathcal P^{\tau'}_{-\nu}(G_{T',\gamma})$. Change the roles of $T$ and $T'$, as the discussion for the case $\sum\nu_l\geq 0$, our result follows in this case.

In case $a=\tau$, assume $G(p)$ corresponds to the $t$-th $\tau$-equivalence class in $G_{T,\gamma}$, and hence $\nu_t=-1$. If $t$ is not in a $\nu$-pair, then $Q'=Q''$, and hence $S_1=S_2$. Our result follows at once in this case. If $t$ is in a $\nu$-pair $\{s,t\}$, then $\nu_s=1$. Since $Q'\in \mathcal P^{\tau'}_{-\nu}(G_{T',\gamma})$ and $-\nu_s=-1$, the $s$-th $\tau'$-equivalence class in $G_{T',\gamma}$ corresponds to the tile $G'(p_{l_s})$ with diagonal labeled $\tau'$. By the same reason, we have $Q''=\mu_{p_{l_s}}Q'$. Change the roles of $T$ and $T'$, as the discussion for the case $\sum\nu_l\geq 0$, our result follows in this case.
\end{proof}

\medskip

\begin{Lemma}\label{special-a}

Keep the foregoing notations. Assume $\tau'\neq\gamma\notin T$. Let $Q$ be a perfect matching of $G_{T,\gamma}$ which can twist on a tile $G(p)$ with diagonal labeled $a=a_q$ for $q=1,2,3,4$. Assume $S_1,S_2\in \mathfrak P$ such that $Q\in S_1$ and $\mu_{p}Q\in S_2$. If all the $\tau$-mutable edges pairs in $Q$ are labeled $a_{q-1},a_{q+1}$ (addition in $\mathbb Z_4$), then

\begin{enumerate}[$(1)$]

  \item $\sum_{P'\in \pi(S_1)}q^{v'_{+}(P')/2}X^{T'}(P')=\sum_{P\in S_1}q^{v_{+}(P)/2}X^{T}(P)$ holds if and only if\\
   $\sum_{P'\in \pi(S_2)}q^{v'_{+}(P')/2}X^{T'}(P')=\sum_{P\in S_2}q^{v_{+}(P)/2}X^{T}(P)$ holds.

  \item $\sum_{P'\in \pi(S_1)}q^{v'_{-}(P')/2}X^{T'}(P')=\sum_{P\in S_1}q^{v_{-}(P)/2}X^{T}(P)$ holds if and only if\\
   $\sum_{P'\in \pi(S_2)}q^{v'_{-}(P')/2}X^{T'}(P')=\sum_{P\in S_2}q^{v_{-}(P)/2}X^{T}(P)$ holds.

\end{enumerate}

\end{Lemma}

\begin{proof}

We shall only prove (1). We may assume the edge labeled $\tau$ of $G(p)$ is in $Q$. Write $Q$ as $(Q_i)$ and assume $G(p)$ is a tile of $G_{T,\gamma_j}$ for some $j$. Thus $\mu_pQ=(\mu_pQ_i)$ with $\mu_pQ_i=Q_i$ for $i\neq j$. By Lemma \ref{com}, there exists $Q'_j\in \psi_{\gamma_j}(Q_j)$ satisfies (a) $Q'_j$ can twist on $G'(p)$, (b) $\mu_pQ'_j\in \psi_{\gamma_j}(\mu_pQ_j)$, (c) all $\tau'$-mutable edges pairs in $Q'_j$ are labeled $a_{q-1},a_{q+1}$. For each $i\neq j$, choose $Q_i'\in \psi_{\gamma_i}(Q_i)$ such that all $\tau'$-mutation edges pairs in $Q'_i$ are labeled $a_{q-1},a_{q+1}$. Since all $\tau$-mutable edges pairs in $Q$ are labeled $a_{q-1},a_{q+1}$, all $\tau$-mutable edges pairs in $\mu_pQ$ are labeled $a_{q-1},a_{q+1}$. By the constructions of $\pi(Q)$ and $\pi(\mu_pQ)$, we have $Q'=(Q'_i)\in \pi(Q)$ and $Q''=(\mu_pQ'_i)\in \pi(\mu_pQ)$, here we mean $\mu_pQ'_i=Q'_i$ for $i\neq j$. Clearly, $\mu_{p}Q'=Q''$. Assume $\mu_pQ\in \mathcal P^{\tau}_{\nu'}(G_{T,\gamma})$. Clearly, $\sum\nu_l>\sum\nu'_l$.

Since $d^T=d^{T'}$, by Lemma \ref{com} (1.d) and Proposition \ref{compare-loc} (2.c), we have $\Omega(p,Q)=\Omega(p,Q')$. Therefore,
$$v_{+}(Q)-v_{+}(\mu_{p}Q)=\Omega(p,Q)=\Omega(p,Q')=v'_{+}(Q')-v'_{+}(\mu_pQ').$$
Thus $v_{+}(Q)=v'_{+}(Q')$ holds if and only if $v_{+}(\mu_{p}Q)=v'_{+}(\mu_pQ')$ holds.

Suppose that $\sum\nu_l\geq 0$ and $\sum\nu'_l\geq 0$. Then $S_1=\{Q\}$ and $S_2=\{\mu_{p}Q\}$. If $a=a_1,a_3$, since all mutation edges pairs in $Q'$ and $Q''$ are labeled $a_2,a_4$, $Q'$ and $Q''$ corresponds to $(0,0,\cdots,0)\in \{0,1\}^{\sum\nu_l}$ and $(0,0,\cdots,0)\in \{0,1\}^{\sum\nu'_l}$, respectively. By Lemma \ref{induction}, our result follows in this case. If $a=a_2,a_4$, since all mutation edges pairs in $Q'$ and $Q''$ are labeled $a_1,a_3$, $Q'$ and $Q''$ corresponds to $(1,1,\cdots,1)\in \{0,1\}^{\sum\nu_l}$ and $(1,1,\cdots,1)\in \{0,1\}^{\sum\nu'_l}$, respectively. By Lemma \ref{induction}, our result follows in this case.

Suppose that $\sum\nu_l\geq 0$ and $\sum\nu'_l< 0$. Then $S_1=\{Q\}$ and $\pi(S_2)=\{Q''\}$. If $a=a_1,a_3$, since all mutation edges pairs in $Q'$ and $\mu_pQ$ are labeled $a_2,a_4$, $Q'$ and $\mu_pQ$ corresponds to $(0,0,\cdots,0)\in \{0,1\}^{\sum\nu_l}$ and $(0,0,\cdots,0)\in \{0,1\}^{-\sum\nu'_l}$, respectively. By Lemma \ref{induction}, our result follows in this case. If $a=a_2,a_4$, since all mutation edges pairs in $Q'$ and $\mu_pQ$ are labeled $a_1,a_3$, $Q'$ and $\mu_pQ$ corresponds to $(1,1,\cdots,1)\in \{0,1\}^{\sum\nu_l}$ and $(1,1,\cdots,1)\in \{0,1\}^{-\sum\nu'_l}$, respectively. By Lemma \ref{induction}, our result follows in this case.

Suppose that $\sum\nu_l\geq 0$ and $\sum\nu'_l< 0$. Change the roles of $T$ and $T'$, as the discussion for the case $\sum\nu_l,\sum\nu'_l\geq 0$, our result follows in this case.
\end{proof}

\medskip

\begin{Proposition}\label{general-a}

Keep the foregoing notations. Assume $\tau'\neq\gamma\notin T$. Let $Q$ be a perfect matching of $G_{T,\gamma}$ which can twist on a tile $G(p)$ with diagonal labeled $a_q$ for $q=1,2,3,4$. Assume $S_1,S_2\in \mathfrak P$ such that $Q\in S_1$ and $\mu_{p}Q\in S_2$. Then

\begin{enumerate}[$(1)$]

  \item $\sum_{P'\in \pi(S_1)}q^{v'_{+}(P')/2}X^{T'}(P')=\sum_{P\in S_1}q^{v_{+}(P)/2}X^{T}(P)$ holds if and only if\\
   $\sum_{P'\in \pi(S_2)}q^{v'_{+}(P')/2}X^{T'}(P')=\sum_{P\in S_2}q^{v_{+}(P)/2}X^{T}(P)$ holds.

  \item $\sum_{P'\in \pi(S_1)}q^{v'_{-}(P')/2}X^{T'}(P')=\sum_{P\in S_1}q^{v_{-}(P)/2}X^{T}(P)$ holds if and only if\\
   $\sum_{P'\in \pi(S_2)}q^{v'_{-}(P')/2}X^{T'}(P')=\sum_{P\in S_2}q^{v_{-}(P)/2}X^{T}(P)$ holds.

\end{enumerate}

\end{Proposition}

\begin{proof}

We shall only prove (1) since (2) can be proved similarly. After do twists from $Q$ on the tiles $G$ with diagonal labeled $\tau$ and the edges labeled $a_q,a_{q+2}$ are in $Q$, we obtain a perfect matching $R$ which satisfies the conditions of Lemma \ref{special-a}. By Lemma \ref{commuta} (1), $\mu_pR$ can be obtained by the same steps of twists from $\mu_pQ$. Assume $S'_1, S'_2\in \mathfrak P$ such that $R\in S'_1$ and $\mu_pR\in S'_2$. Applying Proposition \ref{non-a} step by step to $Q$ and $\mu_pQ$, we have the following
\begin{enumerate}[$(a)$]

   \item  $\sum_{P'\in \pi(S_1)}q^{v'_{+}(P')/2}X^{T'}(P')=\sum_{P\in S_1}q^{v_{+}(P)/2}X^{T}(P)$ holds if and only if\\ $\sum_{P'\in \pi(S'_1)}q^{v'_{+}(P')/2}X^{T'}(P')=\sum_{P\in S'_1}q^{v_{+}(P)/2}X^{T}(P)$ holds.

   \item  $\sum_{P'\in \pi(S_2)}q^{v'_{+}(P')/2}X^{T'}(P')=\sum_{P\in S_2}q^{v_{+}(P)/2}X^{T}(P)$ holds if and only if\\ $\sum_{P'\in \pi(S'_2)}q^{v'_{+}(P')/2}X^{T'}(P')=\sum_{P\in S'_2}q^{v_{+}(P)/2}X^{T}(P)$ holds.

\end{enumerate}
On the other hand, by Lemma \ref{special-a}, $\sum_{P'\in \pi(S'_1)}q^{v'_{+}(P')/2}X^{T'}(P')=\sum_{P\in S'_1}q^{v_{+}(P)/2}X^{T}(P)$ holds if and only if $\sum_{P'\in \pi(S'_2)}q^{v'_{+}(P')/2}X^{T'}(P')=\sum_{P\in S'_2}q^{v_{+}(P)/2}X^{T}(P)$ holds.

Therefore,
%$\sum_{P'\in \pi(S_1)}q^{v'_{+}(P')/2}x(P')=\sum_{P\in S_1}q^{v_{\varepsilo}(P)/2}x(P)$ holds if and only if  $\sum_{P'\in \pi(S_2)}q^{v'_{+}(P')/2}x(P')=\sum_{P\in S_2}q^{v_{\varepsilon}(P)/2}x(P)$ holds.
Our result follows.
\end{proof}

\medskip

Summaries Propositions \ref{non-a} and \ref{general-a}, we obtain.

\medskip

\begin{Proposition}\label{all}

Keep the foregoing notations. Assume $\tau'\neq \gamma\notin T$. Let $Q$ be a perfect matching of $G_{T,\gamma}$ which can twist on a tile $G(p)$. Assume $S_1,S_2\in \mathfrak P$ such that $Q\in S_1$ and $\mu_{p}Q\in S_2$. Then

\begin{enumerate}[$(1)$]

  \item $\sum_{P'\in \pi(S_1)}q^{v'_{+}(P')/2}X^{T'}(P')=\sum_{P\in S_1}q^{v_{+}(P)/2}X^{T}(P)$ holds if and only if\\
   $\sum_{P'\in \pi(S_2)}q^{v'_{+}(P')/2}X^{T'}(P')=\sum_{P\in S_2}q^{v_{+}(P)/2}X^{T}(P)$ holds.

  \item $\sum_{P'\in \pi(S_1)}q^{v'_{-}(P')/2}X^{T'}(P')=\sum_{P\in S_1}q^{v_{-}(P)/2}X^{T}(P)$ holds if and only if\\
   $\sum_{P'\in \pi(S_2)}q^{v'_{-}(P')/2}X^{T'}(P')=\sum_{P\in S_2}q^{v_{-}(P)/2}X^{T}(P)$ holds.

\end{enumerate}

\end{Proposition}

\medskip

\begin{Theorem}\label{mainpre}

Keep the foregoing notations. For any $S\in \mathfrak P$,

\begin{enumerate}[$(1)$]

  \item $\sum_{P\in S}q^{v_{+}(P)/2}X^{T}(P)=\textstyle\sum_{P'\in \pi(S)}q^{v'_{+}(P')/2}X^{T'}(P')$.

  \item $\sum_{P\in S}q^{v_{-}(P)/2}X^{T}(P)=\textstyle\sum_{P'\in \pi(S)}q^{v'_{-}(P')/2}X^{T'}(P')$.

\end{enumerate}

\end{Theorem}

\begin{proof}

We shall only prove (1) because (2) can be proved similarly. First suppose that $\gamma\neq \tau,\tau'$. Choose $P\in S$, by Lemma \ref{transitive}, $P$ can be obtained from $P_{+}(G_{T,\gamma})$ by a sequence of twists. In this case, the equality follows from Proposition \ref{initial} and Proposition \ref{all}. Now suppose that $\gamma=\tau$ or $\tau'$. Then the equality becomes to
$$X^{T}_{\tau}=(X^{T'})^{-e_{\tau'}+(b^{T'}_{\tau'})_{+}}+(X^{T'})^{-e_{\tau'}+(b^{T'}_{\tau'})_{-}}\;\;\text{or}\;\; X^{T'}_{\tau'}=(X^{T})^{-e_{\tau}+(b^T_{\tau})_{+}}+(X^T)^{-e_{\tau}+(b^T_{\tau})_{-}},$$
which clearly holds. Our result follows.
\end{proof}

\medskip

Now we can give the proof of Theorem \ref{mainthm}.

\medskip

{\bf Proof of Theorem \ref{mainthm}:} By Theorem \ref{mainpre}, it suffices to show $v_{+}=v_{-}$. We prove this by induction on $N(\gamma,T)$, the number of crossing points of $\gamma$ with $T$. When $N(\gamma,T)=0$ or $1$, it is easy to see that $v_{+}=v_{-}=0$. Assume $v_{+}=v_{-}$ for all $N(\gamma,T)<d$. When $N(\gamma,T)=d>1$, there exists $\tau\in T$ such that the $N(\gamma,T')<d$. Denote by $v'_{\pm}$ the maximal and minimal valuation maps on $\mathcal P(G_{T',\gamma})$. By induction hypothesis, $v'_{+}=v'_{-}$. By Lemma \ref{maxtomax1}, $P_{+}(G_{T,\gamma})\in \pi'(P_{+}(G_{T',\gamma}))$. Assume $P_{+}(G_{T,\gamma})\in \mathcal P^{\tau}_{\nu}(G_{T,\gamma})$ for some $\nu$.

If $\sum\nu_l\leq 0$, by the dual version of Proposition \ref{2^n} and Lemma \ref{max-min}, $P_{+}(G_{T,\gamma})=P(\lambda)$ with $\lambda=(0,0,\cdots,0)$ or $(1,1,\cdots,1)\in \{0,1\}^{-\sum\nu_l}$, here $P(\lambda)$ means the perfect matching in $\pi'(P_{+}(G_{T',\gamma}))$ which is determined by $\lambda$. By Theorem \ref{mainpre} and Lemma \ref{induction} (2), we have $v_{-}(P_{+}(G_{T,\gamma}))=v'_{-}(P_{+}(G_{T',\gamma})).$ Since $v'_{+}=v'_{-}$, $v'_{-}(P_{+}(G_{T',\gamma}))=0$, and hence $v_{-}(P_{+}(G_{T,\gamma}))=0$. Consequently, $v_{+}=v_{-}$.

If $\sum\nu_l\geq 0$, by Proposition \ref{2^n} and Lemma \ref{max-min}, $P_{+}(G_{T',\gamma})=P'(\lambda)$ with $\lambda=(0,0,\cdots,0)$ or $(1,1,\cdots,1)\in \{0,1\}^{\sum\nu_l}$, here $P'(\lambda)$ means the perfect matching in $\pi(P_{+}(G_{T,\gamma}))$ which is determined by $\lambda$. By Theorem \ref{mainpre} and the dual version of Lemma \ref{induction} (2), we have $v_{-}(P_{+}(G_{T,\gamma}))=v'_{-}(P_{+}(G_{T',\gamma})).$ Since $v'_{+}=v'_{-}$, $v'_{-}(P_{+}(G_{T',\gamma}))=0$, and hence $v_{-}(P_{+}(G_{T,\gamma}))=0$. Consequently, $v_{+}=v_{-}$.

The proof of Theorem \ref{mainthm} is complete.

\medskip

\section{Remarks on the relation to the related works}

In \cite{S,MS,MSW}, by Fomin-Zelevinsky' separation formula \cite{fz4}, it suffices to give the Laurent cluster expansion formula for principal coefficients cluster algebras. However, there is no quantum version of Fomin-Zelevinsky' separation formula, we have to deal with the arbitrary coefficients but not just the principal coefficients. Our proof is different to \cite{S,MS,MSW}, in their proof, they fix the triangle $T$ and change the arc $\gamma$ in the induction step; in our proof, we fix the arc $\gamma$ and change the triangle $T$ in the induction step, this helps us to compare the Laurent expansion of the same cluster variable with respect to different clusters term by term, further, we can deal with the quantum case.

\medskip

\cite{CL} gives the explicit Laurent expansion formulas for quantum cluster algebras of type $A$ and Kronecker type with principal coefficients and principal quantization. We work on the quantum cluster algebras from unpunctured surfaces with arbitrary coefficients and quantization, however, we can not give the explicit formula of $q$-coefficients $v(P)$.

\medskip

We hope our method can be generalized to the quantum cluster algebras from surface with punctures.

\medskip

{\bf Acknowledgements:}\;  {\em The author is thankful to S. Liu, I. Assem, T. Br$\ddot{u}$stle and D. Smith for financial support. Part of the result is posted at the conference Cluster Algebras and Math Physics (East Lansing, May 7-May 12), the author
thanks the organizers for their hospitality. Thanks Ilke Canakci for pointing their
results on quantum Laurent expansion. The author is grateful to Peigen Cao for checking the proof. Last, the author would take the opportunity
to thanks his Ph.D. advisor Prof. Fang Li for continued encouragement these years.}

\end{document}